\documentclass[oneside,11pt,reqno]{amsart}
\usepackage{amsmath,amssymb,amsthm,xypic} 
\input xy
\xyoption{all}

\usepackage{mathrsfs}

\usepackage{hyperref}
\usepackage[shortcuts]{extdash}

\usepackage[totalheight=23 true cm, totalwidth=16 true cm]{geometry}
\usepackage{color}

\title{\'{E}tale difference algebraic groups}

\author{Michael Wibmer}
\address{Michael Wibmer, Institute of Analysis and Number Therory, Graz University of Technology, Kopernikusgasse~24, 8010 Graz, Austria, \url{https://sites.google.com/view/wibmer}}
\email{wibmer@math.tugraz.at}

\thanks{This work was supported by the NSF grants DMS-1760212, DMS-1760413, DMS-1760448 and the Lise Meitner grant M 2582-N32 of the Austrian Science Fund FWF}

\subjclass[2020]{14L15, 12H10, 37B05}
\keywords{Difference algebraic group, \'{e}tale algebraic group, expansive endomorphism, profinite group}

\date{\today}

\newtheorem{theo}{Theorem}[section]
\newtheorem{lemma}[theo]{Lemma}
\newtheorem{prop}[theo]{Proposition}
\newtheorem{cor}[theo]{Corollary}
\newtheorem{defi}[theo]{Definition}
\newtheorem{rem}[theo]{Remark}

\newtheorem*{theononumber}{Theorem}

\theoremstyle{definition}
\newtheorem{ex}[theo]{Example}

\newcommand{\f}{\phi}

\newcommand{\ida}{\mathfrak{a}}

\newcommand{\p}{\mathfrak{p}}

\newcommand{\m}{\mathfrak{m}}

\newcommand{\spec}{\operatorname{Spec}}
\newcommand{\Z}{\mathbb{Z}}

\newcommand{\N}{\mathcal{N}}
\newcommand{\G}{\mathcal{G}}
\newcommand{\HH}{\texttt{H}}
\newcommand{\gal}{\operatorname{Gal}}
\newcommand{\Gl}{\operatorname{GL}}

\newcommand{\Hom}{\operatorname{Hom}}

\newcommand{\V}{\mathbb{V}}

\newcommand{\id}{\operatorname{id}}

\newcommand{\s}{\sigma}
\newcommand{\de}{\delta}

\newcommand{\sdim}{\sigma\text{-}\dim}

\newcommand{\ld}{\operatorname{ld}}

\newcommand{\ks}{$k$\=/$\s$}
\newcommand{\ord}{\operatorname{ord}}

\newcommand{\hs}{^\sigma\!}
\newcommand{\hsr}{^{\sigma^r}\!}
\newcommand{\hsi}{^{\sigma^i}\!}
\newcommand{\nn}{\mathbb{N}}

\def\H{\mathcal{H}}
\renewcommand{\sc}{{\sigma o}}
\newcommand{\Gm}{\mathbb{G}_m}
\newcommand{\ssetale}{{strongly $\sigma$-\'{e}tale}}

\newcommand{\I}{\mathbb{I}}
\newcommand{\GG}{\mathtt{G}}
\newcommand{\NN}{\mathtt{N}}
\newcommand{\Gal}{\mathscr{G}}

\newcommand{\Ga}{\mathbb{G}_a}

\newcommand{\pis}{\pi_0^\sigma}

\newcommand{\C}{\mathbb{C}}
\newcommand{\XX}{\mathtt{X}}
\newcommand{\X}{\mathcal{X}}
\newcommand{\VV}{\mathcal{V}}

\begin{document}

\maketitle


\begin{abstract}
	\'{E}tale difference algebraic groups are a difference analog of \'{e}tale algebraic groups. Our main result is a Jordan-H\"{o}lder type decomposition theorem for these groups. Roughly speaking, it shows that any \'{e}tale difference algebraic group can be build up from simple \'{e}tale algebraic groups and two finite \'{e}tale difference algebraic groups. The simple \'{e}tale algebraic groups occurring in this decomposition satisfy a certain uniqueness property.
\end{abstract}

\section{Introduction}

Affine difference algebraic groups are a generalization of affine algebraic groups. Instead of just algebraic equations, one allows difference algebraic equations as the defining equations. Alternatively, affine difference algebraic groups can be described as affine group schemes with an additional structure (the difference structure). In algebraic terms, to specify an affine difference algebraic group $G$ over a difference field $k$ (i.e., $k$ is a field equipped with an endomorphism $\s\colon k\to k$) is equivalent to specifying a Hopf algebra $k\{G\}$ over $k$ together with a ring endomorphism $\s\colon k\{G\}\to k\{G\}$ extending $\s\colon k\to k$ such that the Hopf algebra structure maps commute with $\s$. It is also required that $G$ is ``of finite $\s$-type'', i.e., there exists a finite subset $B$ of $k\{G\}$ such that $B,\s(B),\s^2(B),\ldots$ generates $k\{G\}$ as a $k$-algebra.

\'{E}tale difference algebraic groups are a difference analog of \'{e}tale algebraic groups. Algebraically, they can be described as the affine difference algebraic groups $G$ such that every element of $k\{G\}$ satisfies a separable polynomial over $k$. For example, the difference algebraic equations $x^n=1$, $\s(x)x=1$ define an \'{e}tale difference algebraic subgroup of the multiplicative group $\Gm$, as long as $n$ is not divisible by the characteristic of $k$. By interpreting algebraic equations as difference algebraic equations, any \'{e}tale algebraic group $\G$ over $k$ defines an \'{e}tale difference algebraic group $[\s]_k\G$ over $k$.
The Hopf algebra corresponding to $[\s]_k\G$ is the Hopf algebra of the affine group scheme $\G\times{\hs\G}\times{^{\s^2}\!\G}\times\ldots$, where ${\hsi\G}$ is the base change of $\G$ along $\s^i\colon k\to k$.


\'{E}tale difference algebraic groups feature prominently in the general structure theory of affine difference algebraic groups, as any affine difference algebraic group $G$ fits into a short exact sequence
$$1\to G^o\to G\to\pi_0(G)\to 1,$$ 
with $G^o$ a connected affine difference algebraic group (the identity component of $G$) and $\pi_0(G)\simeq G/G^o$ an \'{e}tale difference algebraic group (the group of connected components of $G$).

The components referred to here are the components of the underlying group scheme of $G$, i.e., the affine group scheme represented by $k\{G\}$. Typically, there are infinitely many such components, i.e., $k\{\pi_0(G)\}$ is an infinite dimensional $k$-vector space. In this article we also study a true difference analog of the identity component of an algebraic group. The \emph{$\s$\=/identity component} $G^\sc$ of an affine difference algebraic group $G$ is such that $G/G^\sc$ is finite, i.e., $k\{G/G^\sc\}$ is a finite dimensional $k$-vector space.  A \emph{$\s$-infinitesimal} \'{e}tale difference algebraic group is an \'{e}tale difference algebraic group $G$ such that $G(R)=1$ whenever $\s\colon R\to R$ is injective. A $\s$-infinitesimal \'{e}tale difference algebraic group is automatically finite.

Our main result is a Jordan-H\"{o}lder type decomposition theorem for \'{e}tale difference algebraic groups.

\begin{theononumber}[Theorem \ref{theo: Babbitt}]
	Let $G$ be an \'{e}tale difference algebraic group. Then there exists a subnormal series
	$$G\supseteq G_1\supseteq G_2\supseteq \cdots\supseteq G_n\supseteq 1$$
	of difference algebraic subgroups of $G$ such that $G_1=G^\sc$, $G_n$ is $\s$-infinitesimal and $G_i/G_{i+1}\simeq [\s]_k\G_i$ for some $\s$-stably simple \'{e}tale algebraic group $\G_i$ for $i=1,\ldots,n-1$. 
	
	If $$G\supseteq H_1\supseteq H_2\supseteq \cdots\supseteq H_m\supseteq 1$$
	is another subnormal series such that $H_1=G^\sc$, $H_m$ is $\s$-infinitesimal and $H_i/H_{i+1}\simeq [\s]_k\H_i$ for some $\s$-stably simple \'{e}tale algebraic group $\H_i$ for $i=1,\ldots,m-1$, then $m=n$ and there exists a permutation $\tau$ such that $\G_i$ and $\H_{\tau(i)}$ are $\s$-stably equivalent.
\end{theononumber}

Here an \'{e}tale algebraic group $\G$ is called \emph{simple} if its only closed normal subgroups are $1$ and $\G$ and \emph{$\s$-stably simple} if ${\hsi\G}$ is simple for every $i\in \nn$. Finally, two \'{e}tale algebraic groups $\G$ and $\H$ are \emph{$\s$-stably equivalent} if there exist $i,j\in\nn$ such that ${\hsi\G}$ and ${^{\s^j}\!\H}$ are isomorphic.

\medskip

The category of \'{e}tale algebraic groups over $k$ is equivalent to the category of finite groups equipped with a continuous action of the absolute Galois group of $k$. There is a similar combinatorial-arithmetic description of \'{e}tale difference algebraic groups (Theorem \ref{theo: equivalence}): Fix an extension $\s\colon k_s\to k_s$ of $\s\colon k\to k$ to the separable algebraic closure $k_s$ of $k$. Then there exists a unique endomorphism $\s\colon\Gal\to \Gal$ of the absolute Galois group $\Gal=\operatorname{Gal}(k_s/k)$ of $k$ such that for every $\tau\in\Gal$ the diagram
$$
\xymatrix{
k_s \ar^{\s(\tau)}[r] \ar_\s[d] & k_s \ar^\s[d]\\
k_s \ar^\tau[r] & k_s	
}
$$
commutes.  Recall that an endomorphism $\s\colon \GG\to \GG$ of a profinite group $\GG$ is called \emph{expansive} if there exists a normal open subgroup $\NN$ of $\GG$ such that $\bigcap_{i\in\nn}\s^{-i}(\NN)=1$. The category of \'{e}tale difference algebraic groups over $k$ is equivalent to the category of profinite groups $\GG$ equipped with an expansive endomorphism $\s\colon\GG\to\GG$ and a continuous action of $\Gal$ that is compatible with $\s$ in the sense that $\s(\tau(g))=\s(\tau)(\s(g))$ for $\tau\in\Gal$ and $g\in\GG$.

In particular, if $k$ is separably algebraically closed, the category of \'{e}tale difference algebraic groups over $k$ is equivalent to the category of profinite groups equipped with an expansive endomorphism. The study of 
expansive endo- or automorphisms of profinite groups or more generally totally disconnected locally compact groups is an interesting topic in its own right. See e.g., \cite{Kitchens:ExpansiveDynamicsOnZeroDimensionalGroups} \cite{Willis:TheNubOfAnAutomorphismOfaTotallyDisconnectedLocallyCompactGroup}
\cite{GloecknerRaja:ExpansiveAutomorphismsOfTotallyDisconnectedLocallyCompactGroups}, \cite{Willis:TheScaleAndTidySubgroupsForEndomorphismsOfTotallyDisconnectedLocallyCompactGroups}. When translated to profinite groups (via Theorem~\ref{theo: equivalence}) our decomposition theorem recovers results proved by G. Willis (\cite[Section 6]{Willis:TheNubOfAnAutomorphismOfaTotallyDisconnectedLocallyCompactGroup}) for expansive automorphism of profinite groups. (The case of expansive endomorphism of profintite groups is somewhat more complicated than the case of expansive automorphisms but similar ideas and techniques apply.) A restatement and proof of Theorem \ref{theo: Babbitt} in the language of profinite groups can be found in \cite{Wibmer:ExpansiveDynamicsOnProfiniteGroups}.

Our decomposition theorem has some formal similarities with Babbitt's decomposition theorem 
 (\cite[Theorem 5.4.13]{Levin}), an important structure theorem for finitely generated extensions of difference fields whose underlying field extension is Galois. This similarity is no coincidence, in fact, as detailed in \cite[Section 6]{Wibmer:ExpansiveDynamicsOnProfiniteGroups}, one can deduce Babbitt's decomposition theorem from the profinite group version of Theorem \ref{theo: Babbitt} via the Galois correspondence. 
 
 \medskip

The theory of difference algebraic groups is still in its infancy, at least compared to the sister theory of differential algebraic groups, where a large body of foundational material was developed well before the turn of the century.  (See the text books \cite{Kolchin:differentialalgebraicgroups}, \cite{Buium:DifferntialAlgebraicGroupsOfFiniteDimension} and the references given there.)
Therefore, a goal of this article is also to provide some foundational results and ideas to pave the way for a further comprehensive study of affine difference algebraic groups. In this regard, our main contributions are
\begin{itemize}
	\item our study of the difference identity component of an affine difference algebraic group and the associated group of difference connected components,
	\item our study of $\s$-infinitesimal difference algebraic groups (a difference analog of infinitesimal algebraic groups),
	\item our study of the $\s$-Frobenius morphism (a difference analog of the Frobenius morphism of an algebraic group). 
\end{itemize}
Note that, while infinitesimal algebraic groups and the Frobenius morphism only make sense over a field of positive characteristic, $\s$-infinitesimal difference algebraic groups and the $\s$-Frobenius morphism make sense over an arbitrary difference field. Roughly speaking, the idea is that the abstract endomorphism $\s$ assumes the role played, in the study of algebraic groups, by the Frobenius endomorphism $a\mapsto a^p$ in characteristic $p$.

A main motivation for developing the theory of affine difference algebraic groups, is that these groups 
can be used, via appropriate Galois theories (\cite{DiVizioHardouinWibmer:DifferenceGaloisofDifferential}, \cite{OvchinnikovWibmer:SGaloisTheoryOfLinearDifferenceEquations}), to study the difference algebraic relations among the solutions of linear differential and difference equations. In this context, Theorem \ref{theo: Babbitt} sheds light on the possible difference algebraic relations among algebraic solutions of linear differential or difference equations.

\medskip

We conclude the introduction with an overview of the article. In Section \ref{sec: preliminaries and notation} we go through the details of the definition of affine difference algebraic groups and we review the known results on affine difference algebraic groups relevant for our purpose. We then embark on a general study of the difference identity component of a difference algebraic group and the associated group of difference connected components in Section \ref{sec: the difference identity component}. After a brief discussion of basic properties of \'{e}tale difference algebraic groups in Section \ref{sec: basic properties of etale difference algebraic groups}, we establish the combinatorial-arithmetic description of the category of \'{e}tale difference algebraic groups in Section \ref{sec: Expansive endomorphisms and etale difference algebraic groups}. Finally, in Section \ref{sec: decomposition theorem} we prove our decomposition theorem.

\section{Preliminaries and notation} \label{sec: preliminaries and notation}

In this preliminary section we recall the basic definitions and constructions from difference algebra. We also review the required results from \cite{Wibmer:FinitenessPropertiesOfAffineDifferenceAlgebraicGroups} and \cite{Wibmer:AlmostSimpleAffineDifferenceAlgebraicGroups} concerning affine difference algebraic groups.

\medskip

All rings are assumed to be commutative and unital. The natural numbers $\nn$ contain $0$.

\subsection{Difference algebra} \label{subsec: Difference algebra}
 
 Difference algebra is the study of difference equations from an algebraic perspective. Standard references are \cite{Cohn:difference} and \cite{Levin}.
 
 A \emph{difference ring}, or \emph{$\s$-ring} for short, is a ring $R$ together with an endomorphism $\s\colon R\to R$.  A \emph{difference field}, or \emph{$\s$-field} for short, is a difference ring whose underlying ring is a field. A \emph{$\s$-subring} of a $\s$-ring $R$ is a subring $S$ of $R$ such that $\s(S)\subseteq S$. In case $S$ and $R$ are fields, we call $R$ a \emph{$\s$-field extension} of $S$.

  A \emph{morphism of $\s$-rings} $\psi\colon R\to S$ is a morphism of rings such that
 $$
 \xymatrix{
 R \ar^-\psi[r] \ar_\s[d] & S \ar^\s[d] \\
 R \ar^-\psi[r] & S	
 }
 $$  
 commutes. In this situation, we also say that $S$ is an \emph{$R$-$\s$-algebra} or that $S$ is a \emph{difference algebra over $R$}. An \emph{$R$-$\s$-subalgebra} of an $R$-$\s$-algebra $S$ is an $R$-subalgebra of $S$ that is a $\s$-subring. A \emph{morphism of $R$-$\s$-algebras} is a morphism of $\s$-rings that also is a morphism of $R$-algebras. If $S_1$ and $S_2$ are $R$-$\s$-algebras, the tensor product $S_1\otimes_R S_2$ is an $R$-$\s$-algebra via $\s(s_1\otimes s_2)=\s(s_1)\otimes \s(s_2)$. This is, in fact, the coproduct in the category of $R$-$\s$-algebras.
 
 An $R$-$\s$-algebra $S$ is \emph{finitely $\s$-generated} (over $R$) if there exists a finite subset $B$ of $S$ such that $B,\s(B),\s^2(B),\ldots$ generates $S$ as an $R$-algebra.

 A \emph{difference ideal}, or \emph{$\s$-ideal} for short, of a $\s$-ring $R$ is an ideal $I$ of $R$ such that $\s(I)\subseteq I$. In this case, $R/I$ naturally carries the structure of a $\s$-ring such that the canonical map $R\to R/I$ is a morphism of $\s$-rings. For $F\subseteq R$, the smallest $\s$-ideal of $R$ that contains $F$ is called the \emph{$\s$-ideal $\s$-generated by $F$}. It is denoted by $[F]$. Explicitly, we have $[F]=(F,\s(F),\s^2(F),\ldots)$.

 As a matter of convenience, we usually suppress the endomorphism $\s$ in the notation, e.g., we speak of the $\s$-ring $R$, rather than the $\s$-ring $(R,\s)$. In case we have need to indicate that we consider the underlying ring without the endomorphism, we will write $R^\sharp$.
 
 Let $k$ be $\s$-field. The functor $R\rightsquigarrow R^\sharp$ from the category of \ks-algebras to the category of $k$-algebras has a right adjoint $T\rightsquigarrow [\s]_kT$ (\cite[Lemma 1.7]{Wibmer:FinitenessPropertiesOfAffineDifferenceAlgebraicGroups}). Explicitly, for a $k$-algebra $T$, the \ks-algebra $[\s]_kT$ can be described as follows. For $i\in \nn$ let ${\hsi T}=T\otimes_k k$ denote the $k$-algebra obtained from $T$ by base change via $\s^i\colon k\to k$. Set
 $$T[i]=T\otimes_k {\hs T}\otimes_k\ldots\otimes_k {\hsi T}$$
 and let $[\s]_k T$ be the union the $T[i]$'s. The endomorphism $\s\colon [\s]_kT\to [\s]_kT$ is given by
 $$\s((t_0\otimes\lambda_0)\otimes\ldots\otimes (t_i\otimes\lambda_i))=(1\otimes 1)\otimes (t_0\otimes\s(\lambda_0))\otimes\ldots\otimes(t_i\otimes\s(\lambda_i))\in T[i+1]$$
 for $(t_0\otimes\lambda_0)\otimes\ldots\otimes (t_i\otimes\lambda_i)\in{^{\s^0\!}T}\otimes_k\ldots\otimes_k{\hsi T}=T[i]$.
 Note that if $B\subseteq T$ generates $T$ as a $k$-algebra, then $B\subseteq [\s]_kT$ $\s$-generated $[\s]_k T$ as a \ks-algebra.

 \subsection{Affine difference algebraic geometry} \label{subsec: affine difference algebraic geometry}

Let $k$ be a $\s$-field. The \emph{$\s$-polynomial ring} 
$$k\{y_1,\ldots,y_n\}=k[y_1,\ldots,y_n,\s(y_1),\ldots,\s(y_n),\s^2(y_1),\ldots,\s^2(y_n),\ldots]$$
over $k$ in the $\s$-variables $y_1,\ldots,y_n$
is the polynomial ring over $k$ in the variables $\s^i(y_j)$ $(i\in \nn, 1\leq j\leq n)$ equipped with the action of $\s$ that extends $\s\colon k\to k$ and acts on the variables as suggested by their names. If $f\in k\{y_1,\ldots,y_n\}$ is a $\s$-polynomial and $a=(a_1,\ldots,a_n)\in R^n$, where $R$ is a \ks-algebra, then $f(a)\in R$ is defined by substituting $\s^i(y_j)$ with $\s^i(a_j)$. For $F\subseteq k\{y_1,\ldots,y_n\}$, the set of $R$-valued solutions of $F$ is 
$$\V_R(F)=\{a\in R^n|\ f(a)=0\ \forall \ f\in F\}.$$
Note that $R\rightsquigarrow \V_R(F)$ is naturally a functor from the category of \ks-algebras to the category of sets.

\begin{defi}
	An \emph{affine difference variety}, or \emph{affine $\s$-variety} for short, over $k$ is a functor from the category of \ks-algebras to the category of sets that is isomorphic to a functor of the form $R\rightsquigarrow \V_R(F)$ for some $n\geq 1$ and $F\subseteq k\{y_1,\ldots,y_n\}$.
\end{defi}
All difference varieties in this article are affine and for the sake of brevity we shall henceforth drop the attribute affine. A \emph{morphism of $\s$-varieties} (over $k$) is a morphism of functors, i.e., a natural transformation.

The functor $R\rightsquigarrow \V_R(F)$ is represented by the finitely $\s$-generated \ks-algebra $k\{y_1,\ldots,y_n\}/[F]$.  Indeed,
$$\Hom(k\{y_1,\ldots,y_n\}/[F], R) \to \V_R(F),\ \psi\mapsto (\psi(\overline{y_1}),\ldots,\psi(\overline{y_n}))$$
is a bijection that is functorial in $R$. As any finitely $\s$-generated \ks-algebra can be written in the form $k\{y_1,\ldots,y_n\}/[F]$, it follows that a functor from the category of \ks-algebras to the category of sets is a $\s$-variety if and only if it is representable by a finitely $\s$-generated \ks-algebra. Thus, from the Yoneda lemma we obtain:

\begin{rem} \label{rem: equivalence svarieties salgebras}
	The category of $\s$-varieties over $k$ is anti-equivalent to the category of finitely $\s$-generated \ks-algebras.
\end{rem}

For a $\s$-variety $X$ we denote its representing \ks-algebra with $k\{X\}$ and call it the \emph{coordinate ring} of $X$. We will usually identify $X$ with the functor $R\rightsquigarrow\Hom(k\{X\},R)$. For a morphism $\f\colon X\to Y$ of $\s$-varieties, the corresponding morphism $\f^*\colon k\{Y\}\to k\{X\}$ of \ks-algebras is called the \emph{morphism dual to $\f$}.

A \emph{$\s$-closed $\s$-subvariety $X$} of a $\s$-variety $Y$ is a subfunctor $X$ of $Y$ that is defined by a $\s$-ideal $\I(X)$ of $k\{Y\}$. In more detail, the requirement is that for any \ks-algebra $R$, the bijection $Y(R)\simeq \Hom(k\{Y\},R)$ maps $X(R)$ onto $\{\psi \in \Hom(k\{Y\},R)|\ \psi(\I(X))=0\}$. 
We call $\I(X)$ the \emph{defining ideal of $X$} (in $k\{Y\}$). We may write $X\subseteq Y$ to indicate that $X$ is a $\s$-closed $\s$-subvariety of $Y$.

Note that a $\s$-closed $\s$-subvariety is a $\s$-variety it its own right; it is represented by $k\{X\}=k\{Y\}/\I(X)$. The canonical map $k\{Y\}\to k\{Y\}/\I(X)$ is the dual of the inclusion morphism $X\to Y$.
The $\s$-closed $\s$-subvarieties of $Y$ are in bijection with the $\s$-ideals of $k\{Y\}$ (\cite[Lemma~1.4]{Wibmer:FinitenessPropertiesOfAffineDifferenceAlgebraicGroups}).

A morphism $\f\colon X\to Y$ of $\s$-varieties is a \emph{$\s$-closed embedding} if it induces an isomorphism between $X$ and a $\s$-closed $\s$-subvariety of $Y$. This is equivalent to $\f^*\colon k\{Y\}\to k\{X\}$ being surjective (\cite[Lemma 1.6]{Wibmer:FinitenessPropertiesOfAffineDifferenceAlgebraicGroups}).

Let $\f\colon X\to Y$ be a morphism of $\s$-varieties and let $Z$ be a $\s$-closed $\s$-subvariety of $Y$. We define a subfunctor $\f^{-1}(Z)$ of $X$ by $\f^{-1}(Z)(R)=\f_R^{-1}(Z(R))$, where $\f_R\colon X(R)\to Y(R)$ for any \ks-algebra $R$. As 
\begin{align*}
 \f^{-1}(Z)(R)&=\{\psi\in\Hom(k\{X\},R)|\ \I(Z)\subseteq\ker(\psi\f^*)\}\\
 &=\{\psi\in\Hom(k\{X\},R)|\ \f^*(\I(Z))\subseteq \ker(\psi)\},
\end{align*}
we see that $\f^{-1}(Z)$ is the $\s$-closed subscheme of $X$ defined by $\I(\f^{-1}(Z))=[\f^*(\I(Z))]\subseteq k\{X\}$.

For an affine scheme $\mathcal{X}$ of finite type over $k$, the functor $[\s]_k\X$ defined by $([\s]_k\X)(R)=\X(R)$ for any \ks-algebra $R$ is a $\s$-variety over $k$. Indeed, if $k[\X]$ is the coordinate ring of $\X$, then $\Hom(k[\X],R^\sharp)\simeq \Hom([\s]_kk[\X],R)$ for any \ks-algebra $R$. So $k\{[\s]_k\X\}=[\s]_kk[\X]$. For simplicity, we will write $k\{\X\}$ for $k\{[\s]_k\X\}=[\s]_kk[\X]$. By a \emph{$\s$-closed $\s$-subvariety of $\X$} we mean a $\s$-closed $\s$-subvariety of $[\s]_k\X$.

Let $Y$ be a $\s$-closed $\s$-subvariety of $\X$. For $i\in\nn$ let ${\hsi\X}$ be the base change of $\X$ via $\s^i\colon k\to k$ and set
$$\X[i]=\X\times{\hs\X}\times\ldots\times{\hsi\X}.$$
The \emph{$i$-th order Zariski closure $Y[i]$ of $Y$ in $\X$} is the scheme theoretic image of the morphism $Y^\sharp\to \X[i]$ of affine schemes corresponding to the morphism $k[\X[i]]\hookrightarrow k\{\X\}\to k\{Y\}$ of $k$\=/algebras. In other words, if $\I(Y)\subseteq k\{\X\}=[\s]_k k[\X]=\bigcup_{i\in\nn} k[\X[i]]$ is the defining ideal of $Y$ in $[\s]_k\X$, then $Y[i]$ is the closed subscheme of $\X[i]$ defined by the ideal $\I(Y)\cap k[\X[i]]$ of $k[\X[i]]$. We say that $Y$ is \emph{Zariski dense} in $\X$ if $Y[0]=\X$ and we refer to $Y[0]$ as the Zariski closure of $Y$ in $\X$.

Note that the projections $\pi_i\colon\X[i]\to \X[i-1],\ (x_0,\ldots,x_i)\mapsto (x_0,\ldots,x_{i-1})$ restrict to projections $\pi_i\colon Y[i]\to Y[i-1]$.

\subsection{Difference algebraic groups}

The category of $\s$-varieties over $k$ has products. Indeed, if $X$ and $Y$ are $\s$-varieties over $k$, the functor $X\times Y$ defined by $(X\times Y)(R)=X(R)\times Y(R)$ for any \ks-algebra $R$, is a product of $X$ and $Y$. It is represented by $k\{X\}\otimes_k k\{Y\}$. There also is a terminal object, namely the $\s$-variety represented by the \ks-algebra $k$. Therefore we can make the following definition.

\begin{defi}
	A \emph{$\s$-algebraic group} (over $k$) is a group object in the category of $\s$-varieties (over $k$). 
\end{defi}
In other words, a $\s$-algebraic group over $k$ is a functor from the category of \ks-algebras to the category of groups such that the correspond functor to the category of sets is representable by a finitely $\s$-generated \ks-algebra.

A morphism of $\s$-algebraic groups $\f\colon G\to H$ is a morphism of $\s$-varieties such that $\f_R\colon G(R)\to H(R)$ is a morphism of groups for any \ks-algebra. See \cite[Section 2]{Wibmer:FinitenessPropertiesOfAffineDifferenceAlgebraicGroups} for a list of examples of $\s$-algebraic groups.

A \emph{$\s$-closed subgroup} of a $\s$-algebraic group $G$ is a $\s$-closed $\s$-subvariety $H$ of $G$ such that $H(R)$ is a subgroup of $G(R)$ for any \ks-algebra $R$. We may write $H\leq G$ to indicate that $H$ is a $\s$-closed subgroup of $G$. A \emph{$\s$-closed embedding of $\s$-algebraic} is a morphism of $\s$-algebraic groups that is a $\s$-closed embedding of $\s$-varieties.

A \emph{\ks-Hopf algebra} is a \ks-algebra $R$ equipped with the the structure of a Hopf algebra such that the Hopf algebra structure maps (the comultiplication $\Delta\colon R\to R\otimes_k R$, the counit $\varepsilon\colon R\to k$ and the antipode $S\colon R\to R$) are morphisms of \ks-algebras.
From Remark \ref{rem: equivalence svarieties salgebras} we obtain:
\begin{rem}
	The category of $\s$-algebraic groups over $k$ is anti-equivalent to the category of finitely $\s$-generated \ks-Hopf algebras.
\end{rem}

For a $\s$-algebraic group $G$, we write $\m_G$ for the kernel of the counit $\varepsilon\colon k\{G\}\to k$. Note that $\m_G$ defines the trivial subgroup $1$ of $G$.

\begin{lemma}[{\cite[Lemma 2.15]{Wibmer:FinitenessPropertiesOfAffineDifferenceAlgebraicGroups}}] \label{lemma: Hopf induced}
	Let $R$ be a \ks-Hopf algebra and $S$ a $k$-Hopf algebra. If $S\to R$ is a morphism of $k$-Hopf algebras, the induced morphism $[\s]_k S\to R$ is a morphism of \ks-Hopf algebras. 
\end{lemma}

\begin{ex} \label{ex: finite constant}
	To any finite group $\GG$ equipped with an endomorphism $\s\colon \GG\to\GG$ one can associate a $\s$-algebraic group $G$. Since we will refer to this example later, we explain the details. For any \ks-algebra $R$, let $G(R)$ denote the set of all locally constant functions $f\colon \spec(R)\to \GG$ such that
	$$
	\xymatrix{
		\spec(R) \ar^-f[r] \ar_\s[d] & \GG \ar^\s[d] \\
		\spec(R) \ar^-f[r] & \GG
	}
	$$
	commutes, where $\s\colon \spec(R)\to\spec(R)$ is the continuous map induced by $\s\colon R\to R$. Then $G(R)$ is a group under pointwise multiplication.

	Let $k^\GG$ be the finite dimensional $k$-algebra of all maps from $\GG$ to $k$. As explained in \cite[Section 2.3]{Waterhouse:IntrotoAffineGroupSchemes} the $k$-algebra $k^\GG$ naturally has the structure of a $k$-Hopf algebra. Defining $\s\colon k^\GG\to k^\GG$ by $\s(h)(g)=\s(h(\s(g)))$ for $h\colon\GG\to k$ and $g\in\GG$ defines the structure of \ks-Hopf algebra on $k^\GG$. One can show (\cite[Example 2.14]{Wibmer:FinitenessPropertiesOfAffineDifferenceAlgebraicGroups}) that $G$ is represented by the \ks-Hopf algebra $k\{G\}=k^\GG$.
\end{ex}

For further examples of $\s$-algebraic groups see \cite[Section 2]{Wibmer:FinitenessPropertiesOfAffineDifferenceAlgebraicGroups}. Before further discussing $\s$-algebraic groups, let us agree on the following conventions. \medskip

{\bf Notation for algebraic groups:} We us the term ``algebraic group (over $k$)'' synonymous for ``affine group scheme of finite type (over $k$)''. The coordinate ring, i.e., the ring of global section of an algebraic group $\G$ is denoted by $k[\G]$. Following \cite[Def. 5.5]{Milne:AlgebraicGroupsTheTheoryofGroupSchemesOfFiniteTypeOverAField} a morphism $\f\colon \G\to \H$ of algebraic groups is a \emph{quotient map} if the dual map $\f^*\colon k[\H]\to k[\G]$ is injective (equivalently, faithfully flat). By a \emph{closed subgroup} of an algebraic group we mean a closed subgroup scheme.

\medskip

{\bf From now on and throughout this article $k$ denotes an arbitrary $\s$-field. All $\s$\=/varieties, $\s$-algebraic groups and algebraic groups are assumed to be over $k$ (unless the contrary is explicitly indicated).}

\medskip

If $\f\colon X\to Y$ is a morphism of $\s$-varieties, there exists a unique $\s$-closed $\s$-subvariety $\f(X)$ of $Y$ such that $\f$ factors through the inclusion $\f(X)\subseteq Y$ and if $Z$ is any $\s$-closed $\s$-subvariety of $Y$ such that $\f$ factors through $Z$, then $\f(X)\subseteq Z$ (\cite[Lemma 1.5]{Wibmer:FinitenessPropertiesOfAffineDifferenceAlgebraicGroups}). Indeed, $\f(X)$ is the $\s$-closed $\s$-subvariety of $Y$ defined by $\I(\f(X))=\ker(\f^*)$. If $\f\colon G\to H$ is a morphism of $\s$-algebraic groups, then $\f(G)$ is a $\s$-closed subgroup of $H$.

If $\G$ is an algebraic group over $k$, then $[\s]_k\G$ is a $\s$-algebraic group. By a \emph{$\s$-closed subgroup of $\G$}, we mean a $\s$-closed subgroup of $[\s]_k\G$. For $i\in\nn$ the affine scheme $\G[i]=\G\times\ldots\times {\hsi\G}$ is an algebraic group and if $G$ is a $\s$-closed subgroup $\G$, then $G[i]$ is a closed subgroup of $\G[i]$. The projections $\pi_i\colon G[i]\to G[i-1]$ are morphisms of algebraic groups. In fact, they are quotient maps.

A basic fact about $\s$-algebraic groups is that every $\s$-algebraic group is isomorphic to a $\s$-closed subgroup of an algebraic group.

\begin{prop}[{\cite[Prop. 2.16]{Wibmer:FinitenessPropertiesOfAffineDifferenceAlgebraicGroups}}] \label{prop: linearization}
	Let $G$ be a $\s$-algebraic group. Then there exists an algebraic group $\G$ and a $\s$-closed embedding $G\to [\s]_k\G$.
\end{prop} 

These $\s$-closed embeddings of $\s$-algebraic groups into algebraic groups can be used to define three numerical invariants of a $\s$-algebraic group $G$: the \emph{$\s$-dimension} $\sdim(G)$, the \emph{order} $\ord(G)$ and the \emph{limit degree} $\ld(G)$.

\begin{theo}[{\cite[Theorem 3.7]{Wibmer:FinitenessPropertiesOfAffineDifferenceAlgebraicGroups}}] \label{theo: existence of sdim and order}
	Let $G$ be a $\s$-closed subgroup of an algebraic group $\G$. For $i\in\nn$ let $d_i=\dim(G[i])$ denote the dimension of the $i$-th order Zariski closure of $G$ in $\G$. Then there exist $d,e\in \nn$ such that 
	$d_i=d(i+1)+e$ for all sufficiently large $i$. The integer $d$ only depends on $G$ (and not on the choice of the $\s$-closed embedding of $G$ into $\G$). Moreover, if $d=0$, the integer $e$ only depends on $G$.
\end{theo}

The integer $d$ of Theorem \ref{theo: existence of sdim and order} is called the $\s$-dimension of $G$. If $\sdim(G)=0$, the integer $e$ of Theorem \ref{theo: existence of sdim and order} is called the order of $G$.

For an algebraic group $\G$ we denote with $|\G|$ the dimension of $k[\G]$ as a $k$-vector space. (This is infinite if $\G$ has positive dimension.)

\begin{prop} \label{prop: limit degree}
	Let $G$ be a $\s$-closed subgroup of an algebraic group $\G$ and for $i\in\nn$ let $G[i]$ denote the $i$-th order Zariski closure of $G$ in $\G$. Let $\G_i$ denote the kernel of the projection $\pi_i\colon G[i]\to G[i-1]$. Then the sequence $(|\G_i|)_{i\in\nn}$ is non-increasing and therefore eventually constant. The eventual value $\lim_{i\to\infty}|\G_i|$ only depends on $G$ (and not on the choice of the $\s$-closed embedding of $G$ into $\G$).
\end{prop}
\begin{proof}
	This follows by combining Propositions 4.1 and 5.1 in \cite{Wibmer:FinitenessPropertiesOfAffineDifferenceAlgebraicGroups}. 
\end{proof}

The value $\ld(G)=\lim_{i\to\infty}|\G_i|$ from Proposition \ref{prop: limit degree} is called the limit degree of $G$. Note that $\ld(G)$ is finite if and only if $\sdim(G)=0$.
The following Lemma explains the meaning of $\ld(G)=1$.

\begin{lemma}[{\cite[Lemma 5.7]{Wibmer:FinitenessPropertiesOfAffineDifferenceAlgebraicGroups}}] \label{lemma: ld=1}
	Let $G$ be a $\s$-algebraic group. Then $\ld(G)=1$ if and only if $k\{G\}$ is finitely generated as a $k$-algebra.
\end{lemma}

We next discuss quotients. A $\s$-closed subgroup $N$ of a $\s$-algebraic group $G$ is \emph{normal} if $N(R)$ is a normal subgroup of $G(R)$ for any \ks-algebra $R$. We may write $N\unlhd G$ to indicate that $N$ is a normal $\s$-closed subgroup of $G$.

 The \emph{kernel} $\ker(\f)$ of a morphism $\f\colon G\to H$ of $\s$-algebraic groups is defined by $\ker(\f)(R)=\ker(\f_R)$ for any \ks-algebra $R$. Since $\ker(\f)=\f^{-1}(1)$, where $1$ is the trivial subgroup of $H$ (defined by the kernel $\m_H$ of the counit $k\{H\}\to k$), we see that $\ker(\f)$ is the normal $\s$-closed subgroup of $G$ defined by $\I(\ker(\f))=(\f^*(\m_H))\subseteq k\{G\}$.

 A \emph{quotient of $G$ mod $N$} is a $\s$-algebraic group $G/N$ together with a morphism $\pi\colon G\to G/N$ of $\s$-algebraic groups such that $N\subseteq\ker(\pi)$ and for any other morphism $\f\colon G\to H$ of $\s$-algebraic groups such that $N\subseteq\ker(\f)$ there exists a unique morphism $\f'\colon G/N\to H$ such that
 $$
 \xymatrix{
 G \ar^-\pi[rr] \ar_-\f[rd] & & G/N \ar^-{\f'}@{..>}[ld] \\
 & H &	
 }
 $$
commutes.

\begin{theo}[{\cite[Theorem 3.3]{Wibmer:AlmostSimpleAffineDifferenceAlgebraicGroups}}] \label{theo: existence of quotients}
	Let $G$ be $\s$-algebraic group and $N$ a normal $\s$-closed subgroup of $G$. Then a quotient of $G$ mod $N$ exists. Moreover, a morphism $\pi\colon G\to G/N$ of $\s$-algebraic groups is a quotient of $G$ mod $N$ if and only if $N=\ker(\pi)$ and $\pi^*\colon k\{G/N\}\to k\{G\}$ is injective.
\end{theo}

A morphism $\f\colon G\to H$ of $\s$-algebraic groups is a \emph{quotient map} if it is a quotient of $G$ mod $N$ for some normal $\s$-closed subgroup of $G$. Equivalently, $\f(G)=H$, i.e., $\f^*\colon k\{H\}\to k\{G\}$ is injective. See \cite[Prop. 4.10]{Wibmer:AlmostSimpleAffineDifferenceAlgebraicGroups} for further characterizations of quotient maps.

A sequence $1\to N\xrightarrow{\alpha} G\xrightarrow{\beta} H\to 1$ of morphisms of $\s$-algebraic groups is \emph{exact} if $\alpha$ is a $\s$-closed embedding, $\beta$ is a quotient map and $\alpha(N)=\ker(\beta)$.

A \emph{\ks-Hopf subalgebra} of a \ks-Hopf algebra is a Hopf subalgebra that is also a \ks-subalgebra.

\begin{theo}[{\cite[Theorem 4.5]{Wibmer:FinitenessPropertiesOfAffineDifferenceAlgebraicGroups}}] \label{theo: ksHopfsubalgebra finitely sgenerated}
	A \ks-Hopf subalgebra of a finitely $\s$-generated \ks\=/Hopf algebra is finitely $\s$-generated.
\end{theo}

\begin{cor} \label{cor: quotients correspond to ksHopfsubalgebras}
	Let $G$ be a $\s$-algebraic group. There is a bijection between the normal $\s$-closed subgroups of $G$ and the \ks-Hopf subalgebras of $k\{G\}$.
\end{cor}
\begin{proof}
	If $N$ is a normal $\s$-closed subgroup of $G$ and $\pi^*\colon G\to G/N$ the corresponding quotient, then $\pi^*(k\{G/N\})$ is \ks-Hopf subalgebra of $k\{G\}$. Conversely, if $R$ is a \ks-Hopf subalgebra of $k\{G\}$, then $R$ is finitely $\s$-generated by Theorem \ref{theo: ksHopfsubalgebra finitely sgenerated}. So $R=k\{H\}$ for some $\s$-algebraic group $H$. The kernel $N$ of the morphism $G\to H$ of $\s$-algebraic groups corresponding to the inclusion $R\subseteq k\{G\}$ is a normal $\s$-closed subgroup of $G$. By Theorem \ref{theo: existence of quotients} these two constructions are inverse to each other.
\end{proof}

The three numerical invariants are well-behaved under quotients.

\begin{prop} \label{prop: invariants and quotients}
	Let $N$ be a normal $\s$-closed subgroup of a $\s$-algebraic group $G$. Then
	$$\sdim(G)=\sdim(N)+\sdim(G/N),$$
	$$\ord(G)=\ord(N)+\ord(G/N)$$ and
	$$\ld(G)=\ld(G/N)\cdot\ld(N).$$
\end{prop}
\begin{proof}
	This is Corollaries 3.13 and 3.15 in \cite{Wibmer:FinitenessPropertiesOfAffineDifferenceAlgebraicGroups}.
\end{proof}
The formulas in Proposition \ref{prop: invariants and quotients} are written in a form so that they still make sense in case infinite values are involved. For example, if $\ld(G)$ is finite, then also $\ld(N)$ and $\ld(G/N)$ are finite and $\ld(G/N)=\frac{\ld(G)}{\ld(N)}$.

If $X$ is a $\s$-variety over $k$ and $R$ a \ks-algebra, we denote with $X_R$ the functor from the category of $R$-$\s$-algebras to the category of sets such that $X_R(R')=X(R')$ for any $R$-$\s$-algebra $R'$. Note that $X_R$ is represented by $k\{X\}\otimes_k R$. In particular, if $R=k'$ is a $\s$-field extension of $k$, then $X_{k'}$ is a $\s$-variety over $k'$, called the \emph{base change of $X$} along $k\to k'$.

Quotients have all the expected good properties, for example:
\begin{lemma}[{\cite[Lemma 3.9]{Wibmer:AlmostSimpleAffineDifferenceAlgebraicGroups}}] \label{lemma: quotients and base change}
	Let $N$ be a normal $\s$-closed subgroup of a $\s$-algebraic group $G$ and let $k'$ be a $\s$-field extension of $k$. Then $(G/N)_{k'}=G_{k'}/N_{k'}$.
\end{lemma}

\begin{lemma}[{\cite[Cor. 3.4]{Wibmer:AlmostSimpleAffineDifferenceAlgebraicGroups}}] \label{lemma: quotient embedding}
	If $\f\colon G\to H$ is a morphism of $\s$-algebraic groups, the induced morphism $G/\ker(\f)\to H$ is a $\s$-closed embedding.
\end{lemma}

\begin{ex}[{\cite[Example 3.7]{Wibmer:AlmostSimpleAffineDifferenceAlgebraicGroups}}]
	If $\N$ is a normal closed subgroup of an algebraic group $\G$. Then $[\s]_k\G/[\s]_k\N=[\s]_k(\G/\N)$.
\end{ex}

The isomorphism theorems hold for $\s$-algebraic groups. In particular, we have the first isomorphism theorem:

\begin{theo}[{\cite[Theorem 5.2]{Wibmer:AlmostSimpleAffineDifferenceAlgebraicGroups}}] \label{theo: isom1}
	Let $\f\colon G\to H$ be a morphism of $\s$-algebraic groups. Then the induced morphism $G/\ker(\f)\to \f(G)$ is an isomorphism of $\s$-algebraic groups.
\end{theo}

As a corollary to the above theorem we obtain:

\begin{cor} \label{cor: sclosed and quotient is isom}
	A morphism of $\s$-algebraic groups that is a quotient map and a $\s$-closed embedding is an isomorphism. \qed
\end{cor}

The following is the third isomorphism theorem for $\s$-algebraic groups.

\begin{theo}[{\cite[Theorem 5.9]{Wibmer:AlmostSimpleAffineDifferenceAlgebraicGroups}}] \label{theo: isom 3}
	Let $N$ be a normal $\s$-closed subgroup of a $\s$-algebraic group $G$ with quotient map $\pi\colon G\to G/N$. Then the map $H\mapsto \pi(H)=H/N$ defines a bijection between the $\s$-closed subgroups $H$ of $G$ containing $N$ and the $\s$-closed subgroup of $G/N$. The inverse is $H'\mapsto \pi^{-1}(H')$. A $\s$-closed subgroup $H$ of $G$ containing $N$ is normal in $G$ if and only if $H/N$ is normal in $G/N$.  In this case the canonical morphism $G/H\to (G/N)/(H/N)$ is an isomorphism. 
\end{theo}

Let $G$ be a $\s$-algebraic group. A \emph{subnormal series} of $G$ is a sequence
\begin{equation} \label{eq: subnormal series}
G=G_0\supseteq G_1\supseteq\ldots\supseteq G_n=1
\end{equation}
of $\s$-closed subgroups of $G$ such that $G_{i+1}$ is normal in $G_i$ for $i=0,\ldots,n-1$. Another subnormal series 
\begin{equation} \label{eq: subnormal series 2}
G=H_0\supseteq H_1\supseteq\ldots\supseteq H_m=1
\end{equation}
is a refinement of (\ref{eq: subnormal series}) if $\{G_0,\ldots,G_n\}\subseteq \{H_1,\ldots,H_m\}$.
The subnormal series (\ref{eq: subnormal series}) and (\ref{eq: subnormal series 2}) are equivalent if $m=n$ and there exists a permutation $\pi$ such that the factor groups $G_i/G_{i+1}$ and $H_{\pi(i)}/H_{\pi(i)+1}$ are isomorphic for $i=0,\ldots,n-1$.

Our main main decomposition theorem (Theorem \ref{theo: Babbitt}) is reminiscent of the Jordan-H\"{o}der theorem.
The standard proof of the uniqueness part of the Jordan-H\"{o}der theorem proceeds through the Schreier refinement theorem. The following is the Schreier refinement theorem for $\s$-alegbraic groups.

\begin{theo}[{\cite[Theorem 7.5]{Wibmer:AlmostSimpleAffineDifferenceAlgebraicGroups}}] \label{theo: Schreier refinement}
	Any two subnormal series of a $\s$-algebraic group have equivalent refinements.
\end{theo}

\section{The difference identity component} \label{sec: the difference identity component}

There are three types of ``identity components'' for a $\s$-algebraic group $G$. The \emph{identity component} $G^o$, the \emph{$\s$-identity component} $G^\sc$ and the \emph{strong identity component} $G^{so}$. The identity component and the strong identity component are discussed in \cite[Section 6]{Wibmer:AlmostSimpleAffineDifferenceAlgebraicGroups}. These are not relevant for the purpose of this article. However, it might be interesting to note that, more or less by definition, a $\s$-algebraic group $G$ is $\s$-\'{e}tale if and only if $G^o=1$.

In this section we discuss the $\s$-identity component. It plays a crucial role in our decomposition theorem (Theorem \ref{theo: Babbitt}). The $\s$-identity component was already used in \cite{BachmayrWibmer:AlgebraicGroupsAsDifferenceGaloisGroupsOfLinearDifferentialEquations} to derive a necessary condition on a $\s$-algebraic group to be a $\s$-Galois group over the difference-differential field $\C(x)$ with derivation $\de=\frac{d}{dx}$ and endomorphism $\s$ given by $\s(f(x))=f(x+1)$. Some difference algebraic results necessary to define the difference identity component also already appeared in \cite{TomasivWibmer:StronglyEtaleDifferenceAlgebras}.

We will give a difference topological interpretation of the difference identity component. Therefore, we begin with some general difference topological definitions and observations.

\subsection{The difference topology} \label{subsec: the difference topology}


Let $X$ be a topological space equipped with a continuous endomorphism $\s\colon X\to X$. In this situation we may call $X$ a \emph{$\s$-topological space}. A morphism of $\s$-topological spaces is a continuous map that commutes with the action of $\s$. A subset $V$ of $X$ is called \emph{$\s$-invariant} if $\s(V)\subseteq V$. We call $V$ \emph{$\s$-closed} if it is closed and $\s$-invariant. The \emph{$\s$-topology} on $X$ is the topology on $X$ whose closed sets are the $\s$-closed sets.  The \emph{$\s$-connected components} of $X$ are the connected components with respect to the $\s$-topology. We call $X$ \emph{$\s$-connected} if $X$ is connected with respect to the $\s$-topology.

%
%


We are mainly interested in the following example: For a $\s$-ring $R$,  $\spec(R)$ is a naturally a $\s$-topological space. The topology is the usual Zariski topology and
$$\sigma\colon \spec(R)\to\spec(R),\ \p\mapsto \s^{-1}(\p)$$
is the continuous endomorphism induced by $\s\colon R\to R$.

%
%

For a subset $B$ of $R$, let us denote with $\VV(B)=\{\p\in\spec(R)|\ B\subseteq \p\}$ the closed subset of $\spec(R)$ defined by $B$. The map $\ida\mapsto \VV(\ida)$ is a bijection between the set of radical $\s$-ideals of $R$ and the set of $\s$-closed subsets of $\spec(R)$.

%
%
%

Our first goal is to express the property that $\spec(R)$ is $\s$-connected via difference algebraic conditions on $R$. Recall that an element $e$ of a ring is called \emph{idempotent} if $e^2=e$. The \emph{trivial} idempotents are $1$ and $0$. The spectrum of a ring $R$ is connected if and only if $R$ has no non-trivial idempotent elements (\cite[\href{https://stacks.math.columbia.edu/tag/00EF}{Tag 00EF}]{stacks-project}). We will prove a difference analog of this result.


An element $f$ of a $\s$-ring $R$ is called \emph{constant} \index{constant element} if $\s(f)=f$. The subring of all constant elements is denoted by $R^\s.$

\begin{prop} \label{prop: sspec(R) sconnected}
	Let $R$ be a $\s$-ring. If $e\in R$ is a non-trivial constant idempotent element, then
	$$\spec(R)=\VV(e)\uplus\VV(1-e)$$ is a decomposition of $\spec(R)$ into disjoint non-empty $\s$-closed subsets. Conversely, if $\spec(R)=X\uplus Y$ is a decomposition of $\spec(R)$ into disjoint non-empty $\s$-closed subsets, then there exists a non-trivial constant idempotent element $e\in R$ such that $X=\VV(e)$ and $Y=\VV(1-e)$.
\end{prop}
\begin{proof}
	Set $e'=1-e$. Then also $e'$ is constant and idempotent. Because $e$ is constant, the ideal generated by $e$ is a $\s$-ideal and therefore $\VV(e)$ is $\s$-closed. Similarly, $\VV(e')$ is $\s$-closed.
	Since $ee'=0$, every $\p\in\spec(R)$ is contained in $\VV(e)$ or $\VV(e')$. Since $e+e'=1$, no $\p$ can be contained in $\VV(e)$ and $\VV(e)$. Thus $\spec(R)$ is the disjoint union of the $\s$-closed sets $\VV(e)$ and $\VV(e')$.
	
	Suppose $\VV(e)=\emptyset$. Then $1$ must lie in the ideal generated by $e$. So $1=re$ for some $r\in R$. Therefore $e=re^2=re$. Combining the last two equations yields $e=1$; a contradiction. Suppose $\VV(e)=\spec(R)$. Then $e$ must lie in the nilradical of $R$. Because $e$ is idempotent this implies $e=0$; again a contradiction.
	
	\medskip
	
	Now assume that $\spec(R)=X\uplus Y$ is a decomposition of $\spec(R)$ into disjoint non-empty $\s$\=/closed subsets. It is known that any decomposition of $\spec(R)$ into two disjoint non-empty closed subsets arises from a pair $e,e'$ of non-trivial idempotent elements with $ee'=0$ and $e+e'=1$ (\cite[\href{https://stacks.math.columbia.edu/tag/00EE}{Tag 00EE}]{stacks-project}). So $X=\VV(e)$ and $Y=\VV(e')$. It remains to show that $e$ is constant.
	
	Let $\p\in X=\VV(e)$. Since $X$ is $\s$-invariant, $\s^{-1}(\p)\in X$, i.e., $e\in\s^{-1}(\p)$. Thus $\s(e)\in \p$ and $\VV(e)\subseteq\VV(\s(e))$. Similarly, one shows $\VV(e')\subseteq\VV(\s(e'))$. Because $\s(e)\s(e')=0$ and $\s(e)+\s(e')=1$, we have $\spec(R)=\VV(\s(e))\uplus\VV(\s(e'))$. This implies $\VV(e)=\VV(\s(e))$ and $\VV(e')=\VV(\s(e'))$.
	Consequently $\sqrt{(e)}=\sqrt{(\s(e))}$. Because $e$ is idempotent, it follows that $e=a\s(e)$ for some $a\in R$. Since $(1-\s(e))\s(e)=0$, this implies $(1-\s(e))e=0$. Interchanging the roles of $e$ and $\s(e)$ in the last argument, we find $(1-e)\s(e)=0$. Taking the difference of the last two equations yields $\s(e)=e$ as desired.	 
\end{proof}
From Proposition \ref{prop: sspec(R) sconnected} we immediately obtain:

\begin{cor} \label{cor: spec(R) sconnected iff non constant idempotent}
	Let $R$ be a $\s$-ring. Then $\spec(R)$ is $\s$-connected if and only if $R$ contains no non-trivial constant idempotent element. \qed
\end{cor}

Recall that idempotent elements $e_1,\ldots,e_n$ of a ring $R$ are called \emph{orthogonal} if $e_ie_j=0$ for $i\neq j$.

\begin{prop} \label{prop: sconnected components of spec(R)}
	Let $R$ be a $\s$-ring and let $e_1,\ldots,e_n\in R$ be constant orthogonal idempotent elements with $e_1+\cdots+e_n=1$ such that no $e_i$ can be written as a sum of two non-trivial constant orthogonal idempotent elements. Then the $\s$-connected components of $\spec(R)$ are $\VV(1-e_1),\ldots,\VV(1-e_n)$. Moreover, $e_iR$ is naturally a $\s$-ring and $\spec(e_iR)$ is isomorphic to $\VV(1-e_i)$ as a $\s$-topological space.  
\end{prop}
\begin{proof}
	Note that $e_iR$ is a $\s$-ideal of $R$, rather than a $\s$-subring of $R$. However, $e_iR$ is also $\s$-ring with identity element $e_i$. As $e_iR$ and $R/(1-e_i)$ are isomorphic as $\s$-rings, we see that $\spec(e_iR)$ and $\VV(1-e_i)$ are isomorphic as $\s$-topological spaces.
	
	Assume $e\in e_iR$ is a constant idempotent element. Then $e_i=e_ie+e_i(1-e)$ expresses $e_i$ as a sum of orthogonal constant idempotent elements. By assumption $e_ie=e_i$ or $e_ie=0$. But $e_ie=e$, so $e=e_i$ or $e=0$. It follows from Corollary \ref{cor: spec(R) sconnected iff non constant idempotent} that $\VV(1-e_i)\simeq\spec(e_iR)$ is $\s$-connected.
	
	Because the $e_i$'s are orthogonal, the $\VV(1-e_i)$'s are disjoint and because $e_1+\cdots+e_n=1$, their union equals $\spec(R)$. In summary we see that the $\VV(1-e_i)$'s are the $\s$-connected components of $\spec(R)$.
\end{proof}


To show that $\spec([\s]_kT)$ is $\s$-connected for any $k$-algebra $T$, we will need the following result.

\begin{lemma} \label{lemma: constants of sR}
	Let $T$ be a $k$-algebra. Then $([\s]_k T)^\s=k^\s$.
\end{lemma}
\begin{proof}
	Let $(v_i)_{i\in I}$ be a $k$-basis of $T\subseteq [\s]_k T$ that contains $1$. For $n\geq0$ and ${\bf i}=(i_0,\ldots,i_n)\in I^{n+1}$ define $v_{\bf{i}}=v_{i_0}\s(v_{i_1})\cdots\s^n(v_{i_n})$. By construction of $T[n]$ (Section \ref{subsec: Difference algebra}), $V_n:=(v_{\bf i})_{{\bf i}\in I^{n+1}}$ is a $k$\=/basis of $T[n]$. Because $(v_i)_{i\in I}$ contains $1$, we have $V_n\subseteq V_{n+1}$ and $\s(V_n)\subseteq V_{n+1}$.
	
	Let	$a$ be a constant element of $[\s]_k T$ and suppose $a$ does not lie in $k$. Let $n\geq 0$ be minimal such that $a\in T[n]$. Then we can write
	$a=\sum_{v\in V_n} \lambda_v v$
	with $\lambda_v\in k$, where, for some $v\in V_n\smallsetminus V_{n-1}$, the coefficient $\lambda_v$ is non-zero. As $a$ is constant
	\begin{equation} \label{eqn: sum a}
	a=\sum_{v\in V_n}\sigma(\lambda_v)\s(v)	
	\end{equation}
	If $v\in V_n\smallsetminus V_{n-1}$, then $\s(v)\in V_{n+1}\smallsetminus V_n$.
	Therefore equation (\ref{eqn: sum a}) shows that $a\notin T[n]$, a contradiction. 
\end{proof}

The following corollary shows that $\s$-connectedness is a concept that truly belongs to the $\s$-world. The connectedness of $\spec(T)$ and the $\s$-connectedness of $\spec([\s]_k T)$ are unrelated.
\begin{cor} \label{cor: specsR sconnected}
	Let $T$ be a $k$-algebra. Then $\spec([\s]_k T)$ is $\s$-connected.
\end{cor}
\begin{proof}
	By Corollary \ref{cor: spec(R) sconnected iff non constant idempotent} it suffices to show that $[\s]_k T$ has no non-trivial constant idempotent element. But by Proposition \ref{lemma: constants of sR} every constant idempotent of $[\s]_k T$ belongs to $k$ and must therefore be trivial.
\end{proof}

\subsection{Strongly \'{e}tale difference algebras and the strong core}

To motivate the definitions in this subsection, let us recall a possible path to the definition of the connected component $\G^o$ of an algebraic group $\G$. See e.g., \cite[Chapter 6]{Waterhouse:IntrotoAffineGroupSchemes}. Instead of defining $\G^o$ directly, one first constructs the coordinate ring of $\G/\G^o$. Let $\pi_0(k[\G])$ be the union of all \'{e}tale $k$-subalgebras of $k[\G]$. One shows that $\pi_0(k[\G])$ is an \'{e}tale $k$-algebra and a Hopf subalgebra of $k[\G]$. One can then define $\G^o$ as the kernel of the morphism $\G\to \pi_0(\G)$ of algebraic groups corresponding to the inclusion $\pi_0(k[\G])\subseteq k[\G]$ of Hopf algebras.

We will follow here a similar path. The first step is to clarify, what is the appropriate difference analog of an \'{e}tale $k$-algebra in our context. This question is addressed in the following definitions.

\begin{defi}
	A \ks-algebra $R$ is \emph{$\s$-separable} (over $k$) if $\s\colon R\otimes_k k'\to R\otimes_k k'$ is injective for every $\s$-field extension $k'$ of $k$.
\end{defi} 
See \cite[Prop. 1.2]{TomasivWibmer:StronglyEtaleDifferenceAlgebras} for other equivalent characterizations of $\s$-separable \ks-algebras. E.g., a \ks-algebra $R$ is $\s$-separable if and only if $\overline{\s}\colon {\hs R}\to R,\ f\otimes \lambda\mapsto\s(f)\lambda$ is injective. Here ${\hs R}=R\otimes_k k$ is the base change of $R$ via $\s\colon k\to k$.

 Recall (\cite[Chapter V, \S 6]{Bourbaki:Algebra2}) that a $k$-algebra $T$ is \emph{\'{e}tale} if $T\otimes_k\overline{k}$ is isomorphic (as a $\overline{k}$\=/algebra) to a finite direct product of copies of the algebraic closure $\overline{k}$ of $k$. (In particular, $R$ is finite dimensional as a $k$-vector space.)
Following \cite[Def. 1.7]{TomasivWibmer:StronglyEtaleDifferenceAlgebras} we make the following definition.

\begin{defi} \label{defi: strongly setale}
	A $k$-$\s$-algebra is \emph{\ssetale{}} if it is $\s$-separable over $k$ and \'{e}tale as a $k$-algebra. A $\s$-algebraic group $G$ is \emph{\ssetale{}} if $k\{G\}$ is \ssetale{}.
\end{defi}

Let us see some examples of \ssetale{} $\s$-algebraic groups.

\begin{ex} \label{ex: ssetale if and only if}
	Let $0\leq \alpha <n$ and $m\geq 1$ be integers and let $G$ be the $\s$-closed subgroup of the multiplicative group $\Gm$ given by
	$$G(R)=\{g\in R^\times|\ g^n=1,\ \s^m(g)=g^\alpha\}$$ 
	for any \ks-algebra $R$. Then $k\{G\}=k[x,\s(x),\ldots,\s^{m-1}(x)]$, where $x$ denotes the image of the coordinate function on $\Gm$.
	Let us assume that the characteristic of $k$ is zero or does not divide $n$. Then the polynomial $y^n-1$ is separable over $k$ and it follows that $k\{G\}$ is an \'{e}tale $k$-algebra.
	
	We claim that $G$ is \ssetale{} if and only if $\alpha$ and $n$ are relatively prime. Indeed, if there are $a,b\in\nn$ with $1\leq a<n$ and $\alpha a=bn$, then $$\s(\s^{m-1}(x)^a-1)=x^{\alpha a}-1=x^{b n}-1=0$$ and so $\s$ is not injective on $k\{G\}$ and therefore $G$ is not \ssetale{}. On the other hand, if $1=a\alpha+bn$ for $a,b\in\mathbb{Z}$, then
	$$\s(\s^{m-1}(x)^a)=\s^m(x)^a=x^{a\alpha}=x^{a\alpha}x^{bn}=x.$$ 
	This shows that ${\hs(k\{G\})}\to k\{G\}$ is surjective. Because ${\hs(k\{G\})}\to k\{G\}$ is a morphism of finite dimensional $k$-algebras, it must then also be injective. So $k\{G\}$ is $\s$-separable and therefore strongly $\s$-\'{e}tale.
\end{ex}

\begin{ex} \label{ex: split finti ssetale if and only if auto} 
	Let $\GG$ be a finite group with an endomorphism $\s\colon \GG\to \GG$ and let $G$ be the $\s$\=/algebraic group associated to these data as in Example \ref{ex: finite constant}. We claim that $G$ is \ssetale{} if and only if $\s\colon \GG\to \GG$ is an automorphism.
	Clearly $k\{G\}=k^\GG$ is \'{e}tale, so the question is about the $\s$-separability of $k\{G\}$.
	
	If $\s\colon\GG\to\GG$ is not an automorphism, there exist a $g\in\GG$ that does not lie in the image of $\s\colon\GG\to\GG$. Define $h\colon \GG\to k$ by $$h(g')=\begin{cases}
	1 \text{ if } g'=g, \\
	0 \text{ otherwise.}
\end{cases}$$
	Then $\s(h)(g')=\s(h(\s(g')))=\s(0)=0$ for any $g'\in\GG$. Thus $\s(h)=0$ and $\s\colon k\{G\}\to k\{G\}$ is not injective. So $G$ is not \ssetale{}.
	
	Conversely, if $\s\colon \GG\to \GG$ is an automorphism, then ${\hs(k\{G\})}\to k\{G\}$ is an isomorphism. Therefore $k\{G\}$ is $\s$-separable and so \ssetale{}.
\end{ex}


We proceed on our path to define the $\s$-identity component.

\begin{defi}[{\cite[Def. 1.17]{TomasivWibmer:StronglyEtaleDifferenceAlgebras}}]
	Let $R$ be a \ks-algebra. The union $\pis(R)=\pis(R|k)$ of all \ssetale{} \ks-subalgebras of $R$ is called the \emph{strong core} of $R$. 
\end{defi}	

The strong core $\pis(R)$ of $R$ is a $\s$-separable \ks-subalgebra of $R$ (\cite[Rem. 1.18]{TomasivWibmer:StronglyEtaleDifferenceAlgebras}). It is however an open problem if $\pis(R)$ is \ssetale{} (equivalently, finite dimensional as a $k$-vector space) if $R$ is finitely $\s$-generated (\cite[Conjecture 1.19]{TomasivWibmer:StronglyEtaleDifferenceAlgebras}). However, if $R$ is a finitely $\s$-generated \ks-Hopf algebra, then $\pis(R)$ is \ssetale{} (\cite[Theorem 3.2]{TomasivWibmer:StronglyEtaleDifferenceAlgebras}).

The strong core has good functorial properties:

\begin{lemma}[{\cite[Lemma 1.25]{TomasivWibmer:StronglyEtaleDifferenceAlgebras}}] \label{lemma: pis and tensor}
	Let $R$ and $S$ be \ks-algebras.  Then $$\pis(R\otimes_k S)=\pis(R)\otimes_k\pis(S).$$
\end{lemma}

\begin{lemma}[{\cite[Lemma 1.24]{TomasivWibmer:StronglyEtaleDifferenceAlgebras}}] \label{lemma: pis and base change}
	Let $R$ be a \ks-algebra and let $k'$ be a $\s$-field extension of $k$.
	Then $$\pis(R\otimes_k k'|k')=\pis(R|k)\otimes_k k'.$$
\end{lemma}

The following proposition allows us to define the $\s$-identity component and the group of $\s$\=/connected components of a $\s$-algebraic group. It is a difference analog of \cite[Theorem~6.7]{Waterhouse:IntrotoAffineGroupSchemes}.

\begin{prop} \label{prop: exists Gcore}
	Let $G$ be a $\s$-algebraic group. Among the morphisms from $G$ to \ssetale{} $\s$-algebraic groups there exists a universal one.
	
	In more detail: There exists a \ssetale{} $\s$-algebraic group $\pis(G)$, together with a morphism $G\to\pis(G)$ of $\s$-algebraic groups such that for every morphism $G\to H$ of $\s$-algebraic groups with $H$ \ssetale{}, there exists a unique morphism $\pis(G)\to H$ making
	$$
	\xymatrix{
		G \ar[rr] \ar[rd] & & \pis(G) \ar@{..>}[ld] \\
		& H &
	}
	$$
	commutative.
\end{prop}
\begin{proof}
	This was already proved in \cite[Prop. 6.5]{BachmayrWibmer:AlgebraicGroupsAsDifferenceGaloisGroupsOfLinearDifferentialEquations}. For the convenience of the reader and because it is instructive, we sketch the proof: Using Lemma \ref{lemma: pis and tensor} one shows that $\pis(k\{G\})$ is a \ks-Hopf subalgebra. By Theorem \ref{theo: ksHopfsubalgebra finitely sgenerated} it is finitely $\s$-generated. Because $\pis(k\{G\})$ is a union of \ks-algebras that are finite dimensional $k$-vector spaces, it follows from the finite $\s$-generation that in fact $\pis(k\{G\})$ is a finite dimensional $k$-vector space. Thus $\pis(k\{G\})$ is \ssetale{} and corresponds to a \ssetale{} $\s$-algebraic group $\pis(G)$. The inclusion $\pis(k\{G\})\subseteq k\{G\}$ correspond to a morphism $G\to \pis(G)$ of $\s$-algebraic groups. 
	
	Let $k\{H\}$ be a \ssetale{} \ks-Hopf algebra and $k\{H\}\to k\{G\}$ a morphism of \ks\=/Hopf algebras. Since quotients of \ssetale{} \ks-algebras are \ssetale{} (\cite[Lemma~1.15]{TomasivWibmer:StronglyEtaleDifferenceAlgebras}), the image of $k\{H\}$ in $k\{G\}$ is \ssetale{}, i.e., contained in $\pis(k\{G\})$. In other words, $k\{H\}\to k\{G\}$, factors uniquely through the inclusion $\pis(k\{G\})\subseteq k\{G\}$. Dualizing yields the required universal property.
%
%
%
\end{proof}

As a corollary to the above proof we obtain:
\begin{cor} \label{cor: pis ssetale}
	Let $G$ be a $\s$-algebraic group. Then $\pis(k\{G\})$ is \ssetale{}. \qed
\end{cor}

The reader may wonder about the significance of the $\s$-separability assumption in Definition~\ref{defi: strongly setale}. Indeed, one may wonder if only working with \ks-subalgebras that are \'{e}tale as $k$\=/algebras (and not necessarily \ssetale{}) leads to similar results. The following example shows that this is not the case.

\begin{ex} \label{ex: core not Hopfalgebra}
	Let $G$ be the $\s$-algebraic subgroup of $\Gm^2$ given by
	$$G(R)=\left\{\begin{pmatrix} g_1 \\ g_2\end{pmatrix}\in\Gm^2(R) \ \Big| \ g_1^2=1,\ \s(g_1)=1,\ g_2^2=1\right\}$$
	for any \ks-algebra $R$. Let $S$ denote the union of all \ks-subalgebras of $k\{G\}$ that are \'{e}tale as $k$\=/algebras. Then $S$ is a \ks-subalgebra of $k\{G\}$. However, as we will show, $S$ is not a \ks-Hopf subalgebra of $k\{G\}$.
	
	Let us assume that the characteristic of $k$ is not equal to $2$. 
	 With $H_1(R)=\{g\in\Gm(R)|\ g^2=1,\ \s(g)=1\}$ and $H_2(R)=\{g\in\Gm(R)|\ g^2=1\}$ we have $G=H_1\times H_2$. Moreover
	$k\{H_1\}=ke_1\oplus ke_2$ for orthogonal idempotent elements $e_1,e_2\in k\{H_1\}$ with $\s(e_1)=1$ and $\s(e_2)=0$. We have $$k\{G\}=k\{H_1\}\otimes_k k\{H_2\}=(e_1\otimes k\{H_2\})\oplus (e_2\otimes k\{H_2\}).$$ For any element $a\in e_2\otimes k\{H_2\}$ we have $\s(a)=0$ and so $k\{a\}=k[a]$ is an \'{e}tale $k$-algebra. This shows that $e_2\otimes k\{H_2\}$ is contained in $S$. In particular, $S$ has infinite dimension as a $k$-vector space.
	
	Suppose $S$ is a \ks-Hopf subalgebra of $k\{G\}$. Then $S$ is finitely $\s$-generated over $k$ by Theorem~\ref{theo: ksHopfsubalgebra finitely sgenerated} and it follows that $S$ is a finite dimensional $k$-vector space; a contradiction. So $S$ is not a \ks-Hopf subalgebra of $k\{G\}$.
\end{ex}

\subsection{The difference identity component and the group of difference connected components}

\begin{defi}
	Let $G$ be a $\s$-algebraic group. The $\s$-algebraic group $\pis(G)$ from Proposition~\ref{prop: exists Gcore} is called the \emph{group of $\s$-connected components}\index{group of $\s$-connected components} of $G$. The kernel $G^\sc$
	of $G\to\pis(G)$ is called the \emph{$\s$-identity component}\index{$\s$-identity component} of $G$.
\end{defi}
So $k\{\pis(G)\}=\pis(k\{G\})$ and $\pis(G)$ is \ssetale{}. Moreover $G\to \pis(G)$ is a quotient map and $G/G^\sc=\pis(G)$.

Our next goal is to connect $\pis(G)$ with the $\s$-topology on $\spec(k\{G\})$. In particular, we will see that $\pis(G)=1$ if and only if $\spec(k\{G\})$ is $\s$-connected. To this end we need some preparatory results.

Recall that a non-zero idempotent element $e$ is \emph{primitive} if it cannot be written as $e=e'+e''$ for non-zero orthogonal idempotent elements $e'$ and $e''$. In an \'{e}tale $k$-algebra the set $\{e_1,\ldots,e_n\}$ of primitive idempotent elements is finite. Moreover, the $e_i$'s are orthogonal and $e_1+\ldots+e_n=1$.

\begin{lemma}[{\cite[Lemma 1.11]{TomasivWibmer:StronglyEtaleDifferenceAlgebras}}] \label{lemma: structure of sfinite ksalgebras}
	Let $R$ be a \ssetale{} $k$-$\s$-algebra. Then $\s$ induces a bijection on the set of primitive idempotent elements of $R$. 
\end{lemma}

\begin{lemma} \label{lemma: idempotents in ssetale salgebra}
	Let $R$ be a strongly $\s$-\'{e}tale \ks-algebra and let $d_1,\ldots,d_m\in R$ denote the primitive idempotent elements of $R$. Let $A_1\uplus\cdots\uplus A_n=\{1,\ldots,m\}$ be the partition of $\{1,\ldots,m\}$ corresponding to the cycle decomposition of the permutation $\tau$ of the $d_i$'s induced by $\s$ (cf. Lemma \ref{lemma: structure of sfinite ksalgebras}). For $i=1,\ldots,n$ let $e_i=\sum_{j\in A_i} d_j$.
	
	Then $e_1,\ldots,e_n$ are constant orthogonal idempotent elements with $e_1+\cdots+e_n=1$ such that no $e_i$ can be written as a sum of two non-trivial constant orthogonal idempotent elements.
	
\end{lemma}
\begin{proof}
	The $e_i$'s are constant because the $A_i$'s are the orbits of $\tau$. Since the $d_j$'s are orthogonal and the $A_i$'s are disjoint, also the $e_i$'s are orthogonal. We have $e_1+\cdots+e_n=d_1+\cdots+d_m=1$. Any constant idempotent element of $R$ is a sum of some $e_i$'s, so no $e_i$ can be written as a sum of two non-trivial constant orthogonal idempotent elements.
\end{proof}

\begin{prop} \label{prop: pis and scomponents}
	Let $R$ be a \ks-algebra such that $\pis(R)$ is strongly $\s$-\'{e}tale. Let $e_1,\ldots,e_n\in \pis(R)$ be as in Lemma \ref{lemma: idempotents in ssetale salgebra}. Then the $\s$-connected components of $\spec(R)$ are $\VV(1-e_1),\ldots,\VV(1-e_n)$. Moreover $\VV(1-e_i)$ is isomorphic to $\spec(e_iR)$ as a $\s$-topological space.
\end{prop}
\begin{proof}
	By \cite[Lemma 1.12]{TomasivWibmer:StronglyEtaleDifferenceAlgebras} all the constant idempotent elements of $R$ belong to $\pis(R)$. So by Lemma \ref{lemma: idempotents in ssetale salgebra} the constant idempotent elements $e_1,\ldots,e_n$ satisfy the condition of Proposition~\ref{prop: sconnected components of spec(R)}.
\end{proof}

The \emph{topological space of a $\s$-algebraic group} $G$ is $\spec(k\{G\})$ (equipped with the Zariski topology). As in Section \ref{subsec: the difference topology} above, we consider $\spec(k\{G\})$ as a $\s$-topological space. 

\begin{theo}  \label{theo: finitely many scomponents}
	Let $G$ be a $\s$-algebraic group. Then the topological space of $G$ has only finitely many $\s$-connected components. Moreover, the topological space of $G^\sc$ is isomorphic (as a $\s$\=/topological space) to the $\s$-connected component of the topological space of $G$ that contains the identity, i.e., the kernel $\m_G$ of the counit $\varepsilon\colon k\{G\}\to k$. 
\end{theo}
\begin{proof}
	As $\pis(k\{G\})$ is \ssetale{} (Corollary \ref{cor: pis ssetale}), the first statement follows immediately from Theorem \ref{prop: pis and scomponents}. Let $d_1,\ldots,d_m\in\pis(k\{G\})$ and $e_1,\ldots,e_n\in\pis(k\{G\})$ be as in Lemma~\ref{lemma: idempotents in ssetale salgebra}. The counit $\varepsilon\colon \pis(k\{G\})\to k$ maps precisely one $d_i$ to $1\in k$ and all other $d_i$'s to $0$. We may assume that $\varepsilon(d_1)=1$. By Lemma \ref{lemma: structure of sfinite ksalgebras} we have $\s(d_1)\in\{d_1,\ldots,d_m\}$. As $\varepsilon(\s(d_1))=\s(\varepsilon(d_1))=\s(1)=1$, this shows that $\s(d_1)=d_1$. So $e_1=d_1$. The kernel of $\varepsilon$ on $\pis(k\{G\})$ is the ideal generated by $1-e_1$. So, by the definition of $G^\sc$, we have $\I(G^\sc)=(1-e_1)\subset k\{G\}$. Therefore $k\{G^\sc\}=k\{G\}/(1-e_1)$ and consequently $\spec(k\{G^\sc\})$ is isomorphic to $\VV(1-e_1)\subseteq\spec(k\{G\})$.
	We know from Proposition \ref{prop: pis and scomponents} that $\VV(1-e_1)$ is a $\s$-connected component and clearly $\m_G\in\VV(1-e_1)$.
\end{proof}

There does not seem to be an easy formula for the number of $\s$-connected components. Indeed, the following example illustrates that the number of $\s$-connected components does not only depend on the underlying field of the base difference field, but it also depends on the endomorphism $\s\colon k\to k$. 

\begin{ex} \label{ex: sconnected3}
	Let $k=\mathbb{Q}$ (considered as a constant $\s$-field). Let $G$ be the $\s$-closed subgroup of $\Gm$ given by
	$$G(R)=\{g\in R^\times|\ g^3=1,\ \s(g)=g\}$$
	for any \ks-algebra $R$. Then $\spec(k\{G\})$ consists of two elements and $\s$ is the identity map on $\spec(k\{G\})$. Thus $G$ has two $\s$-connected components.
	
	Now assume $k$ is a $\s$-field of characteristic zero that contains the two non-trivial third roots of unity $a_1$ and $a_2$. Then $\spec(k\{G\})$ consists of three elements. The endomorphism $\s$ either permutes or fixes $a_1$ and $a_2$. If $\s$ fixes $a_1$ and $a_2$, then $\s$ is the identity map on $\spec(k\{G\})$ and $G$ has three $\s$-connected components. If $\s$ permutes $a_1$ and $a_2$, then $G$ has two $\s$-connected components. In particular, in this case the number of $\s$-connected components of $G$ is strictly smaller than the vector space dimension of $\pis(k\{G\})=k\{G\}$.
\end{ex}

We next characterize $\s$-connected $\s$-algebraic groups.
\begin{lemma} \label{lemma: characterize sconnected}
	For a $\s$-algebraic group $G$, the following statements are equivalent:
	\begin{enumerate}
		\item The topological space of $G$ is $\s$-connected.
		\item $G=G^\sc$.
		\item $\pis(G)=1$.
	\end{enumerate}
\end{lemma}
\begin{proof}
	The equivalence of (ii) and (iii) is tautological. The implication (ii)$\Rightarrow$(i) follows from the fact that the topological space of $G^\sc$ is $\s$-connected by Theorem \ref{theo: finitely many scomponents}
	
	Finally, let us show that (i) implies (iii).  Let $d_1,\ldots,d_m\in\pis(k\{G\})$ and $e_1,\ldots,e_n\in\pis(k\{G\})$ be as in Lemma \ref{lemma: idempotents in ssetale salgebra}. Then
	$$\pis(k\{G\})=e_1\pis(k\{G\})\times\cdots\times e_n\pis(k\{G\})$$
	and as in the proof of Theorem \ref{theo: finitely many scomponents} we have $e_1=d_1$. So $e_1\pis(k\{G\})=d_1\pis(k\{G\})$ is a finite separable field extension of $k$ (because $\pis(k\{G\})$ is an \'{e}tale $k$-algebra). But the counit $\varepsilon$ identifies $e_1\pis(k\{G\})=\pis(k\{G\})/(1-e_1)$ with $k$. So $e_1\pis(k\{G\})=e_1k\simeq k$. We know from Proposition~\ref{prop: pis and scomponents} that $G$ has $n$ $\s$-connected components. By assumption $n=1$. So $\pis(k\{G\})=k$.
\end{proof}

\begin{defi} \label{defi: sconnected}
	A $\s$-algebraic group satisfying the equivalent conditions of Lemma~\ref{lemma: characterize sconnected} is called \emph{$\s$-connected}.\index{$\s$-connected $\s$-algebraic group}
\end{defi}

As an immediate corollary to Theorem \ref{theo: finitely many scomponents} we obtain:

\begin{cor} \label{cor: Gsc is sconnected}
	Let $G$ be a $\s$-algebraic group. Then $G^\sc$ is $\s$-connected. \qed
\end{cor}

\begin{ex} \label{ex: algebraic group sconnected}
	Let $\G$ be an algebraic group. Then $[\s]_k\G$ is a $\s$-connected $\s$-algebraic group by Corollary \ref{cor: specsR sconnected}.
\end{ex}

\begin{ex}
	The $\s$-algebraic group $G$ from Example \ref{ex: core not Hopfalgebra} is $\s$-connected.
	 With the notation of Example \ref{ex: core not Hopfalgebra}, we clearly have $\pis(k\{H_1\})=k$. Note that $H_2$ is the algebraic group of order two, considered as a $\s$-algebraic group. So it follows from Example \ref{ex: algebraic group sconnected} that $\pis(k\{H_2\})=k$. Finally, using Lemma \ref{lemma: pis and tensor} we find
	$$\pis(k\{G\})=\pis(k\{H_1\}\otimes_k k\{H_2\})=\pis(k\{H_1\})\otimes_k\pis(k\{H_2\})=k.$$
\end{ex}

\begin{ex} \label{ex: sconnected1}
	Let $n\geq 2$ be an integer and let $G$ be the $\s$-algebraic group given by
	$$G(R)=\{g\in R^\times|\ g^n=1,\ \s(g)=1\}\leq\Gm(R)$$
	for any \ks-algebra $R$. We claim that $G$ is $\s$-connected. Let $x\in k\{G\}$ denote the image of the coordinate function on $\Gm$. Then $k\{G\}=k\{x\}=k[x]$ with $\s(x)=1$ and $x^n=1$. Clearly, $\s(x-1)=0$. Since $\m_G=(x-1)$ is a maximal ideal of $k\{G\}$, the kernel of $\s$ on $k\{G\}$ is $\m_G$. Therefore $\s^{-1}(\p)=\m_G$ for every prime ideal $\p$ of $k\{G\}$. Now $\spec(k\{G\})$ is a discrete topological space. Suppose $X_1$ and $X_2$ are non-empty disjoint $\s$-invariant subsets of $\spec(k\{G\})$. A point $x_1$ from $X_1$ is mapped onto $\m_G$ under $\s$, similarly for a point $x_2$ from $X_2$. So $\m_G$ lies in the intersection of $X_1$ and $X_2$; a contradiction. Thus $G$ is $\s$-connected. 
\end{ex}

An example of a $\s$-algebraic group that is not $\s$-connected can be deduced from Example \ref{ex: ssetale if and only if}:
\begin{ex}
 Let $G$ be the $\s$-algebraic group from Example \ref{ex: ssetale if and only if} with $\alpha=1$. Then $\pis(G)=G$ and so $G^\sc=1$. In particular, $G$ is not $\s$-connected.
\end{ex}

The formation of $\pis(G)$ and $G^\sc$ is compatible with base extension:

\begin{prop} \label{prop:  pis and base change}
	Let $G$ be a $\s$-algebraic group and $k'$ a $\s$-field extension of $k$. Then
	$$\pis(G_{k'})=\pis(G)_{k'} \quad \text{ and } \quad (G^\sc)_{k'}=(G_{k'})^\sc.$$
\end{prop}
\begin{proof}
	This is clear from Lemma \ref{lemma: pis and base change}.
\end{proof}

A connected algebraic group is geometrically connected (\cite[Prop. 1.34]{Milne:AlgebraicGroupsTheTheoryofGroupSchemesOfFiniteTypeOverAField}). From Proposition \ref{prop:  pis and base change} we obtain a similar result in our setting:
\begin{cor}
	Let $G$ be a $\s$-connected $\s$-algebraic group. Then $G_{k'}$ is $\s$-connected for every $\s$-field extension $k'$ of $k$. \qed
\end{cor}

Recall (\cite[Def. 6.15]{Wibmer:AlmostSimpleAffineDifferenceAlgebraicGroups}) that a $\s$-closed subgroup $H$ of a $\s$-algebraic group $G$ is \emph{characteristic} if for every \ks-algebra $R$, every automorphism of $G_R$ induces an automorphism of $H_R$.

\begin{theo} \label{theo: Gsc is characteristic}
	Let $G$ be a $\s$-algebraic group. Then $G^\sc$ is a characteristic subgroup of $G$.
\end{theo}
\begin{proof}
	By Lemma \cite[Lemma 6.16]{Wibmer:AlmostSimpleAffineDifferenceAlgebraicGroups} it suffices to show that for every $k$-$\s$-algebra $R$, every automorphism $\psi$ of the $R$-$\s$-Hopf algebra $k\{G\}\otimes_k R$ maps $\pis(k\{G\})\otimes_k R$ into $\pis(k\{G\})\otimes_k R$. Since $\psi$ is an automorphism of the $k$-$\s$-algebra $k\{G\}\otimes_k R$, we have $\psi(\pis(k\{G\}\otimes_k R))\subseteq \pis(k\{G\}\otimes_k R)$. Using Lemma \ref{lemma: pis and tensor}, we obtain 
	\begin{align*}
	\psi(\pis(k\{G\})\otimes 1) & \subseteq \psi(\pis(k\{G\}\otimes_k R))\subseteq \pis(k\{G\}\otimes_k R)= \\
	& =\pis(k\{G\})\otimes_k\pis(R)\subseteq \pis(k\{G\})\otimes_k R.
	\end{align*}
	Thus  $\psi(\pis(k\{G\})\otimes_k R)\subseteq\pis(k\{G\})\otimes_k R$ as required.
\end{proof}

The three basic numerical invariants for $\s$-algebraic groups do not change when passing to the $\s$-identity component.

\begin{prop} \label{prop: invariants for Gsc}
	Let $G$ be a $\s$-algebraic group. Then $\sdim(G^\sc)=\sdim(G)$, $\ord(G^\sc)=\ord(G)$ and $\ld(G^\sc)=\ld(G)$.
\end{prop}
\begin{proof}
	By Proposition \ref{prop: invariants and quotients} it suffices to show that $\sdim(\pis(G))=0$, $\ord(\pis(G))=0$ and $\ld(\pis(G))=1$. Clearly a \ssetale{} $\s$-algebraic group has $\s$-dimension zero and order zero. Furthermore, a \ssetale{} $\s$-algebraic group has limit degree one by Lemma \ref{lemma: ld=1}.
\end{proof}

The following two propositions are useful for determining the $\s$-identity component in concrete examples. Proposition \ref{prop: sconnected component unique} is a difference analog of \cite[Prop. 5.58]{Milne:AlgebraicGroupsTheTheoryofGroupSchemesOfFiniteTypeOverAField}, while Proposition \ref{prop: sconnected exact sequence} is a difference analog of \cite[Prop. 5.59]{Milne:AlgebraicGroupsTheTheoryofGroupSchemesOfFiniteTypeOverAField}.

\begin{prop} \label{prop: sconnected component unique}
	Let $G$ be a $\s$-algebraic group. If $N$ is a normal $\s$-connected $\s$-closed subgroup of $G$ such that $G/N$ is \ssetale{}, then $N=G^{\sc}$. 
\end{prop}
\begin{proof}
	As $G/N$ is \ssetale{}, we obtain from Proposition \ref{prop: exists Gcore} a commutative diagram	
	$$
	\xymatrix{
		G \ar[rr] \ar[rd] & & \pis(G) \ar@{..>}[ld] \\
		& G/N &
	}
	$$
	So $G^\sc\leq N$ and the kernel of $N\to \pis(G)$ is $G^\sc$. But by Proposition \ref{prop: exists Gcore} again, $N\to \pis(G)$ factors through $\pis(N)=1$. So the kernel of $N\to \pis(G)$ is $N$. So $N=G^\sc$.
\end{proof}

Alternatively, Proposition \ref{prop: sconnected component unique} can be restated as: For any $\s$-algebraic group $G$, there exists a unique (up to isomorphism) exact sequence
$$1\to G^\sc\to G\to \pis(G)\to 1$$ with $G^\sc$ $\s$-connected and $\pis(G)$ \ssetale{}. The following example shows how Proposition~\ref{prop: sconnected component unique} can be used to determine the $\s$-identity component of a $\s$-algebraic group.

\begin{ex}
	Let $G$ be the $\s$-algebraic group given by
	$$G(R)=\{g\in R^\times|\ g^4=1,\ g^2\s(g)^2=1\}\leq\Gm(R)$$
	for any \ks-algebra $R$. Let us assume that the characteristic of $k$ is not equal to two. We will show that $G$ is not $\s$-connected. Indeed, $G$ has two $\s$-connected components and $G^{\sc}$ is given by $G^{\sc}(R)=\{g\in R^\times|\ g^2=1\}$.
	
	Let $H$ be the $\s$-algebraic group given by $H(R)=\{g\in R^\times|\ g^2=1,\ \s(g)=g\}\leq\Gm(R)$ for any \ks-algebra $R$.
	It follows from $g^4=1$ and $g^2\s(g)^2=1$ that $\s(g^2)=g^2$.
	Therefore we can define a morphism of $\s$-algebraic groups
	$$\f\colon G\to H,\ g\mapsto g^2.$$ The dual morphism $\f^*\colon k\{H\}\to k\{G\}$ is injective, so $\f$ is a quotient map. The kernel $N$ of $\f$ is given by $N(R)=\{g\in R^\times|\ g^2=1\}$. It follows from Example \ref{ex: algebraic group sconnected} that $N$ is $\s$-connected. Moreover, by Example \ref{ex: ssetale if and only if} the $\s$-algebraic group $H$ is \ssetale{}. Therefore it follows from Proposition \ref{prop: sconnected component unique} that $G^\sc=N$ and $\pis(G)=H$.
	As $\pis(k\{G\})=k\{H\}=k\times k$, with the two non-trivial idempotent elements being constant, it follows from Proposition \ref{prop: pis and scomponents} that $G$ has two $\s$-connected components.
\end{ex}

See \cite[Example 6.8]{BachmayrWibmer:AlgebraicGroupsAsDifferenceGaloisGroupsOfLinearDifferentialEquations}
for another example computation of $\pis(G)$ and $G^\sc$ that uses Proposition \ref{prop: sconnected component unique}.

\begin{prop} \label{prop: sconnected exact sequence}
	Let $1\to N\to G\to H\to 1$ be an exact sequence of $\s$-algebraic groups. If $N$ and $H$ are $\s$-connected, then $G$ is $\s$-connected. Moreover, if $G$ is $\s$-connected, then $H$ is $\s$-connected.
\end{prop}
\begin{proof}
	By Proposition \ref{prop: exists Gcore}, the morphism $N\to G\to \pis(G)$ factors through the morphism $N\to\pis(N)=1$. So $N$ is contained in $G^\sc$, i.e., in the kernel of $G\to \pis(G)$. We therefore have an induced quotient map $G/N\to\pis(G)$. But $G/N\simeq H$ is $\s$-connected, so $G/N\to\pis(G)$ factors through $\pis(G/N)=1$. Thus $\pis(G)=1$.
	
	For a quotient map $G\to H$, we have an inclusion $k\{H\}\subseteq k\{G\}$. So, clearly, $\pis(k\{H\})=k$ if $\pis(k\{G\})=k$.	
\end{proof}

The following example illustrates how Proposition \ref{prop: sconnected exact sequence} can be used to show that a $\s$-algebraic group is $\s$-connected. See Examples \ref{ex: Kitchens} and \ref{ex: babbitt long} for more example applications of Proposition~\ref{prop: sconnected exact sequence}.
\begin{ex} \label{ex: easy Babbitt is sconnected}
	Let $G$ be the $\s$-closed subgroup of $\Gm$ given by
	$$G(R)=\{g\in R^\times|\ g^4=1,\ \s(g)^2=1\}\leq\Gm(R)$$
	for any \ks-algebra $R$. We will show that $G$ is $\s$-connected.
	Let $H$ be the $\s$-algebraic group given by $H(R)=\{g\in R^\times|\ g^2=1\}$ and let $\f\colon G\to H$ be the morphism given by
	$$\f_R\colon G(R)\to H(R),\ g\mapsto \s(g).$$ Then the dual map $\f^*\colon k\{H\}\to k\{G\}$ is injective and so $\f$ is a quotient map. The kernel $N$ of $\f$ is given by $N(R)=\{g\in R^\times|\ g^4=1,\ \s(g)=1\}$ for any \ks-algebra $R$. We have an exact sequence $1\to N\to G\to H\to 1$. Note that $H=[\s]_k\H$ where $\H$ is the algebraic group given by $\H(T)=\{g\in T^\times|\ g^2=1\}$ for any $k$-algebra $T$. So it follows from Example \ref{ex: algebraic group sconnected} that $H$ is $\s$-connected. Because $N$ is $\s$-connected by Example \ref{ex: sconnected1}, Proposition \ref{prop: sconnected exact sequence} shows that $G$ is $\s$-connected.
\end{ex}

We conclude this section with one more example computation.

\begin{ex} \label{ex: finite sconnected iff}
	Let $\GG$ be a finite group with an endomorphism $\s\colon \GG\to \GG$ and let $G$ be the $\s$\=/algebraic group associated to these data as in Example \ref{ex: finite constant}. We will show that $G$ is $\s$-connected if and only if some power of $\s\colon \GG\to \GG$ is the trivial endomorphism ($g\mapsto 1$) of $\GG$.

	Let us start by determining the topological space of $G^\sc$. Note that $\spec(k\{G\})$ is naturally in bijection with $\GG$. 
	 We have a commutative diagram 
	$$
	\xymatrix{
		\spec(k\{G\}) \ar_\s[d] \ar^-\simeq[r] & \GG \ar^\s[d] \\
		\spec(k\{G\}) \ar^-\simeq[r] & \GG
	}
	$$
	
	Let $\NN=\{g\in\GG|\ \exists \ n\geq 1 :\ \s^n(g)=1\}$. Then $\NN$ is a normal subgroup of $\GG$ that is invariant under $\s$. Because any element of $\NN$ will eventually map to $1$ under iterations of $\s$, we see that $\NN$ cannot be written as a disjoint union of $\s$-invariant sets. As also the complement of $\NN$ is stable under $\s$, we see that $\NN$ corresponds to the $\s$-connected component of $\spec(k\{G\})$ that contains $\m_G$ under the bijection $\GG\simeq\spec(k\{G\})$.
	In particular, $G$ is $\s$-connected if and only if $\GG=\NN$, i.e., some power of $\s$ is the trivial endomorphism of $\GG$.

	The cosets $g\NN$ of $\NN$ in $\GG$ need not correspond to $\s$-connected components of $G$ because they may not be stable under $\s$. Note that the induced map $\s\colon\GG/\NN\to\GG/\NN$ is an automorphism (because it is injective). The $\s$-connected components of $G$ are in one-to-one correspondence with the orbits of $\s$ on $\GG/\NN$.
\end{ex}

\section{Basic properties of \'{e}tale difference algebraic groups} \label{sec: basic properties of etale difference algebraic groups}

In this short section we recall the definition of \'{e}tale difference algebraic groups and establish some first properties.


Recall that a $k$-algebra $T$ is \emph{ind-\'{e}tale} if it is a union of \'{e}tale $k$-subalgebras. Equivalently, $T$ is integral over $k$ and a separable $k$-algebra. Yet another characterization of ind-\'{e}tale $k$-algebras is that every element satisfies a separable polynomial over $k$.
Following \cite[Def. 6.1]{Wibmer:AlmostSimpleAffineDifferenceAlgebraicGroups} we make the following definition.

\begin{defi}
	A \ks-algebra is \emph{$\s$-\'{e}tale} if it is finitely $\s$-generated over $k$ and ind-\'{e}tale as a $k$-algebra. A $\s$-algebraic group $G$ is $\s$-\'{e}tale if $k\{G\}$ is a $\s$-\'{e}tale \ks-algebra.
\end{defi}


Of course \ssetale{} $\s$-algebraic groups are $\s$-\'{e}tale. Let us see some examples of $\s$-\'{e}tale $\s$-algebraic groups.

\begin{ex} \label{ex: finite is setale}
	Let $\GG$ be a finite group and $\s\colon \GG\to \GG$ a group endomorphism. Then the $\s$\=/algebraic group $G$ constructed from these data as in Example \ref{ex: finite constant} is a $\s$-\'{e}tale $\s$-algebraic group. This is clear because $k\{G\}$ is a finite direct product of copies of $k$ and therefore \'{e}tale.
\end{ex}

\begin{ex} \label{ex: roots of unity}
	Let $n\geq 2$ be an integer. If the characteristic of $k$ does not divide $n$, the $\s$-algebraic group $G$ given by
	$$G(R)=\{g\in R^\times |\ g^n=1\}\leq\Gm(R)$$
	for any \ks-algebra $R$ is $\s$-\'{e}tale. To verify that $k\{G\}=k\{y\}/[y^n-1]$ is $\s$-\'{e}tale it suffices to observe that $k[y]/(y^n-1)$ is an \'{e}tale $k$-algebra because of our assumption on the characteristic.
\end{ex}

\begin{ex} \label{ex: benign}
	Let $\G$ be an \'{e}tale algebraic group. Then $G=[\s]_k\G$ is a $\s$-\'{e}tale $\s$-algebraic group with $\ld(G)=|\G|$.
	
	As $k[\G]$ is an \'{e}tale $k$-algebra. Also ${\hsi(k[\G])}$ and $k[\G[i]]=k[\G]\otimes_k\cdots\otimes_k {\hsi(k[\G])}$ are \'{e}tale $k$\=/algebras for every $i\in\nn$. Therefore $k\{G\}=\cup k[\G[i]]$ is $\s$-\'{e}tale. The statement that $\ld(G)=|\G|$ is a special case of \cite[Example 5.5]{Wibmer:FinitenessPropertiesOfAffineDifferenceAlgebraicGroups}.
\end{ex}

Note that Example \ref{ex: roots of unity} is a special case of Example \ref{ex: benign} with $\G$ the algebraic group of $n$-th roots of unity, i.e., $\G(T)=\{g\in T^\times|\ g^n=1\}$ for any $k$-algebra $T$. 

One of the main results of this article (Theorem \ref{theo: Babbitt}) shows that any $\s$-\'{e}tale $\s$-algebraic group can be build up (in a rather precise way) from the $\s$-algebraic groups in Example~\ref{ex: benign} and two finite $\s$-\'{e}tale $\s$-algebraic groups. (Here a $\s$-algebraic group $G$ is called finite if $k\{G\}$ is a finite dimensional $k$-vector space.)

\begin{defi}
	A $\s$-algebraic group is \emph{benign} if it is isomorphic to a $\s$-algebraic group of the form $[\s]_k\G$ for an \'{e}tale algebraic group $\G$.
\end{defi}

The usage of the word \emph{benign} originates from \cite[Def. 5.4.7]{Levin}, where it is used to describe an extension $L/K$ of $\s$-fields such $L=[\s]_KM$ for a finite Galois extension $M$ of $K$.

Clearly, a $\s$-\'{e}tale $\s$-algebraic group has $\s$-dimension zero and order zero. However, as seen in Example \ref{ex: benign}, the limit degree may be strictly larger than one.

%
%
%
%


\begin{lemma} \label{lemma: subgroups and quotients of setale are setale}
	Quotients and $\s$-closed subgroups of $\s$-\'{e}tale $\s$-algebraic groups are $\s$-\'{e}tale.
\end{lemma}
\begin{proof}
	Let $G$ be a $\s$-\'{e}tale $\s$-algebraic group. Quotients of $G$ correspond to \ks-Hopf subalgebras of $k\{G\}$ (Corollary \ref{cor: quotients correspond to ksHopfsubalgebras})
	and $\s$-closed subgroups of $G$ correspond to quotients of $k\{G\}$. 
	Thus the claim follows from the fact the subalgebras and quotients of ind-\'{e}tale algebras are ind-\'{e}tale.
%
%
\end{proof}

\begin{lemma} \label{lemma: etale Zariski closures}
	Let $G$ be a $\s$-closed subgroup of an algebraic group $\G$. For $i\geq 0$ let $G[i]$ denote the $i$-th order
	Zariski closure of $G$ in $\G$. Then the following statements are equivalent:
	\begin{enumerate}
		\item $G$ is $\s$-\'{e}tale.
		\item $G[0]$ is an \'{e}tale algebraic group.
		\item $G[i]$ is an \'{e}tale algebraic group for every $i\geq 0$.
	\end{enumerate}
\end{lemma}
\begin{proof}
	Because $k\{G\}$ is the union of the $k[G[i]]$'s (i) and (iii) are equivalent. If $k[G[0]]$ is an \'{e}tale $k$-algebra, all the $k[G[i]]$'s are \'{e}tale $k$-algebras, because $k[G[i]]$ is a quotient of $$k[G[0]][i]=k[G[0]]\otimes_k{\hs(k[G[0]])}\otimes_k\ldots\otimes_k{\hsi(k[G[0]])}$$ and \'{e}taleness is preserved under base change, tensor products and quotients.
\end{proof}

A $\s$-closed subgroup of an \'{e}tale algebraic group is $\s$-\'{e}tale (Lemma \ref{lemma: subgroups and quotients of setale are setale}). Conversely, if $G$ is an \'{e}tale algebraic group, we can embed $G$ as $\s$-closed subgroup into an \'{e}tale algebraic group. For example, we may first embed $G$ into some algebraic group $\G$ and then consider the Zariski closure of $G$ in $\G$. Thus a $\s$-algebraic group is $\s$-\'{e}tale if and only if it is isomorphic to a Zariski-dense $\s$-closed subgroup of an \'{e}tale algebraic group.

\begin{prop} \label{prop: reduced and order zero implies setale}
	A $\s$-algebraic group $G$ is $\s$-\'{e}tale if and only if it is reduced (i.e., $k\{G\}$ is reduced) and has order zero.
\end{prop}
\begin{proof}
	Clearly a $\s$-\'{e}tale $\s$-algebraic group is reduced and has order zero. Conversely, let $G$ be a reduced $\s$-algebraic group of order zero. Embed $G$ as a $\s$-closed subgroup in some algebraic group $\G$ and let $G[i]$ denote the $i$-th order Zariski closure of $G$ in $\G$. Then each $G[i]$ is a reduced finite algebraic group and therefore \'{e}tale (\cite[Ex. 8, p. 53]{Waterhouse:IntrotoAffineGroupSchemes}). It follows from Lemma \ref{lemma: etale Zariski closures} that $G$ is $\s$-\'{e}tale.
\end{proof}

\begin{cor} \label{cor: ssetal equals order zero in char zero}
	Assume that the base $\s$-field $k$ has characteristic zero. Then a $\s$-algebraic group is $\s$-\'{e}tale if and only if it has order zero.
\end{cor}
\begin{proof}
	A $\s$-algebraic group over a $\s$-field of characteristic zero is reduced (\cite[Theorem~11.4, p. 86]{Waterhouse:IntrotoAffineGroupSchemes}). Thus the claim follows from Proposition \ref{prop: reduced and order zero implies setale}.
\end{proof}

\section{Expansive endomorphisms and \'{e}tale difference algebraic groups}
\label{sec: Expansive endomorphisms and etale difference algebraic groups}

The category of \'{e}tale algebraic groups over $k$ is equivalent to the category of finite groups equipped with a continuous action of the absolute Galois group of $k$ (\cite[Theorem 6.4]{Waterhouse:IntrotoAffineGroupSchemes}). The goal of this section is to provide a difference analog of this statement (Theorem \ref{theo: equivalence}): The category of $\s$-\'{e}tale $\s$-algebraic groups is equivalent to the category of profinite groups equipped with an expansive endomorphism and a compatible action of the absolute Galois group of $k$. 

This theorem follows rather directly from a much more general equivalence of categories proved in \cite{TomasicWibmer:DifferenceGaloisTheoryAndDynamics}. The equivalence in \cite{TomasicWibmer:DifferenceGaloisTheoryAndDynamics} works over an arbitrary $\s$-ring in place of our base $\s$-field $k$ and its proof relies on Janelidze's categorical Galois theory. We present here a comparatively short, self-contained proof of Theorem \ref{theo: equivalence} that avoids Janelidze's categorical Galois theory and mainly relies on the classical equivalence of ind-\'{e}tale $k$-algebras and profinite spaces equipped with a continuous action of the absolute Galois group.

\medskip

We denote with $k_s$ the separable algebraic closure of $k$ and with $\Gal=\gal(k_s/k)$ the absolute Galois group of $k$. Note that it is always possible to extend $\s\colon k\to k$ to an endomorphism $\s\colon k_s\to k_s$. However, such an extension is usually not unique. We fix once and for all such an extension $\s\colon k_s\to k_s$. The following lemma shows that we have a natural action of $\s$ on $\Gal$.

\begin{lemma} \label{lemma: action of s on Gal}
	For every $\tau\in \Gal$ there exists a unique $\s(\tau)\in\Gal$ such that
	$$
	\xymatrix{
		k_s \ar^{\s(\tau)}[r] \ar_\s[d] & k_s \ar^\s[d]\\
		k_s \ar^\tau[r] & k_s	
	}
	$$
	commutes. The map $\s\colon \Gal\to \Gal,\ \tau\mapsto \s(\tau)$ is a continuous morphism of groups.
\end{lemma}
\begin{proof}
	This is points (1) and (2) of \cite[Lemma 6.1]{Wibmer:ExpansiveDynamicsOnProfiniteGroups}.
\end{proof}

In the sequel we will always consider $\Gal$ as equipped with the endomorphism $\s\colon \Gal \to\Gal$ as in Lemma \ref{lemma: action of s on Gal}.

Recall that a continuous map $\s\colon X\to X$ on a metric space $(X,d)$ is called \emph{(forward) expansive} if there exists an $\varepsilon>0$ such that for any $x\neq y$ in $X$ there exists an $n\in\nn$ with $d(\s^n(x),\s^n(y))>\varepsilon$. In the context of a continuous group homomorphism $\s\colon \GG \to \GG$ on a topological group $\GG$, this idea translates to: There exists a neighborhood $\texttt{U}$ of the identity $1\in\GG$ such that for any $x\neq y$ in $\GG$ there exists an $n\in\nn$ with $\s^n(x)\notin \s^n(y)\texttt{U}$, equivalently, $y^{-1}x\notin \s^{-n}(\texttt{U})$. Since the open normal subgroups of a profinite group are a neighborhood basis at $1$ (\cite[Theorem 2.1.3]{RibesZaleskii:ProfiniteGroups}), we can assume that $\texttt{U}$ is an open normal subgroup of $\texttt{G}$, in case $\GG$ is a profinite group. We thus arrive at the following definition. 

\begin{defi} \label{defi: expansive endo}
	An endomorphism $\s\colon \GG\to\GG$ of a profinite group $\GG$ is \emph{expansive} if there exists an open normal subgroup $\NN$ of $\GG$ such that $\bigcap_{i\in\nn}\s^{-i}(\NN)=1$.
\end{defi}

Similarly, an automorphism $\s\colon \GG\to\GG$ of a profinite group $\GG$ is \emph{expansive} if there exists an open normal subgroup $\NN$ of $\GG$ such that $\bigcap_{i\in\Z}\s^{-i}(\NN)=1$. Profinite groups, or more general topological groups, such as totally disconnected locally compact groups, equipped with an expansive automorphism have been studied by various authors. See \cite{Kitchens:ExpansiveDynamicsOnZeroDimensionalGroups}, \cite{Fagnani:SomeResultsOnTheClassificationOfExpansiveAutomorphismsOfCompactAbelianGroups},
\cite{KitchensSchmidt:AutomorphismsOfCompactGroups},  \cite{BoyleSchrauder:ZdgroupshiftsAndBernoulliFactors}, \cite{Willis:TheNubOfAnAutomorphismOfaTotallyDisconnectedLocallyCompactGroup}, \cite{GloecknerRaja:ExpansiveAutomorphismsOfTotallyDisconnectedLocallyCompactGroups}, \cite{Shah:ExpansiveAutomorphismsOnLocallyCompactGroups}, \cite[Chapter 3]{Schmidt:DynamicalSystemsOfAlgebraicOrigin} and the references given there.

While there has been some interest in generalizing results from automorphisms to endomorphisms (\cite{Reid:EndomorphismsOfProfiniteGroups}, \cite{BywatersGloecknerTornier:ContractionGroupsAndPassageToSubgroupsAndQuotientsForEndomorphismsOfTotallyDisconnectedLocallyCompactGroups}, \cite{Willis:TheScaleAndTidySubgroupsForEndomorphismsOfTotallyDisconnectedLocallyCompactGroups}, \cite{GiordanoVirili:TopologicalEntropyInTotallyDisconnectedLocallyCompactGroups}, \cite{Wibmer:ExpansiveDynamicsOnProfiniteGroups}), the literature on expansive endomorphisms of profinite groups is rather scarce. 

There are two basic examples of expansive endomorphisms of profinite groups. 

\begin{ex}
	Let $\GG$ be a finite (discrete) group. Then any endomorphism $\s\colon\GG\to \GG$ is expansive since we can choose $\NN=1$ in Definition \ref{defi: expansive endo}.
\end{ex}

\begin{ex}
	Let $\HH$ be a finite (discrete) group and consider $\GG=\HH^\nn$ as a profinite group under componentwise multiplication and equipped with the product topology. Then the shift map
	$$\s\colon \HH^\nn\to \HH^\nn,\ (h_0,h_1,h_2,\ldots)\mapsto (h_1,h_2,\ldots)$$ is expansive. Indeed, we can choose $\NN=1\times \HH\times\HH\ldots\leq \HH^\nn$.
\end{ex}



\begin{defi}
	A \emph{profinite $\Gal$-$\s$-group} is a profinite group $\GG$ equipped with an endomorphism $\s\colon\GG\to\GG$ and a continuous action $\Gal\times \GG\to \GG$ of $\Gal$ on $\GG$ by group automorphisms such that 
	$\s(\tau(g))=\s(\tau)(\s(g))$ for $\tau\in\Gal$ and $g\in\GG$.
	An  \emph{expansive profinite $\Gal$-$\s$-group} is a profinite $\Gal$-$\s$-group such that $\s\colon\GG\to\GG$ is expansive.
\end{defi}

A morphism $\f\colon\GG\to\HH$ of profinite $\Gal$-$\s$-groups is a morphism of profinite groups that commutes with $\s$ and is $\Gal$-equivariant, i.e., $\f(\tau(g))=\tau(\f(g))$ for $\tau\in\Gal$ and $g\in \GG$.
The main result of this section is the following:

\begin{theo} \label{theo: equivalence}
	The category of $\s$-\'{e}tale $\s$-algebraic groups over $k$ is equivalent to the category of expansive profinite $\Gal$-$\s$-groups.
\end{theo}
The proof of Theorem \ref{theo: equivalence} is given at the end of this section, after some preparatory results are established.

The fact that \'{e}tale algebraic groups over $k$ are equivalent to finite groups with a continuous $\Gal$-action follows from the fact that \'{e}tale algebras over $k$ are anti-equivalent to finite sets with a continuous $\Gal$-action (\cite[Theorem 6.3]{Waterhouse:IntrotoAffineGroupSchemes}). For our proof of Theorem \ref{theo: equivalence} we shall need an ``infinite'' version of this anti-equivalence, i.e., a version that applies to ind-\'{e}tale algebras instead of just \'{e}tale algebras. On the side of the $\Gal$-actions one then has to replace finite sets by profinite spaces.

Recall that a \emph{profinite (topological) space} is a topological space that can be written as a projective limit of finite discrete topological spaces. Equivalently, a topological space is profinite, if it is Hausdorff, compact and totally disconnected (\cite[Theorem 1.1.12]{RibesZaleskii:ProfiniteGroups}).

For an ind-\'{e}tale $k$-algebra $T$, the set $\Hom(T,k_s)$ of all $k$-algebra morphisms of $T$ to $k_s$ is naturally a profinite space: As $T=\varinjlim T_i$ is the directed union of its \'{e}tale $k$-subalgebras $T_i$, we see that $\Hom(T,k_s)=\varprojlim\Hom(T_i,k_s)$ is the projective limit of the finite sets $\Hom(T_i,k_s)$.
More explicitly, a basis for the topology of $\Hom(T,k_s)$ is given by the open subsets
$$U(a_1,\ldots,a_n,b_1,\ldots,b_n)=\{\psi\in\Hom(T,k_s)|\ \psi(a_1)=b_1,\ldots,\psi(a_n)=b_n\},$$
where $a_1,\ldots,a_n\in T$ and $b_1,\ldots,b_n\in k_s$.

The action $\Gal\times \Hom(T,k_s)\to \Hom(T,k_s),\ (\tau, \psi)\mapsto \tau\circ \psi$ is continuous (\cite[Lemma~3.5.4]{BorceuxJanelidze:GaloisTheories}). A \emph{profinite $\Gal$-space} is a profinite space together with a continous $\Gal$-action. (According to \cite[Lemma 5.6.4 (a)]{RibesZaleskii:ProfiniteGroups}, this definition is equivalent to Definition 3.5.1 in \cite{BorceuxJanelidze:GaloisTheories}, where a profinite $\Gal$-space is defined to be the projective limit of finite discrete $\Gal$-spaces.) A morphism of profinite $\Gal$-spaces is a continuous $\Gal$-equivariant map. We are now prepared to state the infinitary version of the ``Galois equivalence''.

\begin{theo}[{\cite[Theorem 3.5.8]{BorceuxJanelidze:GaloisTheories}}] \label{theo: classical equivalence}
	The functor $T\rightsquigarrow\Hom(T,k_s)$ defines an anti-equivalence of categories between the category of ind-\'{e}tale $k$-algebras and the category of profinite $\Gal$-spaces. Under this anti-equivalence, surjective morphisms of ind-\'{e}tale $k$-algebras correspond to injective morphisms of profinite $\Gal$-spaces.
\end{theo}
\begin{proof}
	The anti-equivalence is a special case of \cite[Theorem 3.5.8]{BorceuxJanelidze:GaloisTheories}, where an arbitrary Galois extension is allowed in place of $k_s/k$. Note that a $k$-algebra $T$ is split by $k_s$ (in the sense of Definition \cite[Def. 2.3.1]{BorceuxJanelidze:GaloisTheories}) if and only if $T$ is ind-\'{e}tale.
	
	The surjective/injective statement is \cite[Lemma 3.24]{TomasicWibmer:DifferenceGaloisTheoryAndDynamics}.	
%
\end{proof}

We now return to our difference scenario. Our first goal is to add a $\s$ to the anti-equivalence of Theorem \ref{theo: classical equivalence}.

\begin{defi}
	A \emph{profinite $\Gal$-$\s$-space} $\XX$ is a profinite space $\XX$ together with a continuous endomorphism $\s\colon\XX\to \XX$ and a continuous action $\Gal\times \XX\to \XX$ compatible with $\s$ in the sense that
	$$\s(\tau(x))=\s(\tau)(\s(x))$$
	for $\tau\in\Gal$ and $x\in\XX$.	
\end{defi}
A morphism of profinite $\Gal$-$\s$-spaces is a continuous $\Gal$-equivariant map that commutes with the action  of $\s$.

A \ks-algebra is called \emph{ind-\'{e}tale} if it is ind-\'{e}tale as a $k$-algebra. The following two lemmas show that ind-\'{e}tale \ks-algebras give rise to profinite $\Gal$-$\s$-spaces.

\begin{lemma} \label{lemma: prepare for lift}
	Let $T$ be an ind-\'{e}tale $k$-algebra and let $\alpha\colon T\to k_s$ be a morphism of rings such that $\alpha(\lambda a)=\s(\lambda)\alpha(a)$ for $\lambda\in k$ and $a\in T$. Then there exists a unique morphism $\beta\colon T\to k_s$ of $k$-algebras such that 
	$$
	\xymatrix{
	T \ar^\beta[rr] \ar_\alpha[rd] & & k_s \ar^\s[ld]\\
	& k_s &	
	}
	$$
	commutes.
\end{lemma}
\begin{proof}
	The uniqueness of $\beta$ follows from the injectivity of $\s\colon k_s\to k_s$. For the existence of $\beta$ it suffices to show that $\alpha(T)\subseteq \s(k_s)$ (because then, for any $a\in T$, we can define $\beta(a)\in k_s$ to be the unique element of $k_s$ such that $\alpha(a)=\s(\beta(a))$). 
	
	Let $a\in T$. Since $T$ is an ind-\'{e}tale $k$-algebra, $a$ satisfies a monic (separable) polynomial $f\in k[x]$. Write $f=\prod_{i=1}^n(x-a_i)$ with $a_i\in k_s$. Let ${\hs f}\in k[x]$ denote the polynomial obtained from $f$ by applying $\s\colon k\to k$ to the coefficients. Then ${\hs f}=\prod_{i=1}^n(x-\s(a_i))$.
	
	From $f(a)=0$ and $\alpha(\lambda a)=\s(\lambda)\alpha(a)$ for $\lambda\in k$, we obtain ${\hs f}(\alpha(a))=0$. So $\alpha(a)\in \{\s(a_1),\ldots,\s(a_n)\}\subseteq\s(k_s)$.
\end{proof}

Recall (Section \ref{subsec: Difference algebra}) that for a \ks-algebra $R$, we denote with $R^\sharp$ the underlying $k$-algebra.

\begin{lemma}
	Let $R$ be an ind-\'{e}tale $k$-$\s$-algebra. Then, for every morphism $\psi\colon R\to k_s$ of $k$-algebras, there exists a unique morphism $\s(\psi)\colon R\to k_s$ of $k$-algebras such that 
	$$
	\xymatrix{
	R \ar^{\s(\psi)}[r] \ar_\s[d] & k_s \ar^\s[d] \\
	R \ar^\psi[r] & k_s	
	}
	$$
	commutes. Moreover, $\s\colon \Hom(R^\sharp,k_s^\sharp)\to \Hom(R^\sharp,k_s^\sharp),\ \psi\mapsto \s(\psi)$ is continuous and 
	\begin{equation} \label{eq: s action}
		\s(\tau(\psi))=\s(\tau)(\s(\psi))
	\end{equation}
	 for $\tau\in\Gal$, i.e., $\Hom(R^\sharp,k_s^\sharp)$ is a profinite $\Gal$-$\s$-space.
\end{lemma}
\begin{proof}
	The existence and uniqueness of $\s(\psi)$ follows from Lemma \ref{lemma: prepare for lift} applied to $$\alpha\colon R\xrightarrow{\s} R\xrightarrow{\psi}k_s.$$
%
%
%
%
%
%
%
	For $a_1,\ldots,a_n\in R$ and $b_1,\ldots,b_n\in k_s$ we have,
	$$\s^{-1}(U(a_1,\ldots,a_n,b_1,\ldots,b_n))=\{\psi\in\Hom(R^\sharp,k_s^\sharp)|\ \s(\psi)(a_i)=b_i,\ i=1,\ldots,n\}.$$ 
	But
	$$
	\s(\psi)(a_i)=b_i\ \Leftrightarrow \ \s(\s(\psi)(a_i))=\s(b_i)\ \Leftrightarrow \ \psi(\s(a_i))=\s(b_i).
	$$
	Therefore $\s^{-1}(U(a_1,\ldots,a_n,b_1,\ldots,b_n))=U(\s(a_1),\ldots,\s(a_n),\s(b_1),\ldots,\s(b_n))$ is open. Formula (\ref{eq: s action}) follows from the commutative diagram
	$$
	\xymatrix{
	R \ar^-{\s(\psi)}[r] \ar_\s[d]& k_s \ar^\s[d]  \ar^-{\s(\tau)}[r] & k_s \ar^\s[d] \\
	R \ar^\psi[r] & k_s \ar^\tau[r] & k_s	
	}
	$$
\end{proof}	

If $\alpha\colon R\to S$ is a morphism of ind-\'{e}tale \ks-algebras, the commutative diagram
$$
\xymatrix{
R \ar^\alpha[r] \ar_\s[d] & S \ar^-{\s(\psi)}[r] \ar^\s[d] & k_s \ar^\s[d]  \\
R \ar^-\alpha[r] & S \ar^\psi[r] & k_s	
}
$$
shows that the diagram
$$
\xymatrix{
\Hom(S^\sharp,k_s^\sharp) \ar[r] \ar_\s[d] & \Hom(R^\sharp,k_s^\sharp) \ar^\s[d] \\
\Hom(S^\sharp,k_s^\sharp) \ar[r] & \Hom(R^\sharp,k_s^\sharp)
}
$$
commutes. So the induced map $\Hom(S^\sharp,k_s^\sharp)\to \Hom(R^\sharp,k_s^\sharp)$ is a morphism of profinite $\Gal$\=/$\s$\=/spaces. In other words, $R\rightsquigarrow \Hom(R^\sharp,k_s^\sharp)$ is a contravariant functor from the category of ind-\'{e}tale \ks\=/algebras to the category of profinite $\Gal$-$\s$-spaces. To show that this functor defines an anti-equivalence we need some more preparatory results.

For a profinite $\Gal$-space $\XX$, we can twist the action of $\Gal$ on $\XX$ by $\s$ to obtain a new profinite $\Gal$-space ${\hs \XX}$. In detail, ${\hs \XX}=\XX$ as profinite spaces but the action of $\Gal$ on ${\hs \XX}$ is given by $g(x)=\s(g)(x)$ for $x\in{\hs \XX}$ and $g\in\Gal$. Note that for a profinite $\Gal$-$\s$-space $\XX$, the map $\s\colon\XX\to \XX$ can be interpreted as a morphism $\s\colon \XX\to{\hs \XX}$ of profinite $\Gal$-spaces.

\begin{lemma} \label{lemma: twist and profinite}
	Let $T$ be an ind-\'{e}tale $k$-algebra. Then
	$${\hs(\Hom(T,k_s))}\simeq \Hom({\hs T},k_s)$$
	as profinite $\Gal$-spaces.
\end{lemma}
\begin{proof}
Let us first describe the bijection $\eta\colon \Hom(T,k_s)\to\Hom({\hs T},k_s)$. If $\psi\colon T\to k_s$ is a morphism of $k$-algebras, then $\psi'=\eta(\psi)\colon {\hs T}=T\otimes_k k\to k_s,\ a\otimes\lambda\mapsto \s(\psi(a))\lambda$ is a morphism of $k$-algebras. Conversely, if $\psi'\colon {\hs T}\to k_s$ is a morphism of $k$-algebras, then $\alpha\colon T\to {\hs T}\xrightarrow{\psi'} k_s$ (where the first map is $a\mapsto a\otimes 1$) is a morphism of rings satisfying $\alpha(\lambda a)=\s(\lambda)\alpha(a)$ for $\lambda\in k$ and $a\in T$. Thus Lemma \ref{lemma: prepare for lift} yields a (unique) morphism $\psi=\rho(\psi')\colon T\to k_s$ of $k$-algebras such that 
$$
\xymatrix{
T \ar^\psi[r] \ar[d] & k_s \ar^\s[d] \\
{\hs T} \ar^{\psi'}[r] & k_s	
}
$$ 
commutes. The above diagram shows that $\rho$ is the inverse of $\eta$. For $a_1,\ldots,a_n\in T$ and $b_1,\ldots,b_n\in k_s$ we have
$$\rho^{-1}(U(a_1,\ldots,a_n,b_1,\ldots,b_n))=\{\psi'\in\Hom({\hs T},k_s)|\ \rho(\psi')(a_i)=b_i\ i=1,\ldots,n\}.$$
But $$\rho(\psi')(a_i)=b_i \ \Leftrightarrow \ \s(\rho(\psi')(a_i))=\s(b_i)\ \Leftrightarrow \ \psi'(a_i\otimes 1)=\s(b_i).$$ 
So $\rho^{-1}(U(a_1,\ldots,a_n,b_1,\ldots,b_n))=U(a_1\otimes 1,\ldots,a_n\otimes 1,\s(b_1),\ldots,\s(b_n))$ and $\rho$ is continuous. As any continuous bijection between compact Hausdorff spaces is a homeomorphism (\cite[Prop. 13.26]{Sutherland:IntroductionToMetricAndTopologicalSpaces}), we see that $\rho$ and therefore also $\eta$ is a homeomorphism.

For $a\in T$, $\lambda\in k$ and $\psi\in\Hom(T,k_s)$, it follows from the commutative diagram
$$
\xymatrix{
T \ar^\psi[r] \ar[d] & k_s \ar^\s[d]
\ar^-{\s(\tau)}[r] & k_s \ar^\s[d] \\
{\hs T} \ar^-{\eta(\psi)}[r] & k_s \ar^\tau[r] & k_s	
}
$$
that
$$\eta(\s(\tau)(\psi))(a\otimes\lambda)=\s(\s(\tau)(\psi)(a))\lambda=\tau(\eta(\psi))(a\otimes 1)\lambda=\tau(\eta(\psi))(a\otimes \lambda).$$
Thus $\eta(\s(\tau)(\psi))=\tau(\eta(\psi))$ as desired.
\end{proof}	

\begin{lemma} \label{lemma: lower rectangle}
	The isomorphism from Lemma \ref{lemma: twist and profinite} is functorial, i.e., for a morphism $\alpha\colon S\to T$ of ind-\'{e}tale $k$-algebras, we have a commutative diagram 
	\begin{equation} \label{eq: diag twist}
		\xymatrix{
		{\hs(\Hom(T,k_s))} \ar^\simeq[r] \ar[d] & \Hom({\hs T},k_s) \ar[d] \\
		{\hs(\Hom(S,k_s))} \ar^\simeq[r] & \Hom({\hs S},k_s)
	}
	\end{equation}
 in the category of profinite $\Gal$-spaces.
\end{lemma}
\begin{proof}
	Let $\psi\in 	{\hs(\Hom(T,k_s))} $. Both paths in diagram (\ref{eq: diag twist}) yield the element of $\Hom({\hs S},k_s)$ given by ${\hs S}\to k_s,\ s\otimes\lambda\mapsto \s(\psi(\alpha(s)))\lambda$.
\end{proof}

For a \ks-algebra $R$, the map $\overline{\s}\colon {\hs R}\to R,\ a\otimes\lambda\mapsto \s(a)\lambda$ is a morphism of $k$-algebras. Thus, if $R$ is ind-\'{e}tale, we obtain a morphism
$\Hom(R^\sharp,k_s^\sharp)\to \Hom({\hs R}^\sharp,k_s^\sharp)$ of profinite $\Gal$-spaces.

\begin{lemma} \label{lemma: triangle}
	Let $R$ be an ind-\'{e}tale \ks-algebra. Then the composition
	$$\Hom(R^\sharp,k_s^\sharp)\to \Hom({\hs R}^\sharp,k_s^\sharp) \simeq {\hs (\Hom(R^\sharp,k_s^\sharp))}$$
	equals $\s\colon \Hom(R^\sharp,k_s^\sharp)\to \Hom(R^\sharp,k_s^\sharp)$.
\end{lemma}
\begin{proof}
	This follows from the commutative diagram
	$$
	\xymatrix{
	R \ar^-{\s(\psi)}[rr] \ar[d] \ar^\s[rd] & & k_s \ar^\s[d] \\
	{\hs R} \ar^-{\overline{\s}}[r] & R \ar^-\psi[r] & k_s	
	}
	$$
\end{proof}

We are now prepared to prove a $\s$-version of Theorem \ref{theo: classical equivalence}.

\begin{prop} \label{prop: s equivalence}
	The functor $R\rightsquigarrow \Hom(R^\sharp,k_s^\sharp)$ defines an anti-equivalence of categories between the category of ind-\'{e}tale \ks-algebras and the category of profinite $\Gal$-$\s$-spaces.
\end{prop}
\begin{proof}
	We first show that the functor is fully faithful. Let $R$ and $S$ be ind-\'{e}tale \ks-algebras. According to Theorem \ref{theo: classical equivalence} we have a bijection
	$$
	\xi\colon \Hom(R^\sharp,S^\sharp)\simeq \Hom(\Hom(S^\sharp,k_s^\sharp)^\sharp,\Hom(R^\sharp,k_s^\sharp)^\sharp),
	$$
	where $\Hom(S^\sharp,k_s^\sharp)^\sharp$ denotes the profinite $\Gal$-space obtained from the profinite $\Gal$-$\s$-space $\Hom(S^\sharp,k_s^\sharp)$ by forgetting $\s$. It suffices to show that $\xi$ restricts to a bijection between the subsets $\Hom(R,S)\subseteq \Hom(R^\sharp,S^\sharp)$ and $$\Hom(\Hom(S^\sharp,k_s^\sharp),\Hom(R^\sharp,k_s^\sharp))\subseteq \Hom(\Hom(S^\sharp,k_s^\sharp)^\sharp,\Hom(R^\sharp,k_s^\sharp)^\sharp).$$
	In fact, it suffices to show that if $\alpha\colon R\to S$ is a morphism of $k$-algebras, such that $\xi(\alpha)\colon \Hom(S^\sharp,k_s^\sharp)\to\Hom(R^\sharp,k_s^\sharp)$ is a morphism of profinite $\Gal$-$\s$-spaces, then $\alpha$ is a morphism of \ks-algebras. Note that $\alpha\colon R\to S$ is a morphism of \ks-algebras if and only if 
	$$
	\xymatrix{
	{\hs R} \ar^-{{\hs\alpha}}[r] \ar_{\overline{\s}}[d] & {\hs S} \ar^{\overline{\s}}[d] \\
R \ar^\alpha[r] & S	
	}	
	$$
	is a commutative diagram in the category of ind-\'{e}tale $k$-algebras. According to Theorem \ref{theo: classical equivalence} it thus suffices to show that the induced diagram
	\begin{equation} \label{eq: diag for main equ}
		\xymatrix{
		\Hom(S^\sharp,k^\sharp) \ar[r] \ar[d] & \Hom(R^\sharp,k^\sharp) \ar[d] \\
		\Hom({\hs S}^\sharp,k^\sharp) \ar[r] & \Hom({\hs R}^\sharp,k^\sharp)
	}
	\end{equation}
	is a commutative diagram in the category of profinite $\Gal$-spaces. Now (\ref{eq: diag for main equ}) fits into the larger diagram 
	\begin{equation} \label{eq: diag big}
		\xymatrix{
		\Hom(S^\sharp,k^\sharp) \ar@/_3.5pc/_\s[dd] \ar[r] \ar[d] & \Hom(R^\sharp,k^\sharp) \ar[d] \ar@/^3.5pc/^\s[dd] \\
		\Hom({\hs S}^\sharp,k^\sharp) \ar[r] \ar_-\simeq[d]  & \Hom({\hs R}^\sharp,k^\sharp) \ar^-\simeq[d] \\
		{\hs(\Hom(S^\sharp,k^\sharp))} \ar[r] & {\hs(\Hom(R^\sharp,k^\sharp))}
	}
	\end{equation}
	in the category of profinite $\Gal$-spaces. 	
	Note that if $\f\colon \XX\to\mathtt{Y}$ is a morphism of profinite $\Gal$\=/$\s$\=/spaces, then 
	$$
	\xymatrix{
	\XX \ar^\f[r] \ar_\s[d] & \mathtt{Y} \ar^\s[d]\\
	{\hs \XX} \ar^\f[r] &	{\hs \mathtt{Y}}
	}
	$$
	is a commutative diagram in the category of profinite $\Gal$-spaces. Thus the outer rectangle of (\ref{eq: diag big}) commutes. The left and right triangles commute by Lemma \ref{lemma: triangle} and the lower rectangle commutes by Lemma \ref{lemma: lower rectangle}. Thus also the upper rectangle commutes as desired.
	
	It remains to show that the functor $R\rightsquigarrow \Hom(R^\sharp,k_s^\sharp)$ is essentially surjective. Let $\XX$ be a profinite $\Gal$-$\s$-space. From Theorem \ref{theo: classical equivalence} we know that there exists an ind-\'{e}tale $k$-algebra $R$ and an isomorphism $\XX\simeq\Hom(R,k_s^\sharp)$ of profinite $\Gal$-spaces. Again, by Theorem \ref{theo: classical equivalence}, the morphism
	$$\Hom(R,k_s^\sharp)\simeq\XX\xrightarrow{\s}{\hs\XX}\simeq{\hs(\Hom(R,k_s^\sharp))}\simeq \Hom({\hs R},k_s^\sharp)$$ of profinite $\Gal$-spaces is induced by a unique morphism $\overline{\s}\colon {\hs R}\to R$ of $k$-algebras. Then $\s\colon R\to R,\ a\mapsto\overline{\s}(a\otimes 1)$ is a ring endomorphism extending $\s\colon k\to k$, i.e., $R$ is a \ks-algebra.
	 
	 By definition of $\overline{\s}\colon {\hs R}\to R$, the diagram
	 $$
	 \xymatrix{
	 \XX \ar_\simeq[d] \ar^\s[r] & {\hs\XX} \ar^\simeq[d]	\\
	 \Hom(R^\sharp,k_s^\sharp) \ar[rd] & {\hs(\Hom(R^\sharp,k_s^\sharp))} \ar^\simeq[d] \\
	 & \Hom({\hs R}^\sharp,k_s^\sharp)
	 }
	 $$
	commutes. Using Lemma \ref{lemma: triangle} it follows that 
	$$
	\xymatrix{
		\XX \ar_\simeq[d] \ar^\s[r] & {\hs\XX} \ar^\simeq[d]	\\
		\Hom(R^\sharp,k_s) \ar^-\s[r] & {\hs(\Hom(R^\sharp,k_s))} 
	}
	$$
	commutes, i.e., $\XX\simeq\Hom(R^\sharp,k_s^\sharp)$ as profinite $\Gal$-$\s$-spaces.
\end{proof}

Note that the profinite $\Gal$-$\s$-groups are exactly the group objects in the category of profinite $\Gal$-$\s$-spaces. From Proposition \ref{prop: s equivalence} we thus obtain:

\begin{cor} \label{cor:  s equivalence}
	The category of ind-\'{e}tale \ks-Hopf algebras is anti-equivalent to the category of profinite $\Gal$-$\s$-groups. \qed
\end{cor}

\begin{ex} \label{ex: equ for etale}
	Let $\H$ be an \'{e}tale algebraic group. We would like to describe the profinite $\Gal$\=/$\s$\=/group $\GG$ corresponding to the ind-\'{e}tale \ks-Hopf algebra $k\{\H\}$ under the anti-equivalence of Corollary \ref{cor:  s equivalence}. Let $\HH=\Hom(k[\H],k_s)$ be the finite (discrete) group equipped with the continuous $\Gal$-action corresponding to $\H$. Forgetting $\s$ for now, note that $k\{\H\}^\sharp$ is the coproduct of the ${\hsi(k[\H])}$'s and so $\GG^\sharp$ is the product of the $\Hom({\hsi(k[\H])},k_s)$'s. But $\Hom({\hsi(k[\H])},k_s)\simeq {\hsi\HH}$ by Lemma \ref{lemma: twist and profinite}. So $\GG^\sharp=\HH^\nn$ with action of $\Gal$ given by $$\tau(h_0,h_1,h_2,\ldots)=(\tau(h_0),\s(\tau)(h_1),\s^2(\tau)(h_2),\ldots).$$
	Multiplication in $\HH^\nn$ is componentwise and the topology is the product topology.
	
	The morphism $\psi\colon k\{\H\}\to k_s$ of $k$-algebras corresponding to $(\psi_i)_{i\in\nn}\in \HH^\nn$ is determined by $\psi(a_i\otimes\lambda_i)=\s^i(\psi_i(a_i))\lambda_i$ for $a_i\otimes\lambda_i\in k[\H]\otimes_k k={\hsi(k[\H])}$. Thus the morphism $\s(\psi)\colon k\{\H\}\to k_s$ of $k$-algebras such that
	$$
	\xymatrix{
k\{\H\} \ar^-{\s(\psi)}[r] \ar_\s[d] & k_s \ar^\s[d] \\
k\{\H\} \ar^-\psi[r] & k_s		
	}
	$$
	commutes, is given by $\s(\psi)(a_i\otimes\lambda_i)=\s^i(\psi_{i+1}(a_i))\lambda_i$ for $a_i\otimes\lambda_i\in k[\H]\otimes_k k={\hsi(k[\H])}$. This shows that $$\s\colon\HH^\nn\to\HH^\nn,\ (h_0,h_1,h_2,\ldots)\mapsto (h_1,h_2,\ldots)$$ is simply the shift.
\end{ex}

For the proof of Theorem \ref{theo: equivalence} we need one more preparatory result.

\begin{lemma} \label{lemma: s continuous}
	Let $\GG$ be an expansive profinite $\Gal$-$\s$-group. Then there exists an open normal subgroup $\mathtt{U}$ of $\GG$ such that $\bigcap_{i\in\nn}\s^{-i}(\mathtt{U})=1$ and $\tau(\mathtt{U})=\mathtt{U}$ for all $\tau\in\Gal$.
\end{lemma}
\begin{proof}
	This proof has some similarity with the proof of \cite[Lemma 5.6.4 (a)]{ RibesZaleskii:ProfiniteGroups}.
	Since $\s\colon\GG\to\GG$ is expansive, there exists an open normal subgroup $\NN$ of $\GG$ such that $\bigcap_{i\in\nn}\s^{-i}(\mathtt{N})=1$. Set $\mathtt{U}=\bigcap_{\tau\in\Gal}\tau(\NN)$. Then $\mathtt{U}$ is a normal subgroup of $\GG$, $\tau(\mathtt{U})=\mathtt{U}$ for all $\tau\in\Gal$ and  $\bigcap_{i\in\nn}\s^{-i}(\mathtt{U})=1$ because $\mathtt{U}\subseteq\NN$. 
	
	It remains to see that $\mathtt{U}$ is open. Let $g\in \mathtt{U}$. Then $\tau(g)\in\mathtt{U}\subseteq\NN$ for any $\tau\in\Gal$. Since $\NN\subseteq \GG$ is open and the action $\Gal\times \GG\to \GG$ is continuous, there exist open neighborhoods $\tau\in V_\tau\subseteq\Gal$ and $g\in W_\tau\subseteq \GG$ such that $\tau'(g')\in \NN$ for all $\tau'\in V_\tau$ and $g'\in W_\tau$. The $V_\tau$'s are an open cover of $\Gal$. Therefore, there exist $\tau_1,\ldots,\tau_n\in\Gal$ such that $\Gal=V_{\tau_1}\cup\ldots\cup V_{\tau_n}$. Set $W=W_{\tau_1}\cap\ldots\cap W_{\tau_n}$. Then $W$ is an open neighborhood of $g$ and $\tau'(g')\in \NN$ for all $\tau'\in\Gal$ and $g'\in W$. Thus $g'\in\tau(\NN)$ for all $g'\in W$ and all $\tau\in \Gal$, i.e., $W\subseteq \mathtt{U}$.
\end{proof}

\begin{proof}[Proof of Theorem \ref{theo: equivalence}] Given Corollary \ref{cor:  s equivalence},	
	it suffices to show that for an ind-\'{e}tale \ks-Hopf algebra $R$,  
	the profinite $\Gal$-$\s$-group $\GG=\Hom(R^\sharp,k_s^\sharp)$ is expansive if and only if $R$ is finitely $\s$-generated over $k$.
	
	First assume that $R$ is finitely $\s$-generated, so that $R=k\{G\}$ for the $\s$-\'{e}tale $\s$-algebraic group $G=\Hom(R,-)$. By Lemma \ref{lemma: etale Zariski closures} there exist an \'{e}tale algebraic group $\H$ and a $\s$-closed embedding $G\to [\s]_k\H$. On the side of the coordinate rings, this corresponds to a surjective morphism $k\{\H\}\to k\{G\}$ ind-\'{e}tale \ks-Hopf algebras. Which, in turn, according to Theorem~\ref{theo: classical equivalence}, corresponds to an injective morphism $\f\colon \GG\to \Hom(k\{\H\}^\sharp,k_s^\sharp)$ of profinite $\Gal$-$\s$-groups. By Example \ref{ex: equ for etale}, the profinite $\Gal$-$\s$-group $\Hom(k\{\H\}^\sharp,k_s^\sharp)$ can be identified with $\HH^\nn$, where $\HH=\Hom(k[\H],k_s)$ is a finite (discrete) group and $\s\colon \HH^\nn\to\HH^\nn$ is the shift map.
	Set $\mathtt{U}=1\times\HH\times\HH\times\ldots\leq \HH^\nn$. Then $\mathtt{U}$ is an open normal subgroup of $\HH^\nn$ with $\bigcap_{i\in\nn}\s^{-i}(\mathtt{U})=1$. Set $\NN=\f^{-1}(\mathtt{U})\leq\GG$. Then $\NN$ is an open normal subgroup of $\GG$ and
	$$\bigcap_{i\in\nn}\s^{-i}(\NN)=\bigcap_{i\in\nn}\s^{-i}(\f^{-1}(\mathtt{U}))=\bigcap_{i\in\nn}\f^{-1}(\s^{-i}(\mathtt{U}))=\f^{-1}\Big(\bigcap_{i\in\nn}\s^{-i}(\mathtt{U})\Big)=\f^{-1}(1)=1,$$
since $\f$ is injective. Thus $\GG=\Hom(R^\sharp,k_s^\sharp)$ is expansive.

	Conversely, assume that $\GG=\Hom(R^\sharp,k_s^\sharp)$ is expansive. Let $\mathtt{U}$ be as in Lemma \ref{lemma: s continuous} and set $\HH=\GG/\mathtt{U}$. Then $\HH$ is a finite (discrete) group equipped with a continuous action of $\Gal$. We consider $\HH^\nn$ as a profinite $\Gal$-$\s$-group as in Example \ref{ex: equ for etale}. In particular, $\Gal$ is acting on $\HH^\nn$ via
	$\tau(h_0,h_1,h_2,\ldots)=(\tau(h_0),\s(\tau)(h_1),\s^2(\tau)(h_2),\ldots)$ for $\tau\in\Gal$.
	
	The map $\f\colon\GG\to \HH^\nn,\ g\mapsto(\overline{\s^i(g)})_{i\in\nn}$ is a continuous group homomorphism that commutes with $\s$. Moreover, for $\tau\in\Gal$ and $g\in\GG$ we have
	$$\f(\tau(g))=(\overline{\s^i(\tau(g))})_{i\in\nn}=(\overline{\s^i(\tau)(\s^i(g))})_{i\in\nn}=\tau(\overline{\s^i(g)})_{i\in\nn}=\tau(\f(g)).$$
	So $\f$ is a morphism of profinite $\Gal$-$\s$-groups. Since $\bigcap_{i\in\nn}\s^{-i}(\mathtt{U})=1$, the map $\f$ is injective.

	Let $\H$ be the \'{e}tale algebraic group corresponding to the finite group $\HH$ with the continuous $\Gal$-action.	According to Theorem \ref{theo: equivalence}, Theorem \ref{theo: classical equivalence} and Example \ref{ex: equ for etale}, the injective morphims $\f$ of profinite $\Gal$-$\s$-groups corresponds to a surjective morphism $k\{\H\}\to R$ of \ks-algebras. As $k\{\H\}$ is finitely $\s$-generated, it follows that also $R$ is finitely $\s$-generated.	
\end{proof}

\begin{rem} \label{rem: symbolic dynamics}
	It is a natural question to ask, which profinite $\Gal$-$\s$-spaces correspond to the finitely $\s$-generated \ks-algebras under Proposition \ref{prop: s equivalence}? As explained in \cite[Theorem 5.1]{TomasicWibmer:DifferenceGaloisTheoryAndDynamics}, these are exactly the subshifts. The subshifts of finite type, extensively studied in symbolic dynamics (\cite{LindMarcus:IntroductionToSymbolicDynamisAndCoding}, \cite{Kitchens:SymbolicDynamics}), correspond to finitely $\s$-presented \ks-algebras.
\end{rem}

As an immediate corollary to Theorem \ref{theo: equivalence} we obtain:

\begin{cor}
	Let $k$ be a separably algebraically closed $\s$-field. Then the category of $\s$-\'{e}tale $\s$-algebraic groups (over $k$) is equivalent to the category of profinite groups equipped with an expansive endomorphism. \qed
\end{cor}

From Theorem \ref{theo: equivalence} we also obtain a combinatorial-arithmetic description of the category of \ssetale{} $\s$-algebraic groups. 

\begin{cor}
	The category of \ssetale{} $\s$-algebraic groups is equivalent to the category of finite groups equipped with an automorphism and a compatible continuous action of $\Gal$.
\end{cor}
\begin{proof}
	Let $G$ be a $\s$-\'{e}tale $\s$-algebraic group and $\GG=\Hom(k\{G\}^\sharp,k_s^\sharp)$ the corresponding profinite $\Gal$-$\s$-group. Clearly, $k\{G\}$ is an \'{e}tale $k$-algebra if and only if $\GG$ is finite.
	So, assuming that $k\{G\}$ is \'{e}tale, it suffices to show that $k\{G\}$ is $\s$-separable if and only if $\s\colon \GG\to \GG$ is bijective. 
	
	As in the proof of Proposition \ref{prop: s equivalence}, the map $\s\colon \GG\to \GG$ can be interpreted as a morphism $\s\colon \GG\to{\hs\GG}$ of profinite $\Gal$-spaces. According to Lemma \ref{lemma: triangle}, the corresponding morphism of $k$-algebras is $\overline{\s}\colon {\hs(k\{G\})}=k\{G\}\otimes_k k\to k\{G\},\ f\otimes\lambda\mapsto \s(f)\lambda$. So $\s\colon\GG\to\GG$ is bijective if and only if $\overline{\s}\colon {\hs(k\{G\})}\to k\{G\}$ is bijective. Since $k\{G\}$ is a finite dimensional $k$-vector space, the latter is equivalent to $\overline{\s}\colon {\hs(k\{G\})}\to k\{G\}$ being injective. This in turn is equivalent to $k\{G\}$ being $\s$-separable.	
%
%
\end{proof}

We conclude this section with an example, illustrating the equivalence of Theorem \ref{theo: equivalence}.

\begin{ex}
	Let $\HH=\{1,h,h^2,h^3\}$ be the cyclic group of order four (considered as a discrete topological group). As in Example \ref{ex: equ for etale} we consider $\HH^\nn$ as a profinite group via the product topology. The map $\s\colon \HH^\nn\to \HH^\nn$ is the shift. Consider the subgroup $\GG$ of $\HH^\nn$ given by
	$\GG=\{(h_0,h_1,h_2,\ldots)\in\HH^\nn|\ h_i^2=h_{i+1}^2\ \forall \ i\in\nn \}$. Then $\GG$ is a closed subgroup of $\HH^\nn$ and invariant under $\s\colon\HH^\nn\to\HH^\nn$. Since $\s\colon\HH^\nn\to\HH^\nn$ is expansive, also $\s\colon\GG\to\GG$ is expansive. So $\GG$ is a profinite group equipped with an expansive endomorphism. 
	
	Let us also add the action of an absolute Galois group. Let $k=\mathbb{Q}$, considered as a constant $\s$-field and let $\Gal$ denote the Galois group of $k_s=\overline{\mathbb{Q}}$ over $k$. As extension of $\s$ to $k_s$ we choose the identity map. So the action of $\s$ on $\Gal$ is trivial and consequently a compatible action of $\Gal$ on $G$ is a continuous action that commutes with $\s$. Let $\Gal$ act on $\HH$ as $\Gal$ acts on $\{1,i,-1, -i\}=\{a\in \overline{\mathbb{Q}}|\ h^4=1\}$ (under the isomorphism determined by $i\mapsto h$).
	Let $\Gal$ act on $\GG$ by $\tau(h_0,h_1,h_2,\ldots)=(\tau(h_0),\tau(h_1),\tau(h_2),\ldots)$ for $\tau\in\Gal$. Then $\GG$ is an expansive profinite $\Gal$-$\s$-group.
	
	The corresponding $\s$-\'{e}tale $\s$-algebraic group $G$ is the $\s$-closed subgroup of $\Gm$ given by
	$$G(R)=\{g\in R^\times|\ g^4=1,\ \s(g)^2=g^2\}\leq\Gm(R)$$
	for any \ks-algebra $R$.	
\end{ex}

%

\section{A decomposition theorem for \'{e}tale difference algebraic groups}

\label{sec: decomposition theorem}

In this section we establish a rather precise structure theorem for $\s$-\'{e}tale $\s$-algebraic groups (Theorem \ref{theo: Babbitt}). In particular, this theorem shows that any $\s$-\'{e}tale $\s$-algebraic group is build up from benign $\s$-algebraic groups and finite $\s$-\'{e}tale $\s$-algebraic groups.

\subsection{Infinitesimal difference algebraic groups}  \label{subsec: Infinitesimal}
The last $\s$-closed subgroup in the subnormal series of Theorem \ref{theo: Babbitt} is $\s$-infinitesimal. In this subsection we establish the properties of $\s$\=/infinitesimal $\s$-algebraic groups relevant for the proof of Theorem \ref{theo: Babbitt}.

Recall that an algebraic group $\G$ is \emph{infinitesimal} if $\G(T)=1$ for every reduced $k$-algebra $T$. 
The following definition introduces a $\s$-analog of infinitesimal algebraic groups. A \ks-algebra $R$ is called \emph{$\s$-reduced} if $\s\colon R\to R$ is injective.

\begin{defi}
	A $\s$-algebraic group $G$ is $\s$-infinitesimal if $G(R)=1$ for every $\s$-reduced \ks\=/algebra $R$.
\end{defi}

\begin{ex}
	Let $n\geq 2$. The $\s$-closed subgroup $G$ of $\Gm$ given by
	$$G(R)=\{g\in R^\times|\ g^n=1,\ \s(g)=1\}$$ 
	for any \ks-algebra $R$, is $\s$-infinitesimal.
\end{ex}

\begin{ex}
	For $r\geq 1$ the $\s$-closed subgroup $G$ of $\Gl_n$ given by
	$$G(R)=\{g\in\Gl_n(R)|\ \s^r(g)=I_n\}$$
	for any \ks-algebra $R$, is $\s$-infinitesimal. (Here $I_n$ is the $n\times n$ identity matrix.)
\end{ex}
The following lemma gives an algebraic characterization of $\s$-infinitesimal $\s$-algebraic groups.

\begin{lemma} \label{lemma: sinfinitesimal}
	A $\s$-algebraic group $G$ is $\s$-infinitesimal if and only if for every $f\in\m_G$ there exists an $n\in\nn$ such that $\s^n(f)=0$. In fact, if $G$ is $\s$-infinitesimal, then $\s^n(\m_G)=0$ for some 
$n\in\nn$.  
\end{lemma}
\begin{proof}
	 Let $\ida=\{f\in k\{G\}|\ \exists \ n\in\nn : \s^n(f)=0\}$ denote the reflexive closure (cf. \cite[p.~107]{Levin}) of the zero ideal of $k\{G\}$. Then $\ida$ is a $\s$-ideal of $k\{G\}$ and $R=k\{G\}/\ida$ is $\s$-reduced.

	Assume that $G$ is $\s$-infinitesimal. Then the canonical map $k\{G\}\to R$ factors through the counit $\varepsilon\colon k\{G\}\to k$. So $\m_G\subseteq \ida$. 
	
	Conversely, assume that $\m_G\subseteq \ida$. Then $\m_G=\ida$, since $\m_G$ is a maximal ideal of $k\{G\}$. Let $R$ be a $\s$-reduced \ks-algebra and $g\colon k\{G\}\to R$ a morphism of \ks-algebras. To show that $G(R)=1$, it suffices to show that $g(\m_G)=0$. But if $f\in\m_G=\ida$, then there exists an $n\in\nn$ with $\s^n(f)=0$ and so $\s^n(g(f))=g(\s^n(f))=g(0)=0$. Since $\s\colon R\to R$ is injective, this implies $g(f)=0$.
	
	The $\s$-ideal $\m_G$ is finitely $\s$-generated, i.e., of the form  $\m_G=[f_1,\ldots,f_m]$, for $f_1,\ldots,f_m\in\m_G$ (e.g., by \cite[Theorem 4.1]{Wibmer:FinitenessPropertiesOfAffineDifferenceAlgebraicGroups}). So we can find $n\in\nn$ such that $\s^n(f_i)=0$ for $i=1,\ldots,m$. But then $\s^n(\m_G)=0$.	
	%
%
%
%
\end{proof}

\begin{cor} \label{cor: sinfinitesimal has ld 1}
	A $\s$-infinitesimal $\s$-algebraic group has limit degree one and is $\s$-connected.
\end{cor}
\begin{proof}
	Let $G$ be a $\s$-infinitesimal $\s$-algebraic group. According to Lemma \ref{lemma: ld=1} it suffices to show that $k\{G\}$ is finitely generated as a $k$-algebra. 
	We have $k\{G\}=k\oplus\m_G$ (direct sum of $k$-vector spaces). Therefore, we can find a finite $\s$-generating set $B$ of $k\{G\}$ that is contained in $\m_G$. By Lemma~\ref{lemma: sinfinitesimal} there exists an $n\in\nn$ such that $\s^n(B)=0$. So $k\{G\}$ is generated by $B,\s(B),\ldots,\s^{n-1}(B)$ as a $k$-algebra.
	
	Suppose $G$ is not $\s$-connected. Then $k\subsetneqq \pis(k\{G\})\subseteq k\oplus\m_G=k\{G\}$ and so $\pis(k\{G\})$ contains an element $f$ of $\m_G$. But then $\s^n(f)=0$ for some $n$. This contradicts the $\s$-separability of $\pis(k\{G\})$. 
\end{proof}

A $\s$-algebraic group $G$ is \emph{finite} if $k\{G\}$ is a finite dimensional $k$-vector space.

\begin{cor}
	A $\s$-infinitesimal $\s$-\'{e}tale $\s$-algebraic group is finite.
\end{cor}
\begin{proof}
	If $G$ is a $\s$-infinitesimal $\s$-\'{e}tale $\s$-algebraic group, then $k\{G\}$ is ind-\'{e}tale and finitely generated as a $k$-algebra by Corollary \ref{cor: sinfinitesimal has ld 1}. Thus $k\{G\}$ is an \'{e}tale $k$-algebra.
\end{proof}

\begin{lemma} \label{lemma: Gsc sinfinitesimal}
	A reduced finite $\s$-connected $\s$-algebraic group is $\s$-infinitesimal.
\end{lemma}
\begin{proof}
Because $G$ is finite, $\spec(k\{G\})$ is a finite discrete topological space. Set
	$$Z=\{\p\in\spec(k\{G\})|\ \exists\ n\in\nn:\ \s^{-n}(\p)=\m_G\}.$$
	Then $Z$ and the complement of $Z$ are $\s$-closed subsets of $\spec(k\{G\})$. Because $\spec(k\{G\})$ is $\s$\=/connected and $\m_G\in Z$, we must have $Z=\spec(k\{G\})$. It follows that there exists an $n\in\nn$ such that $\s^{-n}(\p)=\m_G$ for every prime ideal $\p$ of $k\{G\}$. The nilradical of $k\{G\}$, which, by assumption, is trivial, is the intersection of all prime ideals of $k\{G\}$.  Therefore $\s^{-n}(0)=\s^{-n}(\sqrt{0})=\m_G$. So $G$ is $\s$-infinitesimal by Lemma \ref{lemma: sinfinitesimal}.
%
%
\end{proof}


\begin{lemma} \label{lemma: inverse image of sinfinitesimal is sinfintesimal}
	Let $\f\colon G\to H$ be a morphism of $\s$-algebraic groups and let $H'$ be a $\s$\=/infinitesimal  $\s$-closed subgroup of $H$. If $\ker(\f)$ is $\s$-infinitesimal, then $\f^{-1}(H')$ is $\s$-infinitesimal.
\end{lemma}
\begin{proof}
	Let $R$ be a $\s$-reduced \ks-algebra. Then
	$$\f^{-1}(H')(R)=\{g\in G(R)|\ \f(g)\in H'(R)=1\}=\ker(\f)(R)=1.$$
\end{proof}

\subsection{The difference Frobenius morphism}
%
The Frobenius morphism and the closely related Frobenius kernels play an important role in the representation theory of algebraic groups in positive characteristic. (See e.g., \cite[Section 9, Part I]{Jantzen:RepresentationsOfAlgebraicGroups}.) In this subsection we introduce a difference analog of the Frobenius morphism and establish the properties relevant for the proof of Theorem \ref{theo: Babbitt}.

The idea is that in most of the constructions and results from \cite[Section 9, Part I]{Jantzen:RepresentationsOfAlgebraicGroups}, the Frobenius endomorphism $a\mapsto a^p$ can be replaced by our ``abstract'' endomorphism $\s$.

%
%
%
%


Let $R$ be a \ks-algebra. For $r\in\nn$ the map $\s^r\colon k\to k,\ \lambda\mapsto\s^r(\lambda)$ is a morphism of difference rings. We denote by
$_{\s^r}R$ the $k$-$\s$-algebra obtained from $R$ by restriction of scalars via $\s^r\colon k\to k$. That is, $_{\s^r}R$ equals $R$ as a difference ring, but the $k$-algebra structure map is $k\to {_{\s^r}R},\ \lambda\mapsto \s^r(\lambda)$. Note that $\s^r\colon R\to {_{\s^r}R},\ a\mapsto \s^r(a)$ is a morphism of $k$-$\s$-algebras.

Let $X$ be a $\s$-variety and let ${\hsr X}$ denote the $\s$-variety obtained from $X$ by base change along $\s^r\colon k\to k$. So
$${\hsr X}(R)=X({_{\s^r}R})$$
for any \ks-algebra $R$ and ${\hsr X}$ is represented by $k\{{\hsr X}\}={\hsr (k\{G\})}=k\{G\}\otimes_k k$.

We define a morphism $$F_X^r\colon X\to{\hsr X}$$
of $\s$-varieties over $k$
by $(F_X^r)_R:=X(\s^r)\colon X(R)\to X({_{\s^r}R})={\hsr X}(R)$
for any \ks-algebra $R$. This makes sense because $\s^r\colon R\to {_{\s^r}R}$ is a morphism of $k$-$\s$-algebras. Moreover, if $\psi\colon R\to R'$ is a morphism of \ks-algebras, then

$$
\xymatrix{
	R \ar[d]_-\psi \ar[r]^-{\s^r} &  {_{\s^r}R} \ar[d]^-\psi  \\
	R' \ar[r]^-{\s^r} & {_{\s^r}R'}	
}
$$
is a commutative diagram of \ks-algebras. Therefore
$$
\xymatrix{
	X(R) \ar[d]_-{X(\psi)} \ar[rr]^-{(F_X^r)_R} & & {\hsr X}(R) \ar[d]^-{{\hsr X}(\psi)} \\
	X(R') \ar[rr]^-{(F_X^r)_{R'}} & &{\hsr X}(R')	
}
$$
commutes, so that $F_X^r$ is indeed natural in $R$, as required.
We write $F_X$ for $F_X^1$ and call $F_X$ the \emph{$\s$-Frobenius morphism} of $X$.

If $\f\colon X\to Y$ is a morphism of $\s$-varieties, then
$$
\xymatrix{
	X(R) \ar[d]_-{X(\s^r)} \ar[r]^-{\f_R} &  Y(R) \ar[d]^-{Y(\s^r)} \\
	X(_{\s^r}R) \ar[r]^-{\f_{(_{\s^r}R)}} & Y(_{\s^r}R)	
}
$$
commutes for any $k$-$\s$-algebra $R$. Therefore, we have a commutative diagram
$$
\xymatrix{
	X \ar[r]^-\f \ar[d]_-{F_X^r} &  Y \ar[d]^-{F_Y^r} \\
	{\hsr X} \ar[r]^-{{\hsr\f}} & {\hsr Y}	
}
$$
of $\s$-varieties.

The dual morphism
$(F_X^r)^*\colon {\hsr(k\{X\})}\to k\{X\}$ is the image of $\id\in X(k\{X\})=\Hom(k\{X\},k\{X\})$ in
$${\hsr X}(k\{X\})=\Hom(k\{X\},{_{\s^r}(k\{X\})})=\Hom(\hsr(k\{X\}),k\{X\})$$
under $(F_X^r)_{k\{X\}}$. 
Thus $(F_X^r)^*$ is given by
\begin{equation} \label{eqn: dual for sFrob}
(F_X^r)^*\colon {\hsr(k\{X\})}\to k\{X\},\quad f\otimes\lambda\mapsto\s^r(f)\lambda.
\end{equation}



Note that if $G$ is a $\s$-algebraic group. Then $F_G^r\colon G\to{\hsr G}$ is a morphism of $\s$-algebraic groups.

%
%

\begin{ex} \label{ex: explicit sfrob}
	Let $G$ be a $\s$-algebraic group. By Proposition \ref{prop: linearization} we may embed $G$ as a $\s$-closed subgroup into some $\Gl_n$. Then also ${\hsr G}$ is naturally embedded into $\Gl_n$: The equations defining ${\hsr G}$ as a $\s$-closed subgroup of $\Gl_n$ are obtained from the equations defining $G$ as a $\s$-closed subgroup of $\Gl_n$ by applying $\s^r$ to the coefficients of the equations. The endomorphism $\f$ of the $\s$-algebraic group $[\s]_k\Gl_n$ given by
	$$\f_R\colon \Gl_n(R)\to\Gl_n(R),\ g\mapsto \s^r(g)$$
	where $\s^r(g)$ is the matrix obtained from $g$ by applying $\s$ to the coefficients, restricts to a morphism $\f\colon G\to {\hsr G}$. The dual map of $\f$ on $k\{\Gl_n\}=k\{x_{ij},\frac{1}{\det(x)}\}$ is given by
	$$\f^*\colon k\{x_{ij},\tfrac{1}{\det(x)}\}\to k\{x_{ij},\tfrac{1}{\det(x)}\},\ x_{ij}\mapsto \s^r(x_{ij}).$$
	It follows that the dual map of $\f\colon G\to {\hsr G}$ agrees with the dual map of $F_G^r$ given in (\ref{eqn: dual for sFrob}). In other words, $\f$ agrees with $F_G^r$. (The advantage of the abstract description is that it shows that $\f$ does not depend on the chosen embedding of $G$ into $\Gl_n$.)
\end{ex}

Because of Example \ref{ex: explicit sfrob}, we may sometimes simply write $\s^r(g)$ instead of $F_G^r(g)$, for $g\in G(R)$ and $R$ a \ks-algebra. 

Our next goal is to understand when the $\s$-Frobenius morphism $F_G\colon G\to{\hsr G}$ is a quotient map. The following definition will help answer this question.

\begin{defi} \label{defi: absolutely sreduced}
	A $\s$-algebraic group $G$ is \emph{$\s$-reduced} if $k\{G\}$ is $\s$-reduced (i.e., $\s\colon k\{G\}\to k\{G\}$ is injective). A $\s$-algebraic group $G$ is \emph{absolutely $\s$-reduced} if $G_{k'}$ is $\s$-reduced for every $\s$-field extension $k'$ of $k$.
\end{defi}
In other words, $G$ is absolutely $\s$-reduced if and only if $k\{G\}$ is a $\s$-separable \ks-algebra.
The following example shows that a $\s$-reduced $\s$-algebraic group need not be absolutely $\s$-reduced.
\begin{ex}
	Let $k$ be a $\s$-field such that there exists a $\lambda\in k$ that is transcendental over $\s(k)$.
	Let $G$ be the $\s$-closed subgroup of the additive group given by
	$$G(R)=\{g\in R|\ \s^2(g)+\lambda\s(g)=0\}\leq\Ga(R)$$
	for any \ks-algebra $R$. Then $k\{G\}=k[y,\s(y)]$. To show that $\s$ is injective on $k\{G\}$, it suffices to show that $\s(y)$ and $\s^2(y)=-\lambda\s(y)$ are algebraically independent over $\s(k)$. But this is guaranteed by the assumption on $\lambda$. Thus $G$ is $\s$-reduced. However, $G$ is not absolutely $\s$\=/reduced. Over the inversive closure $k^*$ of $k$ (\cite[Def. 2.1.6]{Levin}) we can find a $\mu\in k^*$ such that $\s(\mu)=\lambda$ and then $\s(y)+\mu y$ lies in the kernel of $\s$ on $k^*\{G_{k^*}\}=k^*[y,\s(y)]$.
\end{ex}

The following lemma implies that an algebraic group, when considered as a $\s$-algebraic group is absolutely $\s$-reduced.

\begin{lemma} \label{lemma: sT is separable}
	Let $T$ be a $k$-algebra. Then $[\s]_k T$ is $\s$-separable over $k$.
\end{lemma}
\begin{proof}
	For a $\s$-field extension $k'$ of $k$ we have $([\s]_k T)\otimes_k k'=[\s]_{k'} (T\otimes_k k')$. Therefore, it suffices to show that $\s\colon [\s]_k T\to [\s]_kT$ is injective. Indeed, it suffices to show that $\s\colon T[i]\to T[i+1]$ is injective for every $i\geq 0$. But $\s\colon T[i]\to {\hs (T[i])}\to T[i+1]$ is the composition of two injective maps. 
\end{proof}
From Lemma \ref{lemma: sT is separable} we immediately obtain:
\begin{cor} \label{cor: algebraic group is absolutely sreduced}
	Let $\G$ be an algebraic group. Then $[\s]_k\G$ is an absolutely $\s$-reduced $\s$-algebraic group. \qed
\end{cor}

The following lemma explains when $F_G\colon G\to{\hs G}$ is a quotient map.

\begin{lemma} \label{lemma: absolutely sreduced and sFrobenius}
	Let $G$ be a $\s$-algebraic group. The following statements are equivalent:
	\begin{enumerate}
		\item $F_G\colon G\to{\hs G}$ is a quotient map.
		\item $F_G^r\colon G\to{\hsr G}$ is a quotient map for every $r\in\nn$.
		\item $G$ is absolutely $\s$-reduced.
	\end{enumerate}
\end{lemma}
\begin{proof}
	The \ks-algebra $k\{G\}$ is $\s$-separable over $k$ if and only if $\overline{\s}=(F_G)^*\colon {\hs(k\{G\})}\to k\{G\}$ is injective. Thus (i) and (iii) are equivalent. To see that (i) implies (ii), note that the injectivity of $\overline{\s}\colon {\hs(k\{G\})}\to k\{G\}$ can be rephrased as: If $f_1,\ldots,f_n\in k\{G\}$ are $k$-linearly independent, then $\s(f_1),\ldots,\s(f_n)\in k\{
	G\}$ are $k$-linearly independent. If this is true, and $f_1,\ldots,f_n\in k\{G\}$ are $k$-linearly independent, then $\s^r(f_1),\ldots,\s^r(f_n)$ are also $k$-linearly independent, i.e., $(F_G^r)^*\colon {\hsr(k\{G\})}\to k\{G\}$ is injective.
\end{proof}

The following definition introduces a $\s$-analog of the Frobenius kernels (\cite[Section 9.4, Part I]{Jantzen:RepresentationsOfAlgebraicGroups}) 

\begin{defi}
	Let $G$ be a $\s$-algebraic group and $r\in\nn$. The kernel $G_{(r)}$ of the morphism $F_G^r\colon G\to{\hsr G}$ is called the \emph{$r$-th $\s$-Frobenius kernel}\index{$\s$-Frobenius kernel} of $G$.
\end{defi}
Since $F_G^{r+s}=F^s_{{\hsr G}}\circ F_G^r$ we have $$1=G_{(0)}\subseteq G_{(1)}\subseteq G_{(2)}\subseteq\cdots.$$
Moreover, from formula (\ref{eqn: dual for sFrob}) we see that
$$\I(G_{(r)})=(\s^r(\m_G))\subseteq k\{G\}$$
and
$$k\{G/G_{(r)}\}=k[\s^r(k\{G\})]\subseteq k\{G\}.$$

\begin{ex}
	The $r$-th $\s$-Frobenius kernel of $[\s]_k\Gl_n$ is given by
	$$([\s]_k\Gl_n)_{(r)}(R)=\{g\in \Gl_n(R)|\ \s^r(g)=I_n\}$$
	for any \ks-algebra $R$. (Here $I_n$ denotes the $n
	\times n$ identity matrix.)
\end{ex}

As a Corollary to Lemma \ref{lemma: absolutely sreduced and sFrobenius} we obtain:
\begin{cor} \label{cor: sFrobenius quotient for absolutely sreduced}
	Let $G$ be an absolutely $\s$-reduced $\s$-algebraic group and let $r\in\nn$. Then $G/G_{(r)}\simeq {\hsr G}$.
\end{cor}
\begin{proof}
	This follows immediately from Lemma \ref{lemma: absolutely sreduced and sFrobenius} (and Theorem \ref{theo: isom1}).
\end{proof}

The following lemma explains the close connection between $\s$-infinitesimal $\s$-algebraic groups and the $\s$-Frobenius kernels.

\begin{lemma} \label{lemma: Gr sinfinitesimal}
	Let $G$ be a $\s$-algebraic group. Then $G_{(r)}$ is a $\s$-infinitesimal $\s$-algebraic group for every $r\in\nn$. Conversely, if $H$ is a $\s$-infinitesimal $\s$-closed subgroup of $G$, then $H$ is contained in some $G_{(r)}$. In particular, if $G$ is $\s$-infinitesimal, then $G=G_{(r)}$ for some $r\in\nn$.
\end{lemma}
\begin{proof}
	Since $k\{G_{(r)}\}=k\{G\}/(\s^r(\m_G))$, it is clear that the augmentation ideal $\m_{G_{(r)}}$ of $G_{(r)}$ satisfies $\s^r(\m_{G_{(r)}})=0$. So $G_{(r)}$ is $\s$-infinitesimal (Lemma \ref{lemma: sinfinitesimal}).
	
	If $H\leq G$ is $\s$-infinitesimal, then $\s^r(\m_G)\subseteq\I(H)$ for some $r\in\nn$ by Lemma \ref{lemma: sinfinitesimal}. So $\I(G_{(r)})=(\s^r(\m_G))\subseteq\I(H)$ and consequently $H\leq G_{(r)}$.
\end{proof}

%
%
%

For later use we record:

\begin{lemma} \label{lemma: quotients and sFrobenius for absolutely sreduced}
	Let $H$ be a $\s$-closed subgroup of a $\s$-algebraic group $G$ and let $N$ be a normal $\s$-closed subgroup of $H$. If $H/N$ is absolutely $\s$-reduced, then $$(F^r_G)^{-1}({\hsr H})/(F^r_G)^{-1}({\hsr N})\simeq {\hsr(H/N)}$$ for every $r\in\nn$.
\end{lemma}
\begin{proof}
	Using Lemma \ref{lemma: quotients and base change}, we obtain a morphism
	$$(F^r_G)^{-1}({\hsr H})\xrightarrow{F_G^r} {\hsr H}\to{\hsr H}/{\hsr N}={\hsr(H/N)}$$
	with kernel $(F^r_G)^{-1}({\hsr N})$. By Lemma \ref{lemma: quotient embedding}, this yields a $\s$-closed embedding $$\f\colon (F^r_G)^{-1}({\hsr H})/(F^r_G)^{-1}({\hsr N})\to {\hsr(H/N)}.$$
	Since $H/N$ is absolutely $\s$-reduced, $F^r_{H/N}$ is a quotient map (Lemma \ref{lemma: absolutely sreduced and sFrobenius}).
	As
	$$\xymatrix{
		H/N \ar[rr] \ar[rd]_{F_{H/N}^r}& & (F^r_G)^{-1}({\hsr H})/(F^r_G)^{-1}({\hsr N}) \ar[ld]^\f \\
		& {\hsr (H/N)} &
	}
	$$
	commutes, this implies that $\f$ must also be a quotient. Thus $\f$ is an isomorphism by Corollary~\ref{cor: sclosed and quotient is isom}.
\end{proof}

Note that the functor $\G\rightsquigarrow [\s]_k\G$ from the category of algebraic groups to the category of $\s$\=/algebraic groups is not full. For example, for $\lambda_0,\ldots,\lambda_{n-1}\in k$, the morphism $\f\colon [\s]_k\Ga\to [\s]_k\Ga$ given by
$\f(g)= \s^n(g)+\lambda_{n-1}\s^{n-1}(g)+\ldots+\lambda_0g$ for $g\in R$ and any \ks-algebra $R$, is not induced by an endomorphism of the algebraic group $\Ga$. Nevertheless, we have the following result.

\begin{prop} \label{prop: sisom implies isom}
	Let $\G$ and $\H$ be algebraic groups. If $[\s]_k\G$ and $[\s]_k\H$ are isomorphic as $\s$\=/algebraic groups, then $\G$ and $\H$ are isomorphic as algebraic groups.
\end{prop}
\begin{proof}
	The morphism $F_{[\s]_k\G}\colon [\s]_k\G\to {\hs([\s]_k\G)}=[\s]_k{\hs\G}$ of $\s$-algebraic groups has the morphism $$\G\times {\hs\G}\times{^{\s^2\!}\G}\times\ldots\to {\hs\G}\times{^{\s^2\!}\G}\times\ldots,\ (g_0,g_1,g_2,\ldots)\mapsto (g_1,g_2,\ldots)$$ as underlying morphism of affine group schemes. So  $([\s]_k\G)_{(1)}$ has $\G$ as the underlying affine group scheme. If $[\s]_k\G\simeq [\s]_k\H$, then $([\s]_k\G)_{(1)}\simeq ([\s]_k\H)_{(1)}$ and therefore also $\G\simeq \H$.
\end{proof}

\subsection{Simple \'{e}tale algebraic groups} The material in this subsection will be helpful for establishing the uniqueness part of our main decomposition theorem (Theorem \ref{theo: Babbitt}).

\begin{defi}
	Let $\G$ be a non-trivial \'{e}tale algebraic group. Then $\G$ is \emph{simple} if for every normal closed subgroup $\N$ of $\G$ either $\N=\G$ are $\N=1$. Moreover, $\G$ is \emph{$\s$-stably simple} if ${\hsi\G}$ is simple for every $i\in\nn$.
\end{defi}

The following example illustrates that for a simple \'{e}tale algebraic group $\G$, the finite group $\G(k_s)$ need not be a simple group.

\begin{ex} \label{ex: simple etale algebraic group}
	Let $V=\{(0,0), (1,0), (0,1), (1,1)\}$ be the Klein four group. Its automorphism group can be identified with the symmetric group $S_3$ on $(1,0), (0,1), (1,1)$. Let $\Gal$ be the Galois group of $k_s/k$ and let $\Gal\to S_3$ be a surjective continuous morphism (i.e., we assume that $S_3$ is a Galois group over $k$). Let $\G$ be the \'{e}tale algebraic group over $k$ associated to this continuous action of $\Gal$ on $V$ (as in Section \ref{sec: Expansive endomorphisms and etale difference algebraic groups}). Then $\G(k_s)\simeq V$ is not a simple group. However, $\G$ is a simple \'{e}tale algebraic group because no non-trivial proper subgroup of $V$ is invariant under the action of $\Gal$. 
\end{ex}
%
%
%
The following example shows that a simple \'{e}tale algebraic group need not be $\s$-stably simple.

\begin{ex}
	Let $k$ be a $\s$-field and as in Section \ref{sec: Expansive endomorphisms and etale difference algebraic groups} let us fix an extension of $\s$ to the separable algebraic closure $k_s$ of $k$. Let $a\in k_s$ be such that $\s(a)\in k$ and $k(a)\subseteq k_s$ is a Galois extension of $k$ with Galois group isomorphic to $S_3$. Let $\G$ be the simple algebraic group from Example~\ref{ex: simple etale algebraic group} associated with the surjection $\Gal\to S_3$, where $\Gal$ is the Galois group of $k_s/k$. As in Lemma \ref{lemma: action of s on Gal} the profinite group $\Gal$ is equipped with an endomorphism $\s\colon \Gal\to \Gal$. The diagram
	$$
	\xymatrix{
		k_s \ar_\s[d] \ar^{\s(\tau)}[r] & k_s \ar^\s[d]  \\
		k_s \ar^\tau[r] & k_s	
	}
	$$
	shows that $\s(\tau)(a)=a$ for all $\tau\in\Gal$, i.e., $\s(\Gal)$ lies in the kernel of $\Gal\to S_3$. So it follows from Lemma \ref{lemma: twist and profinite} that $\Gal$ acts trivial on ${\hs\G}(k_s)$. Since the Klein four group $V$ is not simple, ${\hs\G}$ is not a simple \'{e}tale algebraic group.
	
\end{ex}


Our main goal here is to describe the normal $\s$-closed subgroups of $\s$-stably simple \'{e}tale algebraic groups and the corresponding quotients.

\begin{lemma} \label{lemma: normal subgroup of products}
	Let $\G$ be a $\s$-stably simple \'{e}tale algebraic group that is not commutative. Let $\N$ be a normal closed subgroup of $\G\times{\hs\G}\times\ldots\times {^{\s^n}\!\G}$. Then $\N=\N_0\times\N_1\times\ldots\times\N_n$, where $\N_i\in\{1,{^{\s^i}\!\G}\}$ for $i=0,\ldots,n$.
\end{lemma}
\begin{proof}
	It suffices to show the following: If $g=(g_0,\ldots,g_n)\in\N(k_s)$ with $g_i\neq 1$, then $\H_i=1\times \ldots\times 1\times{\hsi\G}\times 1\times\ldots\times 1\leq \N$.
	
	Since ${\hsi\G}$ is a non-commutative simple \'{e}tale algebraic group, its center is trivial. Thus there exists an $h\in{\hsi\G}(k_s)$ with $hg_i\neq g_ih$.
	Therefore
	\begin{align*}
	&(1,\ldots,1,h,1,\ldots,1)(g_1,\ldots,g_n)(1,\ldots,1,h,1,\ldots,1)^{-1}(g_1,\ldots,g_n)^{-1}= \\
	& =(1,\ldots,1,hg_ih^{-1}g_i^{-1},1,\ldots,1)
	\end{align*} is a non-trivial element of $\mathcal{N}(k_s)\cap \H_i(k_s)$. Since $\H_i$ is simple, we have $\N\cap \H_i=\H_i$, i.e., $\H_i\leq \N$ as desired.
\end{proof}

\begin{cor} \label{cor: sstably simple non comm} 
	Let $\G$ be a $\s$-stably simple \'{e}tale algebraic group that is not commutative and let $N$ be a proper normal $\s$-closed subgroup of $G=[\s]_k\G$. Then $N=G_{(r)}$ for some $r\in\nn$.  
\end{cor}
\begin{proof}
	For $i\in\nn$ let $N[i]$ denote the $i$-th order Zariski closure of $N$ in $\G$. Since $N$ is normal in $\G$, it follows that $N[i]$ is normal in $\G[i]=\G\times{\hs\G}\times\ldots\times{\hsi\G}$ (Lemma \cite[Lemma~3.10]{Wibmer:AlmostSimpleAffineDifferenceAlgebraicGroups}). As $N$ is a proper subgroup of $G$, $N[i]$ must be a proper subgroup of $\G[i]$ for some $i$. Let $r\in\nn$ be such that $N[i]=\G[i]$ for $i=0,\ldots,r-1$ but $N[r]\subsetneqq \G[r]$. By Lemma \ref{lemma: normal subgroup of products} we must have $N[r]=\G\times{\hs\G}\times\ldots\times{^{\s^{r-1}\!}\G}\times 1$. Thus $N=G_{(r)}$.
\end{proof}

The following example shows that the conclusion of Corollary \ref{cor: sstably simple non comm} does not hold for commutative $\s$-stably simple \'{e}tale algebraic group.

\begin{ex}
	Let $q$ be a prime number different from the characteristic of $k$. Then the algebraic group $\G=\mu_q$ defined by $\G(T)=\{g\in T^\times|\ g^q=1\}$ for any $k$-algebra $T$,
	 is a $\s$\=/stably simple \'{e}tale algebraic group. The proper normal $\s$-closed subgroup $N$ of $\G$ given by $N(R)=\{g\in G(R)|\ \s(g)=g\}$ for any \ks-algebra $R$, is not of the form $G_{(r)}$ for some $r\in \nn$.
\end{ex}

%

For clarity of the exposition, we separate the following elementary lemma from the proof of Proposition \ref{prop: key to uniqueness}.

\begin{lemma} \label{lemma: abstract groups}
	Let $i,r\in \nn$ and let $\GG_0,\GG_1,\GG_2,\ldots,\GG_{i+r}$ be abelian groups. For $j=0,\ldots,i$  let $\psi_j\colon\GG_j\times\GG_{j+1}\times\ldots\times\GG_{j+r-1}\to \GG_{j+r}$ be a morphism of groups and set $$\NN_j=\{(g_j,\ldots,g_{j+r})\in \GG_j\times\ldots\times\GG_{j+r}|\ g_{j+r}=\psi_j(g_j,\ldots,g_{j+r-1})\}\leq \GG_j\times\ldots\times\GG_{j+r}.$$ Then the morphism
	$$\f\colon\GG_0\times\GG_1\times\ldots\times\GG_{i+r}\longrightarrow ((\GG_0\times\ldots\times\GG_r)/\NN_0) \times \ldots \times((\GG_i\times\ldots\times \GG_{i+r})/\NN_i)$$ $$(g_0,g_1,\ldots,g_{i+r})\mapsto (\overline{(g_0,\ldots,g_r)},\ldots,\overline{(g_i,\ldots,g_{i+r})})$$
	is surjective.
\end{lemma}
\begin{proof}
	For $j=0,\ldots,i$ let $(g_{j,0},\ldots,g_{j,r})\in \GG_j\times\ldots\times\GG_{j+r}$. So $h=(\overline{(g_{0,0},\ldots,g_{0,r})},\ldots, \overline{(g_{i,0},\ldots,g_{i,r})})$ is an arbitrary element of the codomain of $\f$. Define $g=(g_0,\ldots,g_{i+r})\in\GG_0\times\ldots\times\GG_{i+r}$ by
	$(g_0,\ldots,g_r)=(g_{0,0},\ldots,g_{0,r})$ and then recursively by
	$$g_{r+j}=\psi_j\Big(g_jg_{j,0}^{-1},g_{j+1}g_{j,1}^{-1},\ldots,g_{j+r-1}g_{j,r-1}^{-1}\Big)g_{j,r}$$ for $j=1,\ldots,i$. Then $(g_jg_{j,0}^{-1},g_{j+1}g_{j,1}^{-1},\ldots,g_{j+r}g_{j,r}^{-1})\in \NN_j$ and so $\overline{(g_j,\ldots,g_{j+r})}=\overline{(g_{j,0},\ldots,g_{j,r})}\in (\GG_j\times\ldots\times\GG_{j+r})/\NN_j$. Thus $\f(g)=h$.
\end{proof}

\begin{defi}
Two \'{e}tale algebraic groups $\G$ and $\H$ are \emph{$\s$-stably equivalent} if there exist $m,n\in\nn$ such that ${^{\s^m\!}\G}\simeq {^{\s^n\!}\H}$.
\end{defi}
Note that this defines an equivalence relation on the class of all \'{e}tale algebraic groups over $k$.

\begin{prop} \label{prop: key to uniqueness}
	Let $\G$ be a $\s$-stably simple \'{e}tale algebraic group and let $N$ be a proper normal $\s$-closed subgroup of $G=[\s]_k\G$. Then $G/N\simeq [\s]_k\H$ for some $\s$-stably simple \'{e}tale algebraic group $\H$ $\s$-stably equivalent to $\G$.
\end{prop}
\begin{proof}
	First assume that $\G$ is not commutative. From Corollary \ref{cor: sstably simple non comm} we know that $N=G_{(r)}$ for some $r\in\nn$. So $G/N=G/G_{(r)}\simeq{\hsr G}=[\s]_k{\hsr\G}$ by Corollaries \ref{cor: algebraic group is absolutely sreduced} and \ref{cor: sFrobenius quotient for absolutely sreduced}.
	
	Now assume that $\G$ is commutative. For $i\in\nn$ let $N[i]$ denote the $i$-th order Zariski closure of $N$ in $\G$. Moreover, let $\N_i$ denote the kernel of the projection $\pi_i\colon N[i]\to N[i-1]$. Then $\N_i$ can be identified with a closed subgroup of ${\hsi\G}$. Since ${\hsi\G}$ is simple and commutative, we have $\N_i\in\{1,{\hsi\G}\}$. Let $r\in\nn$ be such that $\N_i={\hsi\G}$ for $i=1,\ldots,r-1$ and $\N_r=1$. Then $\pi_r\colon N[r]\to N[r-1]$ is an isomorphism and $N[r-1]=\G[r-1]=\G\times\ldots\times{^{\s^{r-1}\!}\G}$. Define $\psi\colon \G\times\ldots\times{^{\s^{r-1}\!}\G}=N[r-1]\xrightarrow{\pi_r^{-1}}N[r]\to {\hsr\G}$, where the last map is the projection onto the last coordinate. Then $\psi$ is a morphism of algebraic groups and 
	\begin{equation} \label{eq: N[r]}
	N[r](T)=\{(g_0,\ldots,g_r)\in\G[i](T)|\ g_r=\psi(g_0,\ldots,g_{r-1})\}
	\end{equation}
	for any $k$-algebra $T$. Moreover,
	\begin{equation} \label{eq: N}
	 N(R)=\{g\in\G(R)|\ \s^r(g)=\psi(g,\ldots,\s^{r-1}(g))\}
	\end{equation} 	
	for any \ks-algebra $R$. Set $\H=\G[r]/N[r]$. We claim that the morphism ${\hsr\G}\to \H,\ g\mapsto \overline{(1,\ldots,1,g)}$ is an isomorphism. Clearly, the kernel is trivial and so the claim follows from $|\H|=\frac{|\G[r]|}{|N[r]|}=\frac{|\G|^{r+1}}{|\G|^r}=|{\hsr\G}|$.

	The inclusion $k[\H]\subseteq k[\G[r]]\subseteq k\{G\}$ of Hopf algebras, gives rise to a morphism $k\{\H\}\to k\{G\}$ of \ks-Hopf algebras (Lemma \ref{lemma: Hopf induced}). Let $\f\colon G\to[\s]_k\H$ be the corresponding morphism of $\s$\=/algebraic groups. Note that $\f$ can also be described as the composition
	\begin{equation} \label{eq: map f}
	\f\colon G=[\s]_k\G\to[\s]_k\G\times\ldots\times [\s]_k{\hsr\G}=[\s]_k\G[r]\to [\s]_k\G[r]/[\s]_kN[r]=[\s]_k(\G[r]/N[r]),
	\end{equation}
	where the first map is given by $g\mapsto (g,\s(g),\ldots,\s^r(g))$ for $g\in\G(R)$ and $R$ a \ks-algebra.
	
	We claim that $\f\colon G\to [\s]_k\H$ is a quotient map with kernel $N$. (So $G/N\simeq [\s]_k\H$ as desired.)	
	Indeed, $\ker(\f)=N$ by (\ref{eq: N[r]}), (\ref{eq: N}) and (\ref{eq: map f}). To see that 
	$\f$ is a quotient map, we have to show that $\f^*\colon k\{\H\}\to k\{G\}$ is injective. It suffices to show that the restriction of $\f^*$ to $k[\H[i]]\to k[\G[i+r]]$ is injective for every $i\in\nn$. This restriction corresponds to the morphism
	$$\f_i\colon\G\times\ldots\times{^{\s^{i+r}}\!\G}\to \H\times\ldots\times{\hsi\H},\ (g_0,\ldots,g_{i+r})\mapsto \left(\overline{(g_0,\ldots,g_r)}, \overline{(g_1,\ldots,g_{1+r})},\ldots, \overline{(g_i,\ldots,g_{i+r})}\right)$$
	of \'{e}tale algebraic groups. It suffices to show that $\f_i$ is surjective on the $k_s$-points. But this follows from Lemma \ref{lemma: abstract groups} with $\GG_j=({^{\s^j}\!\G})(k_s)$ for $j=0,\ldots,i+r$ and $\psi_j\colon\GG_j\times\ldots\times\GG_{j+r-1}\to\GG_{j+r}$ the base change of $\psi\colon \G\times\ldots\times{^{\s^{r-1}}\!\G}\to {\hsr\G}$ via $\s^j\colon k\to k$ (and then evaluated at $k_s$) for $j=0,\ldots,i$.
\end{proof}

\begin{cor} \label{cor: normal ssubgroup of simple has ld 1}
	Let $\G$ be a $\s$-stably simple \'{e}tale algebraic group. Then $\ld(N)=1$ for every proper normal $\s$-closed subgroup $N$ of $\G$.
\end{cor}
\begin{proof}
	In case $\G$ is non-commutative, this follows from Lemma \ref{lemma: Gr sinfinitesimal} and Corollaries  \ref{cor: sstably simple non comm} and \ref{cor: sinfinitesimal has ld 1}. 
	In case $\G$ is commutative, this follows from the proof of Proposition \ref{prop: key to uniqueness} ($\N_r=1$).
\end{proof}
The following lemma provides a converse to Corollary \ref{cor: normal ssubgroup of simple has ld 1}.

\begin{lemma} \label{lemma: implies sstably simple}
	Let $\G$ be a non-trivial \'{e}tale algebraic group. If $\ld(N)\in\{1,|\G|\}$ for every normal $\s$-closed subgroup $N$ of $\G$, then $\G$ is $\s$-stably simple. 
\end{lemma}
\begin{proof}
	Suppose $\G$ is not $\s$-stably simple. Then there exists an $r\in\nn$ and a closed normal subgroup $\N$ of ${\hsr\G}$ such that $1<|{\hsr\G}/\N|<|\G|$. 
	Using Corollary \ref{cor: algebraic group is absolutely sreduced} and Lemma \ref{lemma: absolutely sreduced and sFrobenius} we see that the morphism
	$$\f\colon [\s]_k\G\xrightarrow{F^r_{[\s]_k\G}}{\hsr([\s]_k\G)}=[\s]_k{\hsr\G}\to [\s]_k{\hsr\G}/[\s]_k\N= [\s]_k({\hsr\G}/\N)$$
	is a quotient map. Thus, $N=\ker(\f)$ is a normal $\s$-closed subgroup of $G=[\s]_k\G$ and $\ld(G/N)=\ld([\s]_k({\hsr\G}/\N))=|{\hsr\G}/\N|$ by Example \ref{ex: benign}. So  $1<\ld(N)<|\G|$ by Proposition~\ref{prop: invariants and quotients}; a contradiction.
\end{proof}

\subsection{The decomposition theorem}

Some more preparations are in order before we can finally tackle the proof of the main decomposition theorem (Theorem \ref{theo: Babbitt}).
The following proposition provides an important step in the proof. Roughly speaking, it shows that in a subnormal series we can get rid of the top $\s$-infinitesimal quotient.

 For simplicity, we call a $\s$-\'{e}tale $\s$-algebraic group $G$ \emph{$\s$-stably simple benign} if $G\simeq [\s]_k\G$, where $\G$ is a $\s$-stably simple \'{e}tale algebraic group.
\begin{prop} \label{prop: exchange sinfinitesimal for Babbitt}
	Let $G$ be a $\s$-algebraic group with a subnormal series
	$$G\supseteq G_1\supseteq G_2\supseteq \cdots\supseteq G_{n+1}\supseteq 1$$
	such that $G/G_1$ is $\s$-infinitesimal, $G_i/G_{i+1}$ is $\s$-stably simple benign for $i=1,\ldots n$ and $G_{n+1}$ is $\s$-infinitesimal. Then there exists a subnormal series
	$$G\supseteq H_1\supseteq H_2\supseteq \cdots\supseteq H_n\supseteq 1$$
	such that $G/H_1$ and $H_i/H_{i+1}$ ($i=1,\ldots,n-1$) are $\s$-stably simple benign and $H_n$ is $\s$\=/infinitesimal.
\end{prop}
\begin{proof}
	Since $G/G_1$ is $\s$-infinitesimal, there exists an $r\in\nn$ such that $(G/G_1)_{(r)}=G/G_1$ (Lemma~\ref{lemma: Gr sinfinitesimal}). It is then clear from the commutative diagram
	$$
	\xymatrix{
		G \ar[r] \ar[d]_{F^r_G} &  G/G_1 \ar[d]^{F^r_{G/G_1}} \\
		{\hsr G} \ar[r] & {\hsr (G/G_1)}
	}
	$$
	that $(F^r_G)^{-1}({\hsr G_1})=G$. Define
	$$H_i=(F^r_G)^{-1}({\hsr G_{i+1}}) \quad \text{ for } \quad i=0,\ldots,n.$$
	(So $H_0=G$.) Because benign $\s$-algebraic groups are absolutely $\s$-reduced (Corollary~\ref{cor: algebraic group is absolutely sreduced}), it follows from Lemma \ref{lemma: quotients and sFrobenius for absolutely sreduced} that $H_i/H_{i+1}\simeq {\hsr (G_{i-1}/G_i)}$ for $i=0,\ldots,n$.
	So, if $G_{i-1}/G_i\simeq [\s]_k\G_{i}$ with $\G_{i}$ a $\s$-stably simple \'{e}tale algebraic group, then $H_i/H_{i+1}\simeq {\hsr([\s]_k\G_i)}=[\s]_k{\hsr\G_i}$. 

	Finally, as $G_{n+1}$ is $\s$-infinitesimal, also ${\hsr G_{n+1}}$ is $\s$-infinitesimal and therefore the $\s$-algebraic group $H_n=(F^r_G)^{-1}({\hsr G_{n+1}})$ is $\s$-infinitesimal by Lemmas \ref{lemma: inverse image of sinfinitesimal is sinfintesimal} and \ref{lemma: Gr sinfinitesimal}.	
\end{proof}

\begin{lemma} \label{lemma: for Babbitt 1step}
	Let $G$ be a $\s$-\'{e}tale $\s$-algebraic group and $\G$ an \'{e}tale algebraic group. If $\f\colon G\to[\s]_k\G$ is a $\s$-closed embedding of $\s$-algebraic groups such that $\ld(G)=|\G|$, then $\f$ is an isomorphism. In particular, $G$ is benign.
\end{lemma}
\begin{proof}
	We identify $G$ with $\f(G)$. For $i\in\nn$ let $G[i]$ denote the $i$-th order Zariski closure of $G$ in $\G$ and let $\G_i\leq G[i]$ denote the kernel of the projection $\pi_i\colon G[i]\to G[i-1]$ ($\G_0=G[0]$). By Proposition \ref{prop: limit degree} the sequence $(|\G_i|)_{i\in\nn}$ is non-increasing and stabilizes with value $\ld(G)=|\G|$. But as $\G_0$ is an algebraic subgroup of $\G$, we must have $|\G_i|=|\G|$ for all $i\in\nn$.
	
	Since $\pi_i\colon G[i]\to G[i-1]$ is a quotient map with kernel $\G_i$, we have $|G[i]|=|G[i-1]|\cdot|\G_i|$ for $i\geq 1$. As $\G_0=\G$, we find 
	$$|G[i]|=|\G_i|\cdots|\G_0|=|\G|^{i+1}.$$ But $G[i]\leq \G\times {\hs\G}\times\cdots\times  {\hsi\G}=\G[i]$. Therefore $G[i]=\G[i]$ for all $i$ and it follows that $G=[\s]_k\G$ is benign.
\end{proof}

Let $G$ be a $\s$-\'{e}tale $\s$-algebraic group. Let $\G$ be an \'{e}tale algebraic group and $\f\colon G\to [\s]_k\G$ a morphism of $\s$-algebraic groups. Then $\f(G)$ is a $\s$-\'{e}tale $\s$-closed subgroup of $\G$ (Lemma \ref{lemma: subgroups and quotients of setale are setale}). Let $\f(G)[0]$ and $\f(G)[1]$ denote the Zariski closures of $\f(G)$ in $\G$ of order $0$ and $1$ respectively. Then $\f(G)[0]$ and $\f(G)[1]$ are \'{e}tale algebraic groups (Lemma \ref{lemma: etale Zariski closures}). Let $\pi_1\colon \f(G)[1]\to\f(G)[0]$ be the natural projection (as in the end of Section \ref{subsec: affine difference algebraic geometry}). We abbreviate
$$|\f|_1:=|\ker(\pi_1)|.$$
By Proposition \ref{prop: limit degree} we have $|\f|_1\leq\ld(\f(G))$. If $\ker(\f)$ is $\s$-infinitesimal (e.g., trivial), then $$\ld(\f(G))=\ld(G/\ker(\f))=\frac{\ld(G)}{\ld(\ker(\f))}=\ld(G)$$
by Corollary \ref{cor: sinfinitesimal has ld 1}. So $|\f|_1\leq\ld(G).$ 

\begin{defi}
	Let $G$ be a $\s$-\'{e}tale $\s$-algebraic group and $\G$ an \'{e}tale algebraic group. A morphism $\f\colon G\to[\s]_k\G$ of $\s$-algebraic groups is a \emph{standard embedding of $G$ (into $\G$)}\index{standard embedding} if
	\begin{itemize}
		\item $\f$ is a $\s$-closed embedding,
		\item $\f(G)$ is Zariski dense in $\G$ and
		\item $\ld(G)=|\f|_1$.
	\end{itemize}
\end{defi}

In the context of symbolic dynamics (cf. Remark \ref{rem: symbolic dynamics}) the construction in the proof of the following lemma is known as ``passing to a higher block shift'' (\cite[Section \S 1.4]{LindMarcus:IntroductionToSymbolicDynamisAndCoding}).

\begin{lemma} \label{lemma: exists standard embedding}
	Let $G$ be a $\s$-\'{e}tale $\s$-algebraic group. Then there exists a standard embedding of $G$.
\end{lemma}
\begin{proof}
	By Proposition \ref{prop: linearization} there exists an algebraic group $\H$ and a $\s$-closed embedding $G\to[\s]_k\H$ of $\s$-algebraic groups. By Lemma \ref{lemma: etale Zariski closures} the Zariski closures $G[i]_\H$ of $G$ in $\H$ are \'{e}tale algebraic groups. Let $\H_i$ denote the kernel of the projection $\pi_i\colon G[i]_\H\to G[i-1]_\H$. By Proposition \ref{prop: limit degree} there exists an $m\geq 0$ such that $|\H_{m+1}|=\ld(G)$.
	
	Set $\G=G[m]_\H$. The inclusion
	$k[\G]\subseteq k\{G\}$ of $k$-algebras, induces a surjective morphism $k\{\G\}=[\s]_kk[\G]\to k\{G\}$ of \ks-algebras.
	 Since $k[\G]\subseteq k\{G\}$ is an inclusion of Hopf algebras, $k\{\G\}\to k\{G\}$ is a morphism of \ks-Hopf algebras (Lemma \ref{lemma: Hopf induced}). Let $\f\colon G\to[\s]_k\G$ denote the corresponding $\s$-closed embedding of $\s$-algebraic groups and for $i\geq 0$ let $G[i]_\G$ denote the $i$-th order Zariski closure of $G$ in $\G$. Then $G[i]_\G=G[m+i]_\H$ for every $i\geq 0$. Let $\G_i$ denote the kernel of the projection $\pi_i\colon G[i]_\G\to G[i-1]_\G$. Then $\G_i=\H_{m+i}$ for $i\geq 1$. Therefore $|\f|_1=|\G_1|=|\H_{m+1}|=\ld(G)$.
	So $\f$ is a standard embedding.
\end{proof}


%

\begin{defi}
	Let $G$ be a $\s$-\'{e}tale $\s$-algebraic group and $\G$ an \'{e}tale algebraic group. A morphism $\f\colon G\to[\s]_k\G$ of $\s$-algebraic groups is a \emph{substandard embedding} of $G$ (into $\G$) if
	\begin{itemize}
		\item $\ker(\f)$ is $\s$-infinitesimal,
		\item $\f(G)$ is Zariski dense in $\G$ and
		\item $\ld(G)=|\f|_1$.
	\end{itemize}
	A substandard embedding $G\to[\s]_k\G$ of $G$ is \emph{minimal} if $|\G|\leq |\H|$ for any substandard embedding $G\to[\s]_k\H$ of $G$.
\end{defi}
Note that, despite the name, a substandard embedding need not be a $\s$-closed embedding. Since a standard embedding is a substandard embedding it is clear from Lemma~\ref{lemma: exists standard embedding} that for any $\s$-\'{e}tale $\s$-algebraic group there exists a minimal substandard embedding.

\begin{lemma} \label{lemma: for Babbitt step1 sr}
	Let $G$ be a $\s$-\'{e}tale $\s$-algebraic group and let $\f\colon G\to [\s]_k\G$ be a minimal substandard embedding. Then for every $r\in\nn$ the dimension of $k[\s^r(k[\G])]\subseteq k\{G\}$ as a $k$-vector space equals $|\G|$. 
\end{lemma}
\begin{proof}
	First of all, note that $k[\G]$ can be identified with a $k$-Hopf subalgebra of $k\{G\}$ because $\f(G)$ is Zarsiki dense in $\G$.
	By Corollary \ref{cor: algebraic group is absolutely sreduced} the $\s$-algebraic group $[\s]_k\G$ is absolutely $\s$-reduced. It therefore follows from Corollary \ref{cor: sFrobenius quotient for absolutely sreduced} that $$[\s]_k\G/([\s]_k\G)_{(r)}={\hsr ([\s]_k\G)}=[\s]_k({\hsr \G}).$$
	The image of $k[{\hsr \G}]$ under the dual map of $$\widetilde{\f}\colon G\xrightarrow{\f} [\s]_k\G\to[\s]_k\G/([\s]_k\G)_{(r)}=[\s]_k({\hsr \G})$$ equals $k[\s^r(k[\G])]$.
	Let $\G'\leq{\hsr \G}$ denote the Zariski closure of $\widetilde{\f}(G)$ in ${\hsr \G}$. Then $k[\G']=k[\s^r(k[\G])]$ and $\widetilde{\f}$ induces a morphism $\f'\colon G\to[\s]_k\G'$ of $\s$-algebraic groups.
	
	We will show that $\f'$ is a substandard embedding. By construction, $\f'(G)$ is Zariski dense in $\G'$. Because $\ker(\f)$ and $([\s]\G)_{(r)}$ are $\s$-infinitesimal (Lemma \ref{lemma: Gr sinfinitesimal}), it follows from Lemma \ref{lemma: inverse image of sinfinitesimal is sinfintesimal} that $\ker(\f')=\ker(\widetilde{\f})=\f^{-1}([\s]_k\G)_{(r)})$ is $\s$-infinitesimal.
	
	Let $\f(G)[1]\leq \G\times {\hs\G}$ denote the first order Zariski closure of $\f(G)$ in $\G$ and let $\G_1$ denote the kernel of $\pi_1\colon\f(G)[1]\to\G$. Then $|\G_1|=\ld(G)$ because $\f$ is a substandard embedding. Similarly, let $\f'(G)[1]\leq \G'\times {\hs(\G')}\leq {\hsr\G}\times {^{\sigma^{r+1}}\!\G}$ denote the first order Zariski closure of $\f'(G)$ in $\G'$ and let $\G'_1$ denote the kernel of $\pi_1\colon\f'(G)[1]\to\G'$.
	The surjective morphism of $k$-algebras
	$${\hsr (k[\f(G)[1]])}=k[k[\G],\s(k[\G])]\otimes_k k\to k[\s^r(k[\G]),\s^{r+1}(k[\G])],\ f\otimes\lambda\mapsto\s^r(f)\lambda $$ corresponds to a closed embedding $\f'(G)[1]\to {\hsr(\f(G)[1])}$ of algebraic groups, that maps $\G_1'$ into ${\hsr \G_1}$. Therefore $|\G'_1|\leq |{\hsr \G_1}|=|\G_1|=\ld(G)$.
	On the other hand, $|\G'_1|\geq\ld(\f'(G))$ by Proposition \ref{prop: limit degree} and
	$$\ld(\f'(G))=\ld(G/\ker(\f'))=\frac{\ld(G)}{\ld(\ker(\f'))}=\ld(G),$$
	because $\ker(\f')$ is $\s$-infinitesimal and therefore $\ld(\ker(\f'))=1$ by Corollary \ref{cor: sinfinitesimal has ld 1}.
	It follows that $|\G'_1|\geq \ld(G)$ and thus $|\G'_1|=\ld(G)$.
	
	So $\f'$ is a substandard embedding. Because $\f$ is a minimal substandard embedding, it follows that $|\G'|\geq |\G|$. But $\G'\leq {\hsr \G}$ and so $\G'={\hsr\G}$. Therefore $k[\s^r(k[\G])]=k[\G']={\hsr(k[\G])}$ has dimension $|\G|$ over $k$.	
\end{proof}

Finally, we are prepared for the proof of our decomposition theorem for $\s$-\'{e}tale $\s$-algebraic groups.

\begin{theo} \label{theo: Babbitt}
	Let $G$ be a $\s$-\'{e}tale $\s$-algebraic group. Then there exists a subnormal series
	\begin{equation} \label{eq: subnormal series G}
		G\supseteq G_1\supseteq G_2\supseteq \cdots\supseteq G_n\supseteq 1
	\end{equation}
	such that $G_1=G^\sc$, $G_n$ is $\s$-infinitesimal and $G_i/G_{i+1}\simeq [\s]_k\G_i$ for some $\s$-stably simple \'{e}tale algebraic group $\G_i$ for $i=1,\ldots,n-1$. 
	
	If 
	\begin{equation} \label{eq: subnormal series H}
	G\supseteq H_1\supseteq H_2\supseteq \cdots\supseteq H_m\supseteq 1
	\end{equation}	
	is another subnormal series such that $H_1=G^\sc$, $H_m$ is $\s$-infinitesimal and $H_i/H_{i+1}\simeq [\s]_k\H_i$ for some $\s$-stably simple \'{e}tale algebraic group $\H_i$ for $i=1,\ldots,m-1$, then $m=n$ and there exists a permutation $\tau$ such that $\G_i$ and $\H_{\tau(i)}$ are $\s$-stably equivalent.
\end{theo}
\begin{proof}
	We first handle the existence part.	We will prove the existence of the required subnormal series by induction on $\ld(G)$. If $\ld(G)=1$, then $k\{G\}$ is finitely generated as a $k$-algebra by Lemma \ref{lemma: ld=1}. Because $G$ is $\s$-\'{e}tale, this implies that $G$ is finite.
	Thus $G^\sc$ is a finite $\s$-connected $\s$-\'{e}tale $\s$-algebraic group and therefore $\s$-infinitesimal by Lemma \ref{lemma: Gsc sinfinitesimal}. So the subnormal series $G\supseteq G^\sc\supseteq 1$ has the required properties.
	
	We may thus assume that $\ld(G)>1$. Replacing $G$ with $G^\sc$, we may also assume that $G$ is $\s$-connected (Corollary \ref{cor: Gsc is sconnected}).
	
	\medskip
	
	{\bf Step 1:} Assume that there exists no normal $\s$-closed subgroup $N$ of $G$ such that $1<\ld(N)<\ld(G)$. We will show that there exists a normal $\s$-infinitesimal $\s$-closed subgroup $G_2$ of $G$ such that $G/G_2\simeq[\s]_k\G$ for a $\s$-stably simple \'{e}tale algebraic group $\G$. (Thus the existence part is satisfied with $n=2$ and $G_1=G$.)
	
	Let $\f\colon G\to\G$ be a minimal substandard embedding. We will eventually show that
	\begin{equation} \label{eqn: Babbitt step 1}
	\ld(G)=|\G|.
	\end{equation}

	\medskip
	
	Let $\f(G)[1]\leq \G\times{\hs\G}$ denote the first order Zariski closure of $\f(G)$ in $\G$. We have morphisms of algebraic groups
	$$\G\times{\hs\G}\to \G,\ (g_0,g_1)\mapsto g_0 \quad
	\text{ and } \quad \G\times{\hs\G}\to {\hs\G},\ (g_0,g_1)\mapsto g_1$$ that induce morphisms
	$\pi_1\colon \f(G)[1]\to\G(=\f(G)[0])$ and $\s_1\colon \f(G)[1]\to {\hs\G}$. Set $\G_1=\ker(\pi_1)$ and $\G'=\ker(\s_1)$. Then $\G'\G_1$ is a normal closed subgroup of $\f(G)[1]$ and we define
	$\H=\f(G)[1]/\G'\G_1.$ We will next show that
	\begin{equation} \label{eqn: formula in Babbit step 1}
	|\G|=|\H|\cdot \ld(G). 	
	\end{equation}
	Because $\G'\G_1\simeq \G'\times\G_1$, it follows from $|\H|=\frac{|\f(G)[1]|}{|\G'\G_1|}$ that
	\begin{equation} \label{eqn: Babbitt 3}
	|\H|\cdot |\G_1|=\frac{|\f(G)[1]|}{|\G'|}=|\f(G)[1]/\G'|.
	\end{equation} As $\f$ is a substandard embedding, $|\G_1|=\ld(G)$. Since
	$k[\f(G)[1]/\G']=k[\s(k[\G])]\subseteq k\{G\}$, it follows from Lemma \ref{lemma: for Babbitt step1 sr} that $|\f(G)[1]/\G'|=|\G|$. Thus (\ref{eqn: formula in Babbit step 1}) follows from (\ref{eqn: Babbitt 3}).
	
	\medskip
	
	We have $k[\f(G)[1]/\G_1]=k[\G]$ and $k[\f(G)[1]/\G']=k[\s(k[\G])]$. Therefore $$k[\H]=k[\G]\cap k[\s(k[\G])]\subseteq k\{G\}.$$
	Since $k[\H]$ is a Hopf subalgebra of $k\{G\}$, we have an induced morphism $k\{\H\}\to k\{G\}$ of \ks-Hopf algebras (Lemma \ref{lemma: Hopf induced}) that corresponds to a morphism $\f'\colon G\to [\s]_k\H$ of $\s$-algebraic groups. Let $N\unlhd G$ denote the kernel of $\f'$. By our assumption on $G$, we have $\ld(N)=1$ or $\ld(N)=\ld(G)$. So we have to distinguish these two cases.
	
	\medskip
	
	{\bf Case 1:} Let us first suppose that $\ld(N)=1$. Then $\ld(G/N)=\ld(G)$. We will show that $\f'$ is a substandard embedding. Because $k[\H]\subseteq k\{G\}$, we see that $\f'(G)$ is Zariski dense in $\H$. Let $\f'(G)[1]\leq \H\times{\hs \H}$ denote the first order Zariski closure of $\f'(G)$ in $\H$ and let $\H_1$ denote the kernel of $\pi_1\colon \f'(G)[1]\to\H$.
	Using Proposition \ref{prop: limit degree} we see that
	\begin{equation} \label{eqn: Babbitt step1 inequality1}
	|\f'(G)[1]|=|\H|\cdot|\H_1|\geq |\H|\cdot \ld(\f'(G))=|\H|\cdot \ld(G/N)=|\H|\cdot \ld(G).
	\end{equation}
	We will show that $\f'(G)[1]=\f(G)[1]/\G'$. Because $$k[\f'(G)[1]]=k[k[\H],\s(k[\H])]\subseteq k[\s(k[\G])]=k[\f(G)[1]/\G'],$$ it suffices to show that $|\f'(G)[1]|=|\f(G)[1]/\G'|$. 
	But using Lemma \ref{lemma: for Babbitt step1 sr}, we find that
	\begin{equation} \label{eqn: Babbitt step 1 inequality2}
	|\f'(G)[1]|\leq |\f(G)[1]/\G'|=\dim_k(k[\s(k[\G])])=|\G|=|\H|\cdot\ld(G).
	\end{equation}
	The combination of equations (\ref{eqn: Babbitt step1 inequality1}) and (\ref{eqn: Babbitt step 1 inequality2}) shows that $|\f'(G)[1]|=|\H|\cdot\ld(G)$ and that $|\H_1|=\ld(\f'(G))$. It also follows that $k[k[\H],\s(k[\H])]=k[\s(k[\G])]$ and therefore $$k\{G/N\}=k\{k[\H]\}=k\{\s(k[\G])\}.$$
	We will next show that $N=\ker(\f')$ is $\s$-infinitesimal. As $k\{k[\G]\}=k\{\f(G)\}\subseteq k\{G\}$, it follows that
	$$k\{\s(k[\G])\}=k[\s(k\{\f(G)\})]=k\{\f(G)/\f(G)_{(1)}\}\subseteq k\{G\}.$$
	So $k\{G/N\}=k\{\f(G)/\f(G)_{(1)}\}$ and therefore $N$ is the kernel of $$G\xrightarrow{\f}\f(G)\to\f(G)/\f(G)_{(1)}.$$ Thus $N$ is $\s$-infinitesimal by Lemmas \ref{lemma: Gr sinfinitesimal} and \ref{lemma: inverse image of sinfinitesimal is sinfintesimal}. Therefore
	$$|\H_1|=\ld(\f'(G))=\ld(G/N)=\frac{\ld(G)}{\ld(N)}=\ld(G)$$
	by Corollary \ref{cor: sinfinitesimal has ld 1}. In summary, we have shown that $\f'$ is a substandard embedding. But $|\H|<|\H|\cdot\ld(G)=|\G|$ because $\ld(G)>1$. This contradicts the assumption that $\f$ is a minimal substandard embedding. Thus the case $\ld(N)=1$ cannot occur. 
	
	\medskip
	
	{\bf Case 2:} So we must have $\ld(N)=\ld(G)$. Then $\ld(G/N)=1$ and therefore $G/N$ is finite (Lemma \ref{lemma: ld=1}). Because $G$ is $\s$-connected, also $G/N$ is $\s$-connected (Proposition~\ref{prop: sconnected exact sequence}). So $G/N$ is a finite, $\s$-connected, $\s$-\'{e}tale difference algebraic group and must therefore be $\s$-infinitesimal by Lemma \ref{lemma: Gsc sinfinitesimal}. So by Lemma \ref{lemma: sinfinitesimal} there exists an $r\in\nn$ such that $\s^r(\m_{G/N})=0$. Since $k\{G/N\}=k\{k[\H]\}$, we have $\m_\H\subseteq \m_{G/N}$, in particular, $\s^r(\m_\H)=0$.
	
	On the other hand, $\m_\H\subseteq k[\H]\subseteq k[\G]$ and by Lemma \ref{lemma: for Babbitt step1 sr} the dimension of $k[\s^r(k[\G])]$ as a $k$-vector space equals $|\G|$. So no non-zero element of $k[\G]$ can map to zero under $\s^r$. Thus $\m_\H=0$. This means that $\H=1$. Thus $|\G|=\ld(G)$ by (\ref{eqn: formula in Babbit step 1}) and finally (\ref{eqn: Babbitt step 1}) is proved.

	\medskip
	
	The $\s$-algebraic group $G_2=\ker(\f)$ is $\s$-infinitesimal because $\f$ is a substandard embedding and Lemma \ref{lemma: for Babbitt 1step} applied to the induced embedding $G/G_2\to[\s]_k\G$ shows that $G/G_2\simeq [\s]_k\G$ is benign, because
	$$\ld(G/G_2)=\frac{\ld(G)}{\ld(G_2)}=\ld(G)=|\G|$$
	by Corollary \ref{cor: sinfinitesimal has ld 1}.
	It remains to show that $\G$ is $\s$-stably simple. 
	
	As $\ld(N)\in \{1,\ld(G)\}$ for every normal $\s$-closed subgroup $N$ of $G$, we also have $\ld(N)\in \{1,\ld(G/G_2)\}$ for ever normal $\s$-closed subgroup $N$ of $G/G_2$ by Theorem \ref{theo: isom 3}.
	Thus $\G$ is $\s$-stably simple by Lemma \ref{lemma: implies sstably simple} and $G=G_1\supseteq G_2\supseteq 1$ is the required subnormal series.

	%

	\medskip
	
	{\bf Step 2:} Assume that there exists a normal $\s$-closed subgroup $N$ of $G$ such that $1<\ld(N)<\ld(G)$. Because $N^\sc$ is a characteristic subgroup of $N$ (Theorem \ref{theo: Gsc is characteristic}), it follows that $N^\sc$ also is a normal $\s$-closed subgroup of $G$. We have $\ld(N)=\ld(N^\sc)$ by Proposition~\ref{prop: invariants for Gsc}. Replacing $N$ by $N^\sc$, we may thus assume that $N$ is $\s$-connected.
	
	Because $\ld(G/N)=\ld(G)/\ld(N)<\ld(G)$ we can apply the induction hypothesis to $G/N$. As $G$ is $\s$-connected, also $G/N$ is $\s$-connected (Proposition \ref{prop: sconnected exact sequence}). So we obtain a subnormal series
	$$G/N\supseteq G_1/N\supseteq\cdots\supseteq G_n/N$$ for $G/N$,
	where $G\supseteq G_1\supseteq\cdots\supseteq G_n\supseteq N$ is a subnormal series for $G$ such that $G_n/N$ is $\s$-infinitesimal and $(G_i/N)/(G_{i+1}/N)=G_i/G_{i+1}$ is $\s$-stably simple benign for $i=0,\ldots,n-1$, where $G_0:=G$ (Theorem \ref{theo: isom 3}).
	
	As $\ld(N)<\ld(G)$, we can also apply the induction hypothesis to $N$. Since $N$ is $\s$-connected, we obtain a subnormal series
	$$N\supseteq N_1\supseteq\cdots\supseteq N_m,$$
	with $N_m$ $\s$-infinitesimal and $N_i/N_{i+1}$ $\s$-stably simple benign for $i=0,\ldots,m-1$ ($N_0:=N$). By Proposition~\ref{prop: exchange sinfinitesimal for Babbitt}, the subnormal series
	$$G_n\supseteq N\supseteq N_1\supseteq\cdots\supseteq N_m$$
	can be replaced by a subnormal series
	$$G_n\supseteq H_1\supseteq\cdots\supseteq H_m,$$
	with $G_n/H_1$ and $H_i/H_{i+1}$ $\s$-stably simple benign for $i=1,\ldots,m-1$ and $H_m$ $\s$-infinitesimal. Then
	$$G\supseteq G_1\supseteq\cdots\supseteq G_n\supseteq H_1\supseteq\cdots\supseteq H_m$$ is a subnormal series for $G$ of the required form.
	
	\medskip
	
	We next address the uniqueness part. We may assume that $G$ is $\s$-connected, i.e., $G=G_1=H_1$. By Theorem \ref{theo: Schreier refinement} the subnormal series (\ref{eq: subnormal series G}) and (\ref{eq: subnormal series H}) have equivalent refinements. Let 
	\begin{equation} \label{eq: long subnormal series}
G=G_1\supseteq G_{1,1}\supseteq G_{1,2}\supseteq\ldots\supseteq G_{1,r_1}\supseteq G_2\supseteq\ldots\ldots\supseteq G_n\supseteq G_{n,1}\supseteq\ldots\supseteq G_{n,r_n}\supseteq 1 
\end{equation}
	be such a refinement (where all the inclusion are strict).
	Then, for $i=1,\ldots,n-1$, the $\s$-algebraic group $G_{i,1}/G_{i+1}$ is a normal proper $\s$-closed subgroup of $G_i/G_{i+1}\simeq[\s]_k\G_i$. By Corollary \ref{cor: normal ssubgroup of simple has ld 1} we have $\ld(G_{i,1}/G_{i+1})=1$. Thus $\ld(G_{i,j}/G_{i+1})=1$ for $j=1,\ldots,r_i$ and so $\ld(G_{i,1})=\ld(G_{i,2})=\ldots=\ld(G_{i,r_i})=\ld(G_{i+1})$. Therefore, of the factor groups of the subnormal series (\ref{eq: long subnormal series}), there are exactly $n-1$ with limit degree $>1$, namely $G_i/G_{i,1}$ for $i=1,\ldots,n-1$.
	
	The same argument applied to (\ref{eq: subnormal series H}), instead of (\ref{eq: subnormal series G}), shows that in the refinement of (\ref{eq: subnormal series H}), there are exactly $m-1$ factor groups with limit degree $>1$. Since the two refinements are equivalent, we deduce that $n=m$.
	
	Moreover, $G_i/G_{i,1}\simeq [\s]_k\G_i'$, where $\G_i'$ is $\s$-stably equivalent to $\G_i$ by Proposition \ref{prop: key to uniqueness}. A similar statement holds for the refinement of the subnormal series (\ref{eq: subnormal series H}). The equivalence of the two refinements combined with Proposition \ref{prop: sisom implies isom} yields the sought for permutation $\tau$.
\end{proof}

We conclude the article with some examples illustrating Theorem \ref{theo: Babbitt}.
The following simple example shows that in the conclusion of Theorem \ref{theo: Babbitt} one cannot replace ``$\s$-stably equivalent'' with ``isomorphic''.

\begin{ex}
	Let $\G$ be a $\s$-stably simple \'{e}tale algebraic group and $G=[\s]_k\G$. For every $r\geq 1$ the normal $\s$-closed subgroup $G_{(r)}$ of $G$ is $\s$-infinitesimal and $G/G_{(r)}\simeq {\hsr G}=[\s]_k{\hsr\G}$ (Lemmas \ref{lemma: Gr sinfinitesimal}, \ref{lemma: absolutely sreduced and sFrobenius} and Corollary \ref{cor: algebraic group is absolutely sreduced}). Thus, for every $r$ the subnormal series
	$G\supseteq G_{(r)}\supseteq 1$ is as required by Theorem \ref{theo: Babbitt}.
\end{ex}

\begin{ex}
	Let $\G$ be an \'{e}tale algebraic group. We would like to find a subnormal series for $G=[\s]_k\G$ as in Theorem \ref{theo: Babbitt}. There exists a subnormal series $\G=\G_0\supseteq \G_1\supseteq\ldots\supseteq \G_m=1$ for $\G$ such that $\G_i/\G_{i+1}$ is simple for $i=0,\ldots,m-1$. However, $\G_i/\G_{i+1}$ may not be $\s$-stably simple. Since the length of such a decomposition series for ${\hsr\G}$ is bounded by $|\G|$, there exists an $r\in\nn$ and a subnormal series  ${\hsr\G}=\G_0\supseteq \G_1\supseteq\ldots\supseteq \G_{n-1}=1$ for ${\hsr\G}$ such that $\G_i/\G_{i+1}$ is $\s$-stably simple for $i=0,\ldots,n-2$. Set $G_i=(F^r_G)^{-1}([\s]_k\G_{i-1})$ for $i=2,\ldots,n$. We claim that
	$$G=G_1\supseteq G_2\supseteq\ldots\supseteq G_{n}\supseteq 1$$ is a subnormal series as in Theorem \ref{theo: Babbitt}.
	
	First of all, note that $G$ is $\s$-connected by Example \ref{ex: algebraic group sconnected}. So $G=G_1$ is justified. As $F^r_G\colon G\to {\hsr G}$ is a quotient map (Lemmas \ref{cor: algebraic group is absolutely sreduced}, \ref{lemma: absolutely sreduced and sFrobenius}), we see, using Theorem \ref{theo: isom 3}, that $G_i/G_{i+1}\simeq [\s]_k\G_{i-1}/[\s]_k\G_i=[\s]_k(\G_{i-1}/\G_i)$ is $\s$-stably simple benign for $i=1,\ldots,n-1$. Moreover, $G_n=G_{(r)}$ is $\s$-infinitesimal by Lemma \ref{lemma: Gr sinfinitesimal}.
\end{ex}

\begin{ex}
	Let $G$ be the $\s$-algebraic group given by
	$$G(R)=\{g\in R^\times|\ g^4=1,\ \s(g)^2=1\}$$
	for any \ks-algebra $R$. We assume that the characteristic of $k$ is not equal to $2$, so that $G$ is $\s$\=/\'{e}tale. We already noted in Example \ref{ex: easy Babbitt is sconnected} that $G$ is $\s$-connected. Let $G_2$ be the $\s$-closed subgroup of $G$ given by $G_2(R)=\{g\in R^\times|\ g^4=1,\ \s(g)=1\}$ for any \ks-algebra $R$. Then $\G$ is $\s$-infinitesimal and the quotient $G/G_2$ is benign. Indeed, $G/G_2=[\s]_k\G$, where $\G$ is the algebraic group given by $\G(T)=\{g\in T^\times|\ g^2=1\}$ for any $k$-algebra $T$. (Cf. Example \ref{ex: easy Babbitt is sconnected}.) Since $\G$ is $\s$-stably simple, we see that
	$$ G=G^{\sc}=G_1\supseteq G_2\supseteq 1$$
	is a subnormal series as in Theorem \ref{theo: Babbitt}. 
\end{ex}

The following example is inspired by \cite[Example 4]{Kitchens:ExpansiveDynamicsOnZeroDimensionalGroups}.

\begin{ex} \label{ex: Kitchens}
	Let $G$ be the $\s$-closed subgroup of $\Gm^2$ given by
	$$G(R)=\left\{(g,h)\in\Gm^2(R)|\ g^4=1,\ h^2=1,\ \s(h)=g^2h\right\}$$
	for any \ks-algebra $R$.
	Let us assume that the characteristic of $k$ is not equal to two, so that $G$ is $\s$-\'{e}tale. Indeed, $G$ is a Zariski dense $\s$-closed subgroup of the \'{e}tale algebraic group $\G=\mu_4\times\mu_2$.
	
	Define a $\s$-closed subgroup $G_2$ of $G$ by
	$$G_2(R)=\{(g,h)\in G(R)|\ h=1\}=\{(g,1)\in\Gm^2(R)|\ g^2=1\}.$$ So $G_2\simeq[\s]_k\mu_2$ is $\s$-stably simple benign. 
	The morphism $\f\colon G\to [\s]_k\Gm,\ (g,h)\to h$ has kernel $G_2$ and $\f(G)=[\s]_k\mu_2$. Thus $G/G_2\simeq[\s]_k\mu_2$ is also $\s$-stably simple benign. So there exists a short exact sequence
	\begin{equation} \label{eq: exact sequence}
	1 \to [\s]_k\mu_2\to G\to [\s]_k\mu_2\to 1.
	\end{equation}
	By Proposition \ref{prop: sconnected exact sequence} and Example \ref{ex: algebraic group sconnected} the $\s$-algebraic group $G$ is $\s$-connected. So $$G=G_1\supseteq G_2\supseteq G_3=1$$ is a subnormal series as in Theorem \ref{theo: Babbitt}. 
\end{ex}

\begin{ex} \label{ex: babbitt long}
	Let $G$ be the $\s$-algebraic group given by
	$$G(R)=\big\{(g_1,g_2,g_3)\in (R^\times)^3|\ g_1^4=g_2^4=g_3^2=1,\ \s(g_1)=g_2^2,\ \s(g_3)=g_3\big\}\leq \Gm(R)^3$$
	for any \ks-algebra $R$. We assume that the characteristic of $k$ is not equal to $2$, so that $G$ is $\s$-\'{e}tale.
	Let $G_1,G_2$ and $G_3$ be the $\s$-closed subgroups of $G$ given by
	$$G_1(R)=\left\{(g_1,g_2,1)\in (R^\times)^3|\ g_1^4=g_2^4=1,\ \s(g_1)=g_2^2\right\}\leq G(R),$$
	$$G_2(R)=\{(g_1,g_2,1)\in (R^\times)^3|\ g_1^4=g_2^2=1,\ \s(g_1)=1\}\leq G(R)$$ and
	$$G_3(R)=\left\{(g_1,1,1)\in (R^\times)^3|\ g_1^4=1,\ \s(g_1)=1 \right\}\leq G(R)$$
	for any \ks-algebra $R$. We will show that
	$$G\supseteq G_1\supseteq G_2\supseteq G_3\supseteq 1$$
	is a subnormal series for $G$ as in Theorem \ref{theo: Babbitt}.
	
	The quotient map $\f\colon G_1\to [\s]_k\mu_4$ given by $$\f_R\colon G_1(R)\to \mu_4(R),\ (g_1,g_2,1)\mapsto g_2$$ has kernel $G_3$.
	So $G_1/G_3\simeq [\s]_k\mu_4$ is benign. As $\f(G_2)=[\s]_k\mu_2\leq[\s]_k\mu_4$, we see that $G_1/G_2\simeq [\s]_k\mu_2$ and $G_2/G_3\simeq [\s]_k\mu_2$ are $\s$-stably simple benign.

	 Clearly $G_3$ is $\s$-infinitesimal. So it remains to show that $G_1=G^\sc$.
	 We have an exact sequence
	 $$1\to G_3\to G_1\to [\s]_k\mu_4\to 1,$$
	 with $G_3$ and $[\s]_k\mu_4$ $\s$-connected (Corollary \ref{cor: sinfinitesimal has ld 1} and Example \ref{ex: algebraic group sconnected}). Therefore $G_1$ is $\s$-connected by Proposition \ref{prop: sconnected exact sequence}. To show that $G_1=G^\sc$, it suffices to show that $G/G_1$ is \ssetale{} (Proposition \ref{prop: sconnected component unique}). 
	 Let $H$ be the $\s$-closed subgroup of $\Gm$ given by 
	 $$H(R)=\{h\in R^\times|\ h^2=1,\ \s(h)=h\}$$
	 for any \ks-algebra $R$. The quotient map $G\to H,\ (g_1,g_2,g_3)\mapsto g_3$ has kernel $G_1$ and so $G/G_1\simeq H$ is \ssetale{} (Example \ref{ex: ssetale if and only if}).
\end{ex}

\bibliographystyle{alpha}
\bibliography{bibdata}

\end{document}